\documentclass{amsart}
\usepackage{amsmath,amssymb}
\title{Exact module categories over $\mathrm{Rep}(u_q(\mathfrak{sl}_2))$}
\date{}

\author[D. Nakamura]{Daisuke Nakamura}
\address[D. Nakamura]{Graduate School of Science and Engineering, Mathematical Science Course, Okayama University of Science \\
  1-1 Ridai-cho, Kita-ku Okayama-shi, Okayama 700-0005, Japan.}
\email{r23nmc9wy@ous.jp}
\author[T. Shibata]{Taiki Shibata}
\address[T. Shibata]{Department of Applied Mathematics,
  Okayama University of Science \\
  1-1 Ridai-cho, Kita-ku Okayama-shi, Okayama 700-0005, Japan.}
\email{shibata@ous.ac.jp}
\author[K. Shimizu]{Kenichi Shimizu}
\address[K. Shimizu]{Department of Mathematical Sciences,
  Shibaura Institute of Technology \\
  307 Fukasaku, Minuma-ku, Saitama-shi, Saitama 337-8570, Japan.}
\email{kshimizu@shibaura-it.ac.jp}



\usepackage{mathrsfs}
\usepackage{color,url}
\usepackage{tikz-cd}
\usepackage[setpagesize=false]{hyperref}
\usepackage{amsthm}
\numberwithin{equation}{section}
\newtheorem{C}{}[section] % the common counter
\newtheorem{lemma}[C]{Lemma}
\newtheorem{claim}[C]{Claim}
\newtheorem{theorem}[C]{Theorem}

\theoremstyle{definition}
\newtheorem{definition}[C]{Definition}

\newtheorem{assumption}[C]{Assumption}
\theoremstyle{remark}
\newtheorem{remark}[C]{Remark}

\newcommand{\id}{\mathrm{id}}
\newcommand{\op}{\mathrm{op}}
\newcommand{\cop}{\mathrm{cop}}
\newcommand{\rev}{\mathrm{rev}}
\newcommand{\eval}{\mathrm{ev}}
\newcommand{\coev}{\mathrm{coev}}
\newcommand{\bfk}{\Bbbk}
\newcommand{\Hom}{\mathrm{Hom}}
\newcommand{\End}{\mathrm{End}}
\newcommand{\iHom}{\underline{\Hom}}
\newcommand{\iEnd}{\underline{\End}}
\newcommand{\Ker}{\mathrm{Ker}}
\newcommand{\Img}{\mathrm{Im}}
\newcommand{\Jac}{\mathrm{Jac}}
\newcommand{\socle}{\mathrm{soc}}
\newcommand{\unitobj}{\mathbf{1}}
\newcommand{\catactl}{\triangleright}
\newcommand{\Rex}{\mathrm{Rex}}
\newcommand{\Rep}{\mathrm{Rep}}
\newcommand{\Vect}{\mathsf{Vec}}
\newcommand{\FiltVect}{\mathsf{FiltVect}}
\newcommand{\GrVect}{\mathsf{GrVect}}
\newcommand{\gr}{\mathop{\mathrm{gr}}\nolimits}
\newcommand{\grc}{\gr_{\mathrm{c}}} % gr w.r.t. coradical filtration
\newcommand{\grL}{\gr_{\mathrm{L}}} % gr w.r.t. Lowey filtration
\newcommand{\Grp}{\mathrm{G}}       % grouplike elements
\newcommand{\lmod}[1]{{#1}\textnormal{-mod}}
\newcommand{\rmod}[1]{\textnormal{mod-}{#1}}
\newcommand{\bimod}[2]{{#1}\textnormal{-mod-}{#2}}
\newcommand{\ullMod}[1]{{{#1}\text{-}\underline{\mathrm{Mod}}}}
\newcommand{\Mod}{\mathfrak{M}}
\newcommand{\EW}{\mathrm{EW}}
\newcommand{\HM}{\mathrm{HM}}
\begin{document}

\begin{abstract}
  We give a complete list of indecomposable exact module categories over the finite tensor category $\Rep(u_q(\mathfrak{sl}_2))$ of representations of the small quantum group $u_q(\mathfrak{sl}_2)$, where $q$ is a root of unity of odd order. Each of them is given as the category of representations of a left comodule algebra over $u_q(\mathfrak{sl}_2)$ explicitly presented by generators and relations.
\end{abstract}

\maketitle

\tableofcontents

\section{Introduction}

Finite tensor categories are widely used in several branches of mathematics and mathematical physics including representation theory, low-dimensional topology, and conformal field theory. As with many other algebraic structures, studying `representations' of a finite tensor category provides an effective means of developing the theory of finite tensor categories and plays an important role in their various applications.
A `representation' of a finite tensor category is mathematically formulated as a module category \cite{MR3242743}. Among them, exact module categories, introduced by Etingof and Ostrik, are especially important. As demonstrated in \cite{MR3242743}, several fundamental results for finite tensor categories were obtained by using the notion of exact module categories.

When $\mathcal{C}$ is a semisimple finite tensor category (that is, a fusion category \cite{MR3242743}), an exact module category is equivalent to a semisimple module category. The classification of indecomposable semisimple module categories was done for a variety of fusion categories; see, {\it e.g.}, \cite{MR1936496,MR1976233,MR2927179,MR3009745,MR2909758,MR3663587,MR4589524,2024arXiv240802794E}. While much progress has been made in the semisimple case, relatively little has been investigated in the non-semisimple case.
Etingof and Ostrik classified indecomposable exact module categories over $\Rep(H)$ for the case where $H$ is the Taft algebra or a finite supergroup algebra \cite{MR2119143}. Later, a general theory was established by Andruskiewitsch and Mombelli. Following their paper \cite{MR2331768}, for a finite-dimensional Hopf algebra $H$, an indecomposable exact $\Rep(H)$-module category is of the form $\Rep(A)$, where $A$ is a right $H$-simple left $H$-comodule algebra with trivial coinvariants (see Subsection \ref{subsec:term-comod-alg} for the terminology). Additionally, Mombelli \cite{MR2678630,MR2989520,MR3264686} and Mombelli and and Garc\'ia Iglesias \cite{MR2860429} developed techniques for the classification of such comodule algebras in the case where $H$ is pointed.

The Taft algebra, dealt with by Etingof and Ostrik in \cite{MR2119143}, is one of the simplest examples of pointed Hopf algebras. The small quantum group $u_q := u_q(\mathfrak{sl}_2)$ should belong to the second simplest class. The main purpose of this paper is to give an explicit list of indecomposable exact module categories over $\Rep(u_q)$ at a root of unity $q$ of odd order. For the associated graded Hopf algebra $\gr(u_q)$, Mombelli \cite{MR2678630} has classified right $\gr(u_q)$-simple left $\gr(u_q)$-comodule algebras with trivial coinvariants up to $\gr(u_q)$-equivariant Morita equivalence. Since $\Rep(u_q)$ and $\Rep(\gr(u_q))$ are categorically Morita equivalent in the sense of \cite[Section 7.12]{MR3242743}, one can, in principle, obtain a list of indecomposable exact module categories over $\Rep(u_q)$ from Mombelli's list.

So, what remains to be done? First, as is well-known, $u_q$ is obtained from $\gr(u_q)$ by 2-cocycle deformation. Fortunately, one can find an explicitly written 2-cocycle $\sigma$ turning $\gr(u_q)$ into $u_q$ in \cite{2010arXiv1010.4976G}. Now let $L$ be a right $\gr(u_q)$-simple left $\gr(u_q)$-comodule algebra with trivial coinvariants. We then obtain a left $u_q$-comodule algebra ${}_{\sigma}L$ by deforming the multiplication of $L$ by $\sigma$. As experts may know, $\Rep({}_{\sigma}L)$ is the indecomposable exact $\Rep(u_q)$-module category corresponding to the $\Rep(\gr(u_q))$-module category $\Rep(L)$ under the categorical Morita equivalence between $\Rep(u_q)$ and $\Rep(\gr(u_q))$.
The main contribution of this paper is an explicit description of ${}_{\sigma}L$ for $L$ in Mombelli's list.
As it turns out, the determination of ${}_{\sigma}L$ is not a trivial task in general. In fact, there is a family $\{ \mathscr{L}(\alpha, \beta; \xi) \}$ of left $\gr(u_q)$-comodule algebras parametrized by $\alpha, \beta, \xi \in \bfk$ with $(\alpha, \beta) \ne (0, 0)$. The comodule algebra $\mathscr{L}(\alpha, \beta; \xi)$ is generated by a single element $W$ subject to a simple relation $W^N = \xi$ (where $N$ is the order of $q$) and such that the coaction of $\gr(u_q)$ is given by
\begin{equation*}
  W \mapsto (\alpha K^{-1} E + \beta F) \otimes 1 + K^{-1} \otimes W,
\end{equation*}
where $E$, $F$ and $K$ are the standard generators of $u_q$ viewed as elements of $\gr(u_q)$. The algebra $\mathscr{A} := {}_{\sigma}(\mathscr{L}(\alpha, \beta; \xi))$ obtained by 2-cocycle deformation is still generated by $W$, and the coaction of $u_q$ on $\mathscr{A}$ is still given by the same formula on the generator $W$. However, the defining relation of $\mathscr{A}$ becomes
\begin{equation*}
  \sum_{k = 0}^{(N-1)/2} \binom{N - k}{k} \frac{N}{N - k}
  \left( \frac{q \alpha \beta}{(q^2 - 1)^2} \right)^{k} W^{N - 2 k} = \xi,
\end{equation*}
as we will prove using an idea from \cite{2024arXiv241010064S} and basic properties of the Dickson polynomials (see Lemma~\ref{lem:cocycle-deform-L4} and Remark~\ref{rem:A4-introduction}). Perhaps surprisingly, the cocycle deformation can also change semisimplicity. In fact, the algebra $\mathscr{L}(\alpha, \beta; \xi)$ is semisimple if and only if $\xi \ne 0$, whereas the algebra $\mathscr{A}$ can be non-semisimple even when $\xi \neq 0$ (see Remark \ref{rem:A4-semisimplicity}).

The main goal of this paper was described above. In addition, there is another objective closely related to it. After publication of the papers \cite{MR2331768,MR2678630,MR3264686,MR2860429} mentioned above, the theory of finite tensor categories has been updated substantially. A secondary purpose of this paper is to refine several results on Hopf algebras and comodule algebras given in these papers with the help of the latest understanding of the theory of finite tensor categories and their modules. Let $H$ be a finite-dimensional pointed Hopf algebra. Notably, we generalize a criterion for when two right $H$-simple left $H$-comodule algebras are $H$-equivariantly Morita equivalent \cite[Theorem 4.2]{MR2860429} from the viewpoint of exact module categories (see Theorems~\ref{thm:Morita-in-pointed-FTC} and~\ref{thm:H-Morita-pointed}). Mombelli's criterion of right $H$-simplicity of filtered $H$-comodule algebras \cite[Proposition 4.4 and Corollary 4.5]{MR2678630} is also extended to a broader setting (Theorem~\ref{thm:H-simplicity-comod-alg}). Based on these results, in \S\ref{subsec:classification-strategy}, we propose a classification strategy of right $H$-simple $H$-comodule algebras, slightly modifying that proposed in \cite{MR2678630,MR2860429}. In this paper, we actually classify indecomposable exact $\Rep(u_q)$-module categories from scratch using our strategy. Mombelli's list of $\gr(u_q)$-comodule algebras (with minor modifications) will also be obtained as an intermediate step of our computation (see Remark~\ref{rem:Mombelli-list-gr-uq-comod-alg}).

\subsection{Organization of this paper}

In Section \ref{sec:FTC-and-modules}, we review basic results on finite tensor categories and finite module categories, and add some new results on the Morita theory for algebras in a finite tensor category. This section specifically focuses on simple algebras and division algebras in a finite tensor category. As in ordinary ring theory, a simple algebra in a finite tensor category $\mathcal{C}$ is defined to be an algebra $A$ in $\mathcal{C}$ that is simple as an $A$-bimodule in $\mathcal{C}$ (Definition \ref{def:simple-alg-in-FTC}). Generalizing a result of Kesten and Walton \cite{2025arXiv250202532K}, we show that an algebra $A$ in $\mathcal{C}$ is simple as a left $A$-module in $\mathcal{C}$ if and only if it is simple as a right $A$-module in $\mathcal{C}$ (Lemma~\ref{lem:exact-alg-left-and-right-simple}). Such an algebra is called a division algebra in $\mathcal{C}$ (Definition~\ref{def:division-alg-in-C}). We provide some Morita-theoretic results on division algebras in $\mathcal{C}$ (Lemmas \ref{lem:exact-algebra-indec-Morita-equiv-to-right-simple}, \ref{lem:division-alg-Morita-equiv} and \ref{lem:division-alg-is-haploid}). One of the most important results in this section is Theorem~\ref{thm:Morita-in-pointed-FTC}, which states that two division algebras $A$ and $B$ in a pointed finite tensor category $\mathcal{C}$ are Morita equivalent if and only if there exists an invertible object $g \in \mathcal{C}$ and an isomorphism $B \cong g \otimes A \otimes g^*$ of algebras in $\mathcal{C}$.

In Section~\ref{sec:exact-mod-and-comod-alg}, we review pioneering work on comodule algebras relevant to our study, notably from \cite{MR2331768,MR2286047,MR2860429}, and offer a new perspective on some of these results. After introducing basic terminology concerning comodule algebras, we deploy the Morita duality method: Given a finite tensor category $\mathcal{C}$, the 2-category $\ullMod{\mathcal{C}}$ of finite left $\mathcal{C}$-module categories is introduced in Section~\ref{sec:FTC-and-modules}. We now fix a finite-dimensional Hopf algebra $H$, set $\mathcal{C} = \Rep(H)$, and define $\mathcal{D}$ to be the category of finite-dimensional left $H$-comodules. We show that there is an equivalence
\begin{equation*}
  \ullMod{\mathcal{C}} \to (\ullMod{\mathcal{D}})^{\op},
  \quad \mathcal{M} \mapsto \mathcal{M}^* := \Rex_{\mathcal{C}}(\mathcal{M}, \Vect)
\end{equation*}
of 2-categories, which we call the Morita duality between $\mathcal{C}$ and $\mathcal{D}$ (see Subsection~\ref{subsec:morita-duality} for the detail). By the Morita duality, we show that two algebras $A$ and $B$ in $\mathcal{D}$ are Morita equivalent in $\mathcal{D}$ if and only if $\Rep(A)$ and $\Rep(B)$ are equivalent as left $\mathcal{C}$-module categories (Theorem~\ref{thm:H-Morita}). By translating Theorem~\ref{thm:Morita-in-pointed-FTC} via the Morita duality, we obtain \cite[Theorem 4.2]{MR2860429} in a generalized form. With emphasis on the use of the Morita duality, we also give a new proof of \cite[Theorem 3.3]{MR2331768}, which states that an indecomposable exact $\Rep(H)$-module category is of the form $\Rep(A)$, where $A$ is a right $H$-simple left $H$-comodule algebra with trivial coinvariants (see Theorem~\ref{thm:AM-exact-mod-cat}). We also derive properties of 2-cocycle deformation of comodule algebras by using the Morita duality (Subsection~\ref{subsec:cocycle-deform}).

In Section~\ref{sec:exact-comod-alg-over-pointed}, we discuss the case where $H$ is a finite-dimensional pointed Hopf algebra. Let $A$ be a finite-dimensional left $H$-comodule algebra. Then $A$ has a canonical filtration $\{ A_n \}_{n \ge 0}$ compatible with the coradical filtration $\{ H_n \}_{n \ge 0}$ of $H$. Moreover, the associated graded algebra $\gr(A)$ has a natural structure of a left $\gr(H)$-comodule algebra over $\gr(H)$. Under the assumption that $H_0$ is semisimple and the base field is an algebraically closed field of characteristic zero, Mombelli proved that $A$ is right $H$-simple if and only if $\gr(A)$ is right $\gr(H)$-simple, if and only if $A_0$ is right $H_0$-simple \cite[Proposition 4.4 and Corollary 4.5]{MR2678630}. The main result of this section is that his result still holds without the assumption on the base field (Theorem~\ref{thm:H-simplicity-comod-alg}). Now let $U$ be a finite-dimensional pointed Hopf algebra with group of grouplike elements $\Gamma$.
As explained in Subsection~\ref{subsec:classification-strategy}, under Assumption~\ref{assumptions-for-classification-strategy}, right $U$-simple left $U$-comodule algebras are classified up to $U$-Morita equivalence in the following way:
\begin{description}
  \setlength{\itemsep}{5pt}
\item[Step 1] Classify all graded left coideal subalgebras of $H$.
\item[Step 2] For each graded coideal subalgebra $L$ of $H$ obtained in Step 1, find all left $H$-comodule algebras $\mathcal{L}$ such that $\gr(\mathcal{L}) \cong L$ as graded $\gr(H)$-comodule algebras.
\item[Step 3] Determine whether left $H$-comodule algebras obtained in Step 2 are $H$-Morita equivalent by, for example, using Theorem~\ref{thm:H-Morita-pointed}.
\item[Step 4] For each representative $\mathcal{L}$ of an $H$-Morita equivalence class of left $H$-comodule algebras obtained in Step 2, compute the 2-cocycle deformation ${}_{\sigma}\mathcal{L}$. The resulting algebras are representatives of right $U$-simple left $U$-comodule algebras under the $U$-Morita equivalence.
\end{description}

In Section~\ref{sec:exact-comod-alg-uqsl2}, we apply the above classification strategy to the small quantum group $u_q := u_q(\mathfrak{sl}_2)$ at a root of unity $q$ of odd order. Step 1 has been done in \cite{2024arXiv241010064S}. Steps 2 and 3 are addressed by \cite[Lemma 5.5]{MR2678630} and Theorem~\ref{thm:H-Morita-pointed}, respectively. The most complicated part is Step 4 for $\mathcal{L}$ the left $\gr(u_q)$-comodule algebra that we denote by $\mathscr{L}_4(\alpha, \beta; \xi)$. We use arguments on the minimal polynomial of certain elements of $u_q$ given in \cite{2024arXiv241010064S} along with basic properties of Dickson polynomials to complete this part. Our classification result of indecomposable exact module categories over $\Rep(u_q(\mathfrak{sl}_2))$ is given as Theorem~\ref{thm:conclusion}.

\subsection{Acknowledgements}

The authors would like to express their gratitude to Iv\'an Angiono, Cris Negron, and Christoph Schweigert for fruitful discussions on our project during the workshop `Tensor Categories, Quantum Symmetries, and Mathematical Physics' held at MATRIX in November 2024.
The third author (K.\ S.) would also like to thank Mart\'in Mombelli and Chelsea Walton for valuable comments in private communications.
% KAKENHI
The second author (T.\ S.) is supported by JSPS KAKENHI Grant Numbers JP22K13905 and JP26K06774.
K.\ S. is supported by JSPS KAKENHI Grant Number JP24K06676.

\subsection{Declaration of AI use}

ChatGPT (GPT-6) was used solely for proofreading and checking the final manuscript, including correcting grammatical and typographical errors, improving linguistic clarity, and identifying minor errors and inconsistencies. The authors take full responsibility for the accuracy and integrity of the manuscript.

\section{Finite tensor categories and their modules}
\label{sec:FTC-and-modules}

\subsection{Notation and convention}

Throughout this paper, we work over a field $\bfk$.
We denote by $\Vect$ the category of finite-dimensional vector spaces over the field $\bfk$.
By an algebra, we mean an associative unital algebra over $\bfk$. Given algebras $A$ and $B$, we denote by $\lmod{A}$, $\rmod{B}$ and $\bimod{A}{B}$ the categories of finite-dimensional left $A$-modules, right $B$-modules and $A$-$B$-bimodules, respectively. In the Introduction, $\lmod{A}$ was written as $\Rep(A)$ to avoid the confusing phrase `$H$-mod-module category' for a Hopf algebra $H$. The symbol $\Rep(A)$ is rarely used in the main body of this paper.

Our main reference on monoidal categories and module categories is \cite{MR3242743}. All monoidal categories and module categories over them are assumed to be strict. Given an object $X$ of a rigid monoidal category, we denote its left dual object, the evaluation and the coevaluation (in the sense of \cite[Section 2.10]{MR3242743}) by $X^*$, $\eval_X : X^* \otimes X \to \unitobj$ and $\coev_X : \unitobj \to X \otimes X^*$, respectively. The right dual object of $X$ is denoted by ${}^* \! X$. Given algebras $A$ and $B$ in a monoidal category $\mathcal{C}$, we denote by ${}_A\mathcal{C}$, $\mathcal{C}_B$ and ${}_A\mathcal{C}_B$ the category of left $A$-modules, right $B$-modules and $A$-$B$-bimodules in $\mathcal{C}$, respectively.

\subsection{Finite tensor categories and their modules}
\label{subsec:finite-mod-cat}

A {\em finite abelian category} is a $\bfk$-linear category that is equivalent to $\lmod{A}$ for some finite-dimensional algebra $A$. Given finite abelian categories $\mathcal{M}$ and $\mathcal{N}$, we denote by $\Rex(\mathcal{M}, \mathcal{N})$ the category of $\bfk$-linear right exact functors from $\mathcal{M}$ to $\mathcal{N}$. For finite-dimensional algebras $A$ and $B$, there is the equivalence $\bimod{A}{B} \approx \Rex(\rmod{A}, \rmod{B})$ of linear categories, which we call the {\em Eilenberg-Watts equivalence}. By using this equivalence, one can show that a left (right) exact functor between finite abelian categories has a left (right) adjoint.

A {\em finite tensor category} is a finite abelian category equipped with a structure of a rigid monoidal category such that the tensor product is bilinear and the unit object is absolutely simple.
Let $\mathcal{C}$ be a finite tensor category.
A {\em finite left $\mathcal{C}$-module category} is a finite abelian category $\mathcal{M}$ equipped with a structure of a left $\mathcal{C}$-module category such that the action $\catactl : \mathcal{C} \times \mathcal{M} \to \mathcal{M}$ is bilinear and right exact in each variable. A {\em finite right $\mathcal{C}$-module category} is defined analogously.

Let $\mathcal{M}$ and $\mathcal{N}$ be left module categories over a monoidal category $\mathcal{C}$. An {\em oplax left $\mathcal{C}$-module functor} from $\mathcal{M}$ to $\mathcal{N}$ is a functor $F: \mathcal{M} \to \mathcal{N}$ equipped with a natural transformation $\xi_{X,M} : F(X \catactl M) \to X \catactl F(M)$ ($X \in \mathcal{C}, M \in \mathcal{M}$) satisfying certain axioms (see Appendix~\ref{appendix:equiv-EW}). A {\em strong left $\mathcal{C}$-module functor} is an oplax left $\mathcal{C}$-module functor with invertible structure morphism. When $\mathcal{C}$ is rigid, every oplax left $\mathcal{C}$-module functor is strong \cite[Lemma 2.10]{2014arXiv1406.4204D}\footnote{The preprint \cite{2014arXiv1406.4204D} has already been published as \cite{MR3934626}; however, we cite the preprint version since \cite[Lemma 2.10]{2014arXiv1406.4204D} has been removed in the published version.} and therefore we simply call it a {\em left $\mathcal{C}$-module functor}.

\begin{definition}
  Given a finite tensor category $\mathcal{C}$, we denote by $\ullMod{\mathcal{C}}$ the 2-category of finite left $\mathcal{C}$-module categories, $\bfk$-linear right exact left $\mathcal{C}$-module functors and their morphisms.
  The Hom-category from $\mathcal{M}$ to $\mathcal{N}$ in $\ullMod{\mathcal{C}}$ is written as $\Rex_{\mathcal{C}}(\mathcal{M}, \mathcal{N})$.
\end{definition}

\subsection{Morita theory in finite tensor categories}
\label{subsec:Morita-theory}

Let $\mathcal{C}$ be a finite tensor category, and let $A$ be an algebra in $\mathcal{C}$.
For $X \in \mathcal{C}$ and $M \in \mathcal{C}_A$ with action $a_M : M \otimes A \to M$, the object $X \otimes M$ becomes a right $A$-module in $\mathcal{C}$ by the action $\id_X \otimes a_M$.
This construction makes $\mathcal{C}_A$ a finite left $\mathcal{C}$-module category.
The category ${}_A\mathcal{C}$ is a finite right $\mathcal{C}$-module category in a similar way.

\begin{definition}
  \label{def:Morita-in-FTC}
  Two algebras $A$ and $B$ in $\mathcal{C}$ are said to be {\em Morita equivalent in $\mathcal{C}$} if $\mathcal{C}_A \approx \mathcal{C}_B$ in $\ullMod{\mathcal{C}}$.
\end{definition}

Let $\mathcal{M}$ be a finite left $\mathcal{C}$-module category.
For each $M \in \mathcal{M}$, the functor $\mathcal{C} \to \mathcal{M}$ defined by $X \mapsto X \catactl M$ is right exact and therefore it has a right adjoint. Hence there is a functor $\iHom_{\mathcal{M}} : \mathcal{M}^{\op} \times \mathcal{M} \to \mathcal{C}$ together with a natural isomorphism
\begin{equation}
  \label{eq:iHom-adjunction}
  \Hom_{\mathcal{C}}(X, \iHom_{\mathcal{M}}(M, N)) \cong \Hom_{\mathcal{M}}(X \catactl M, N)
\end{equation}
for $M, N \in \mathcal{M}$ and $X \in \mathcal{C}$. We call $\iHom_{\mathcal{M}}$ the {\em internal Hom functor} and write it as $\iHom$ when $\mathcal{M}$ is clear from the context. By using the natural isomorphism \eqref{eq:iHom-adjunction}, one can define morphisms
\begin{equation*}
  \iHom(Y,Z) \otimes \iHom(X,Y) \to \iHom(X,Z)
  \quad \text{and} \quad
  \unitobj \to \iHom(X, X)
\end{equation*}
in $\mathcal{C}$ for $X, Y, Z \in \mathcal{M}$ that are `associative' and `unital' in a certain sense. In particular, the object $\iEnd(X) := \iHom(X,X)$ is an algebra in $\mathcal{C}$, which we call the {\em internal endomorphism algebra}.

Generalizing some notions in the theory of rings and modules, we introduce important classes of objects of $\mathcal{M}$ in terms of the internal Hom functor:

\begin{definition}[\cite{2014arXiv1406.4204D}]
  \label{def:C-projective}
  We fix an object $M \in \mathcal{M}$.
  \begin{enumerate}
  \item $M$ is said to be {\em $\mathcal{C}$-projective} if $\iHom(M, -)$ is exact.
  \item $M$ is said to be {\em a $\mathcal{C}$-generator} if $\iHom(M, -)$ is faithful.
  \item $M$ is said to be {\em $\mathcal{C}$-injective} if $\iHom(-, M)$ is exact.
  \end{enumerate}
\end{definition}

Many conditions equivalent to these are exhibited in \cite{2014arXiv1406.4204D,2024arXiv240202929S}. Among others, we note that an object $M \in \mathcal{M}$ is $\mathcal{C}$-projective if and only if $P \catactl M$ is projective for all projective objects $P \in \mathcal{C}$ (see \cite[Lemma 5.3]{2024arXiv240202929S}). Also, $M$ is a $\mathcal{C}$-generator if and only if for every $N \in \mathcal{M}$ there exists an object $X \in \mathcal{C}$ such that $X \catactl M$ has $N$ as a quotient object \cite[Lemma 2.22]{2014arXiv1406.4204D}.

One can show that every finite left $\mathcal{C}$-module category is equivalent to $\mathcal{C}_A$ for some algebra $A$ in $\mathcal{C}$. More precisely, as an application of the Barr-Beck monadicity theorem, we obtain the following theorem for module categories:

\begin{lemma}[the Morita theorem {\cite[Theorem 2.24]{2014arXiv1406.4204D}}]
  Let $\mathcal{M}$ be a finite left $\mathcal{C}$-module category, and let $G \in \mathcal{M}$ be an object. Then the functor
  \begin{equation*}
    \iHom(G, -) : \mathcal{M} \to \mathcal{C}_{\iEnd(G)}
  \end{equation*}
  is an equivalence in $\ullMod{\mathcal{C}}$ if and only if $G$ is a $\mathcal{C}$-projective $\mathcal{C}$-generator.
\end{lemma}

In particular, if $G$ is a projective generator of $\mathcal{M}$, then the above functor is an equivalence in $\ullMod{\mathcal{C}}$. An important consequence is that every finite left $\mathcal{C}$-module category is equivalent to $\mathcal{C}_A$ for some algebra $A$ in $\mathcal{C}$ \cite{MR3934626}.

By the same argument as in the ordinary Morita theory for rings and modules, one can show that two algebras $A$ and $B$ in $\mathcal{C}$ are Morita equivalent if and only if there is a $\mathcal{C}$-projective $\mathcal{C}$-generator $P \in \mathcal{C}_A$ such that $\iEnd(P) \cong B$ as algebras in $\mathcal{C}$.

A ring $R$ and a matrix ring over $R$ are typical examples of Morita equivalent rings.
Lemma~\ref{lem:Morita-eq-matrix-alg} below gives analogous examples in finite tensor categories.
Let $A$ be an algebra in $\mathcal{C}$ with multiplication $\mu$ and unit $\eta$.
For a non-zero object $X \in \mathcal{C}$, the object $X \otimes A \otimes X^*$ is an algebra by the multiplication
\begin{equation*}
  (\id_X \otimes \mu \otimes \id_{X^*}) \circ (\id_X \otimes \id_A \otimes \eval_X \otimes \id_A \otimes \id_{X^*})
\end{equation*}
and the unit $(\id_X \otimes \eta \otimes \id_{X^*}) \circ \coev_X$. When $\mathcal{C} = \Vect$, the algebra $X \otimes A \otimes X^*$ is isomorphic to the matrix algebra over $A$ of degree $\dim_{\bfk}(X)$.

Now let $\mathcal{M}$ be a finite left $\mathcal{C}$-module category, and let $X \in \mathcal{C}$ and $M \in \mathcal{M}$ be objects. By using well-known natural isomorphisms for the internal Hom functor (see, {\it e.g.}, \cite[Section 7.9]{MR3242743}), we obtain an isomorphism
\begin{equation}
  \label{eq:internal-End-X-M}
  \iEnd(X \catactl M) \cong X \otimes \iEnd(M) \otimes X^*
\end{equation}
of algebras in $\mathcal{C}$. We use this to prove:

\begin{lemma}
  \label{lem:Morita-eq-matrix-alg}
  Let $A$ be an algebra in $\mathcal{C}$, and let $X$ be a non-zero object of $\mathcal{C}$. Then the algebras $A$ and $X \otimes A \otimes X^*$ are Morita equivalent.
\end{lemma}
\begin{proof}
  Let $\iHom_A$ and $\iEnd_A$ denote the internal Hom and End for $\mathcal{C}_A$, respectively. There are natural isomorphisms
  \begin{equation*}
    \iHom_A(X \otimes A, M) \cong \iHom_A(A, M) \otimes X^* \cong M \otimes X^*
    \quad (M \in \mathcal{C}_A).
  \end{equation*}
  Thus $X \otimes A$ is a $\mathcal{C}$-projective $\mathcal{C}$-generator of $\mathcal{C}_A$.
  By the Morita theorem, $A$ and $\iEnd_A(X \otimes A)$ are Morita equivalent.
  By \eqref{eq:internal-End-X-M}, the latter is isomorphic to
  \begin{equation*}
    \iEnd_A(X \otimes A) \cong X \otimes \iEnd_A(A) \otimes X^* \cong X \otimes A \otimes X^*. \qedhere
  \end{equation*}
\end{proof}

\subsection{Eilenberg-Watts equivalence}
\label{subsec:EW-equivalence}

Let $\mathcal{C}$ be a finite tensor category, and let $A$ be an algebra in $\mathcal{C}$. The tensor product $X \otimes_A Y$ of $X \in \mathcal{C}_A$ and $Y \in {}_A\mathcal{C}$ is defined to be the coequalizer of $a_X \otimes \id_Y$ and $\id_X \otimes a_Y$, where $a_X$ and $a_Y$ are the right and left actions of $A$ on $X$ and $Y$, respectively. There is an equivalence
\begin{equation}
  \label{eq:EW-equivalence-in-FTC}
  {}_A\mathcal{C}_B \to \Rex_{\mathcal{C}}(\mathcal{C}_A, \mathcal{C}_B),
  \quad M \mapsto (-) \otimes_A M
\end{equation}
of categories, which we call the {\em Eilenberg-Watts equivalence} in $\mathcal{C}$ (see \cite[Theorem 4.2]{MR0498792}, where such an equivalence has been established in a more general setting).
Since every finite left $\mathcal{C}$-module category is equivalent to $\mathcal{C}_R$ for some algebra $R$ in $\mathcal{C}$, the equivalence \eqref{eq:EW-equivalence-in-FTC} shows that $\Rex_{\mathcal{C}}(\mathcal{M}, \mathcal{N})$ is a finite abelian category if $\mathcal{M}$ and $\mathcal{N}$ are finite left $\mathcal{C}$-module categories.

Under the Eilenberg-Watts equivalence \eqref{eq:EW-equivalence-in-FTC}, the tensor product over an algebra corresponds to the composition of functors in the reversed order.
By this observation, we see that two algebras $A$ and $B$ in $\mathcal{C}$ are Morita equivalent if and only if there are objects $P \in {}_A\mathcal{C}_B$ and $Q \in {}_B\mathcal{C}_A$ such that
\begin{equation}
  \label{eq:Morita-context}
  \text{$P \otimes_B Q \cong A$ in ${}_A\mathcal{C}_A$ and $Q \otimes_A P \cong B$ in ${}_B\mathcal{C}_B$}.
\end{equation}

Similarly to \eqref{eq:EW-equivalence-in-FTC}, there is also an equivalence between ${}_A\mathcal{C}_B$ and the category of $\bfk$-linear right exact right $\mathcal{C}$-module functors from ${}_B\mathcal{C}$ to ${}_A\mathcal{C}$. Since the condition \eqref{eq:Morita-context} is left-right symmetric, we have the following conclusion: Two algebras $A$ and $B$ in $\mathcal{C}$ are Morita equivalent in $\mathcal{C}$ if and only if ${}_A\mathcal{C}$ and ${}_B\mathcal{C}$ are equivalent as finite right $\mathcal{C}$-module categories.

\subsection{Indecomposable module categories}

Let $\mathcal{C}$ be a finite tensor category, and let $A$ be an algebra in $\mathcal{C}$. An {\em ideal} of $A$ is a subobject of $A$ in ${}_A\mathcal{C}_A$.
We say that $A$ is indecomposable if it is indecomposable as an object of ${}_A\mathcal{C}_A$, or, equivalently, it is non-zero and cannot be written as $A = I \oplus J$ for some non-zero ideals $I$ and $J$ of $A$.

A {\em module full subcategory} of a left or right finite $\mathcal{C}$-module category $\mathcal{M}$ is a non-empty full subcategory of $\mathcal{M}$ closed under finite direct sums, subquotients and the action of $\mathcal{C}$. We say that $\mathcal{M}$ is {\em indecomposable} if it is non-zero and cannot be written as $\mathcal{M} = \mathcal{V} \oplus \mathcal{W}$ for some non-zero module full subcategories $\mathcal{V}$ and $\mathcal{W}$ of $\mathcal{M}$.
This property is a categorical counterpart of the indecomposability of an algebra:

\begin{lemma}
  \label{lem:alg-in-C-indecomposability}
  For an algebra $A$ in $\mathcal{C}$, the following are equivalent:
  \begin{enumerate}
  \item $A$ is indecomposable as an algebra in $\mathcal{C}$.
  \item $\mathcal{C}_A$ is indecomposable as a finite left $\mathcal{C}$-module category.
  \item ${}_A\mathcal{C}$ is indecomposable as a finite right $\mathcal{C}$-module category.
  \end{enumerate}
\end{lemma}
\begin{proof}
  We set $\mathcal{M} = \mathcal{C}_A$ and $\mathcal{E} = \Rex_{\mathcal{C}}(\mathcal{M}, \mathcal{M})$.
  A decomposition of $\mathcal{M}$ into a direct sum of left $\mathcal{C}$-module subcategories corresponds to a direct sum decomposition of $\id_{\mathcal{M}}$ as an object of $\mathcal{E}$, and therefore $\mathcal{M}$ is indecomposable if and only if $\id_{\mathcal{M}}$ is indecomposable as an object of $\mathcal{E}$.
  We have an equivalence $\mathcal{E} \approx {}_A\mathcal{C}_A$ and, under this equivalence, the object $A \in {}_A\mathcal{C}_A$ corresponds to $\id_{\mathcal{M}} \in \mathcal{E}$. Hence the indecomposability of $\mathcal{M}$ is equivalent to the indecomposability of $A \in {}_A\mathcal{C}_A$. This shows (1) $\Leftrightarrow$ (2). A left-right reversed argument shows (1) $\Leftrightarrow$ (3).
\end{proof}

\subsection{Exact module categories}

Let $\mathcal{C}$ be a finite tensor category.
We recall the following important class of finite $\mathcal{C}$-module categories and algebras:

\begin{definition}
  \label{def:exact-alg-in-FTC}
  An {\em exact left $\mathcal{C}$-module category} is a finite left $\mathcal{C}$-module category $\mathcal{M}$ where every object of $\mathcal{M}$ is $\mathcal{C}$-projective.
  An algebra $A$ in $\mathcal{C}$ is {\em exact} if $\mathcal{C}_A$ is an exact left $\mathcal{C}$-module category.
\end{definition}

By a characterization of $\mathcal{C}$-projectivity given after Definition~\ref{def:C-projective}, one finds that the above definition agrees with the original definition of exact module categories given in \cite{MR2119143}.
We note the following characterization of the exactness of an algebra:

\begin{lemma}
  \label{lem:exact-algebra-rigitidy}
  An algebra $A$ in $\mathcal{C}$ is exact if and only if the tensor product of the monoidal category ${}_A\mathcal{C}_A$ is exact in each variable.
\end{lemma}
\begin{proof}
  This lemma can be proved by adding a little argument to known results on exact module categories exhibited in \cite{MR3242743}; see, {\it e.g.}, \cite[Theorem 5.1]{2024arXiv240806314S} for details.
\end{proof}

We recall that every finite left $\mathcal{C}$-module category is equivalent to $\mathcal{C}_A$ for some algebra $A$ in $\mathcal{C}$. By rephrasing Lemma~\ref{lem:exact-algebra-rigitidy} by using the Eilenberg-Watts equivalence, we obtain:

\begin{lemma}
  \label{lem:exact-module-cat-endo}
  A finite left $\mathcal{C}$-module category $\mathcal{M}$ is exact if and only if the tensor product of the monoidal category $\Rex_{\mathcal{C}}(\mathcal{M}, \mathcal{M})$ is exact in each variable.
\end{lemma}

We are especially interested in indecomposable exact algebras in $\mathcal{C}$.
To give basic properties of such algebras in the next subsection, we prepare Lemma~\ref{lem:exact-algebra-indecomposability} below on indecomposable exact module categories.

\begin{lemma}
  \label{lem:exact-algebra-indecomposability}
  Let $\mathcal{M}$ be a non-zero exact left $\mathcal{C}$-module category. Then $\mathcal{M}$ is indecomposable if and only if every non-zero object of $\mathcal{M}$ is a $\mathcal{C}$-generator.
\end{lemma}
\begin{proof}
  If $\mathcal{M}$ is decomposed as $\mathcal{M} = \mathcal{V} \oplus \mathcal{W}$ for some non-zero $\mathcal{C}$-module full subcategories $\mathcal{V}$ and $\mathcal{W}$ of $\mathcal{M}$, then every non-zero object of $\mathcal{V}$ and $\mathcal{W}$ is not a $\mathcal{C}$-generator. The `if' part is proved by contraposition.

  The converse is \cite[Exercise 7.10.4]{MR3242743}. For the reader's convenience, we give a proof. We assume that $\mathcal{M}$ is indecomposable. Since every non-zero object of $\mathcal{M}$ has a simple quotient, it suffices to show that every simple object of $\mathcal{M}$ is a $\mathcal{C}$-generator. So, we fix a simple object $V$ of $\mathcal{M}$ and aim to show that $V$ is a $\mathcal{C}$-generator.

  We first show that every simple object $W \in \mathcal{M}$ is a quotient of $Q \catactl V$ for some projective $Q \in \mathcal{C}$. By \cite[Proposition 7.7.2]{MR3242743}\footnote{The base field $\bfk$ is assumed to be algebraically closed in \cite[Proposition 7.7.2]{MR3242743} to use the multiplicity formula $[M : V] = \dim_{\bfk} \Hom_{\mathcal{A}}(P(V), M)$ for objects $M$ and $V$ of a finite abelian category $\mathcal{A}$, where $V$ is simple and $P(V)$ is a projective cover of $V$. The proof of \cite{MR3242743} still works in our setting since the equivalence $[M : V] \ne 0$ $\Leftrightarrow$ $\Hom_{\mathcal{A}}(P(V), M) \ne 0$ still holds without the assumption that $\bfk$ is algebraically closed.}, there is an object $X \in \mathcal{C}$ such that $X \catactl V$ has $W$ as a composition factor. Thus there is a subobject $V'$ of $X \catactl V$ and an epimorphism $f : V' \to W$ in $\mathcal{M}$. Let $\pi : P \to \unitobj$ be the projective cover of the unit object $\unitobj \in \mathcal{C}$.
  Since $P \catactl V'$ is projective, and since every projective object of an exact $\mathcal{C}$-module category is injective \cite[Corollary 7.6.4]{MR3242743}, the inclusion morphism $P \catactl V' \hookrightarrow P \catactl (X \catactl V)$ splits. Letting $s$ be the splitting morphism, we obtain an epimorphism $(\pi \catactl f) s : P \catactl (X \catactl V) \to W$ in $\mathcal{M}$. Thus $Q = P \otimes X \in \mathcal{C}$ meets the requirements.

  Now we prove that $V$ is in fact a $\mathcal{C}$-generator. By the above argument, for every simple object $W \in \mathcal{M}$, there is a projective object $Q \in \mathcal{C}$ such that $Q \catactl V$ has $W$ as a quotient. Since $Q \catactl V$ is projective, the projective cover of $W$ is a direct summand of $Q \catactl V$. In particular, every indecomposable projective object of $\mathcal{M}$ is a quotient of $Q \catactl V$ for some projective object $Q \in \mathcal{C}$. Since every object of $\mathcal{M}$ is a quotient of a finite direct sum of indecomposable projective objects, $V$ is a $\mathcal{C}$-generator. The proof is done.
\end{proof}

\subsection{Simple algebras}

Let $\mathcal{C}$ be a finite tensor category. 
As a natural generalization of the notion of a simple ring, we introduce:

\begin{definition}
  \label{def:simple-alg-in-FTC}
  An algebra $A$ in $\mathcal{C}$ is said to be {\em simple} if it is simple in ${}_A\mathcal{C}_A$.
\end{definition}

An important fact is that a simple algebra in $\mathcal{C}$ is exact. We would like to add a few words about the background.
The exactness of simple algebras was first conjectured by Etingof and Ostrik in \cite{MR4237968}. They have also verified their conjecture for the case where $\mathcal{C} = \lmod{H}$ for some finite-dimensional Hopf algebra $H$ based on Skryabin's results on relative Hopf modules \cite{MR2286047}.
Finally, just recently, Coulembier, Stroi\'nski and Zorman showed that Etingof and Ostrik's conjecture is true in general \cite{2025arXiv250106629C}.

It is well-known that a ring $R$ is simple as a left $R$-module if and only if it is simple as a right $R$-module, if and only if $R$ is a division ring.
Kesten and Walton \cite{2025arXiv250202532K} proved that an analogous result holds in a multi-fusion category.
By using Lemma~\ref{lem:exact-algebra-indecomposability}, we extend their result to a non-semisimple setting as follows:

\begin{lemma}
  \label{lem:exact-alg-left-and-right-simple}
  An algebra $A$ in a finite tensor category $\mathcal{C}$ is simple as a left $A$-module in $\mathcal{C}$ if and only if it is simple as a right $A$-module in $\mathcal{C}$.
\end{lemma}
\begin{proof}
  We only give a proof of the `only if' part since the converse is similar.
  Suppose that $A$ is simple in ${}_A\mathcal{C}$.
  Since the double dual functor induces a bijection between subobjects of $A \in \mathcal{C}_A$ and subobjects of $A^{**} \in \mathcal{C}_{A^{**}}$, it suffices to show that $A^{**}$ is simple in $\mathcal{C}_{A^{**}}$. By the formula
  \begin{equation*}
    \iHom_A(M, N) = (M \otimes_A {}^*\!N)^*
    \quad (M, N \in \mathcal{C}_A)
  \end{equation*}
  of the internal Hom functor of $\mathcal{C}_A$ \cite[Example 7.9.8]{MR3242743}, we can identify the algebra $\iEnd_A(A^*)$ with $A^{**}$.
  The main result of \cite{2025arXiv250106629C} implies that $\mathcal{C}_A$ is indecomposable and exact.
  Thus, by Lemma~\ref{lem:exact-algebra-indecomposability}, the object $A^* \in \mathcal{C}_A$ is a $\mathcal{C}$-projective $\mathcal{C}$-generator. Hence the Morita theorem gives an equivalence
  \begin{equation*}
    \Phi := \iHom_A(A^*, -) : \mathcal{C}_A \to \mathcal{C}_{A^{**}}
  \end{equation*}
  of left $\mathcal{C}$-module categories.
  Since the duality functor $(-)^* : {}_A\mathcal{C} \to \mathcal{C}_A$ is an anti-equivalence, the object $A^* \in \mathcal{C}_A$ is simple.
  Hence $A^{**} \in \mathcal{C}_{A^{**}}$, which is isomorphic to $\Phi(A^*)$, is also simple as an object of $\mathcal{C}_{A^{**}}$.
  The proof is done.
\end{proof}

Following the terminology adopted in \cite{2025arXiv250202532K}, we introduce:

\begin{definition}
  \label{def:division-alg-in-C}
  A {\em division algebra} in $\mathcal{C}$ is an algebra $A$ in $\mathcal{C}$ that is simple as a left (or, equivalently, right) module over $A$ in $\mathcal{C}$.
\end{definition}

The following lemma, which is useful for classifying indecomposable exact algebras up to Morita equivalence, is an analogue of the fact that a finite-dimensional simple algebra is Morita equivalent to a division algebra.

\begin{lemma}
  \label{lem:exact-algebra-indec-Morita-equiv-to-right-simple}
  A non-zero algebra $A$ in $\mathcal{C}$ is indecomposable exact if and only if it is Morita equivalent to a division algebra in $\mathcal{C}$.
\end{lemma}
\begin{proof}
  The `if' part follows from the fact that a simple algebra in $\mathcal{C}$ is indecomposable exact \cite{2025arXiv250106629C}, and that the indecomposability and the exactness of an algebra are Morita invariants.
  We prove the `only if' part. Let $A$ be a non-zero indecomposable exact algebra in $\mathcal{C}$.
  We fix a simple object $M \in \mathcal{C}_A$ and set $D := \iEnd(M)$.
  By Lemma~\ref{lem:exact-algebra-indecomposability} and the Morita theorem, we have an equivalence $\Phi := \iHom(M, -) : \mathcal{C}_A \to \mathcal{C}_D$ of left $\mathcal{C}$-module categories. Since $D = \Phi(M)$ and $M \in \mathcal{C}_A$ is simple, the object $D \in \mathcal{C}_D$ is simple. Hence $A$ is Morita equivalent to the division algebra $D$ in $\mathcal{C}$.
\end{proof}

We also give a criterion for two division algebras to be Morita equivalent:

\begin{lemma}
  \label{lem:division-alg-Morita-equiv}
  Two division algebras $A$ and $B$ in $\mathcal{C}$ are Morita equivalent in $\mathcal{C}$ if and only if there are an indecomposable exact left $\mathcal{C}$-module category $\mathcal{M}$ and two simple objects $V$ and $W$ of $\mathcal{M}$ such that $A \cong \iEnd(V)$ and $B \cong \iEnd(W)$ as algebras in $\mathcal{C}$.
\end{lemma}
\begin{proof}
  The `if' part follows from the Morita theorem and Lemma~\ref{lem:exact-algebra-indecomposability}. To prove the converse, we assume that $A$ and $B$ are Morita equivalent in $\mathcal{C}$. By definition, there is an equivalence $\Phi : \mathcal{C}_B \to \mathcal{C}_A$ of finite left $\mathcal{C}$-module categories. We now set $\mathcal{M} = \mathcal{C}_A$, $V = A \in \mathcal{M}$ and $W = \Phi(B) \in \mathcal{M}$. Then, since $A$ and $B$ are division algebras in $\mathcal{C}$, the objects $V$ and $W$ are simple. Moreover, we have
  \begin{equation*}
    \iEnd_A(V) = \iEnd_A(A) \cong A,
    \quad \iEnd_A(W) = \iEnd_{A}(\Phi(B)) \cong \iEnd_{B}(B) \cong B
  \end{equation*}
  as algebras. The proof is done.
\end{proof}

\subsection{Haploid algebras}

Let $\mathcal{C}$ be a finite tensor category.
We say that an algebra $A$ in $\mathcal{C}$ is {\em haploid} \cite{MR2029790} if $\dim_{\bfk} \Hom_{\mathcal{C}}(\unitobj, A) = 1$.

\begin{lemma}
  \label{lem:haploid-implies-indecomposable}
  A haploid algebra in $\mathcal{C}$ is indecomposable.
\end{lemma}
\begin{proof}
  Let $A$ be a haploid algebra in $\mathcal{C}$.
  We assume that $A = I_1 \oplus I_2$ for some ideals $I_1$ and $I_2$ of $A$, and let $i_k : I_k \to A$ ($k = 1, 2$) be the inclusion morphisms.
  Then, since $\Hom_{\mathcal{C}}(\unitobj, A) = \Hom_{\mathcal{C}}(\unitobj, I_1) \oplus \Hom_{\mathcal{C}}(\unitobj, I_2)$, the unit $\unitobj \to A$ factors through either $i_1$ or $i_2$. In the former case, we have $A = I_1$ by the unit axiom. In the latter case, we have $A = I_2$. Thus $A$ is indecomposable.
\end{proof}

The converse of this lemma does not hold. For example, the matrix algebra of degree $\ge 2$ is an indecomposable non-haploid algebra in $\Vect$. A finite field extension of $\bfk$ of degree $\ge 2$ is also an indecomposable non-haploid algebra in $\Vect$.
We note that the former example is Morita equivalent to $\bfk$, while the latter does not exist when $\bfk$ is algebraically closed. In a finite tensor category, we have:

\begin{lemma}
  \label{lem:division-alg-is-haploid}  
  Suppose that $\bfk$ is algebraically closed.
  Then every division algebra in $\mathcal{C}$ is haploid.
  In particular, every indecomposable exact algebra in $\mathcal{C}$ is Morita equivalent to a haploid algebra in $\mathcal{C}$.
\end{lemma}
\begin{proof}
  If $A$ is a division algebra in $\mathcal{C}$, then we have isomorphisms
  \begin{equation*}
    \Hom_{\mathcal{C}}(\unitobj, A)
    \cong \Hom_{\mathcal{C}_A}(A,A) \cong \bfk
  \end{equation*}
  by Schur's lemma and the assumption that $\bfk$ is algebraically closed.
  The rest of the statement follows from Lemma~\ref{lem:exact-algebra-indec-Morita-equiv-to-right-simple}.
\end{proof}

In connection with this lemma, we note that a haploid exact algebra in $\mathcal{C}$ is not necessarily a division algebra in $\mathcal{C}$.
To see this, we fix a non-zero object $M$ of an indecomposable exact left $\mathcal{C}$-module category $\mathcal{M}$ and consider the algebra $B = \iEnd(M)$ in $\mathcal{C}$. By the Morita theorem, we have an equivalence $\mathcal{M} \approx \mathcal{C}_B$ of left $\mathcal{C}$-module categories which sends $M$ to $B$.
Now we assume that $M$ is not simple and $\End_{\mathcal{M}}(M) \cong \bfk$.
Then, by the same argument as the proof of Lemma \ref{lem:exact-algebra-indec-Morita-equiv-to-right-simple}, $B$ is a haploid exact algebra that is not a division algebra.

\subsection{Morita equivalence in pointed finite tensor categories}
\label{subsec:pointed-FTCs}

A {\em pointed finite tensor category} is a finite tensor category where every simple object is invertible (with respect to the tensor product, up to isomorphism). We conclude this section by giving the following useful criterion for two division algebras in a pointed finite tensor category to be Morita equivalent:

\begin{theorem}
  \label{thm:Morita-in-pointed-FTC}
  Let $\mathcal{C}$ be a pointed finite tensor category.
  Then two division algebras $A$ and $B$ in $\mathcal{C}$ are Morita equivalent in $\mathcal{C}$ if and only if there exists a simple object $g \in \mathcal{C}$ such that $B \cong g \otimes A \otimes g^{*}$ as algebras in $\mathcal{C}$.
\end{theorem}

Let $\mathcal{C}$ be a finite tensor category, and let $\mathcal{M}$ be a finite left $\mathcal{C}$-module category. If $g \in \mathcal{C}$ is invertible, then the functor $g \catactl (-) : \mathcal{M} \to \mathcal{M}$ is an equivalence. In particular, the object $g \catactl V$ is simple if $V \in \mathcal{M}$ is. Hence the group $\Grp(\mathcal{C})$ of the isomorphism classes of invertible objects of $\mathcal{C}$ acts on the set $\mathrm{Irr}(\mathcal{M})$ of the isomorphism classes of simple objects of $\mathcal{M}$. To prove Theorem \ref{thm:Morita-in-pointed-FTC}, we note:

\begin{lemma}
  \label{lem:Morita-in-pointed-FTC-lemma}
  Suppose that $\mathcal{C}$ is a pointed finite tensor category and $\mathcal{M}$ is an indecomposable exact left $\mathcal{C}$-module category. Then the action of $\Grp(\mathcal{C})$ on $\mathrm{Irr}(\mathcal{M})$ is transitive.
\end{lemma}
\begin{proof}
  Let $V$ and $W$ be simple objects of $\mathcal{M}$. By Lemma~\ref{lem:exact-algebra-indecomposability}, there is an object $P \in \mathcal{C}$ such that $W$ is a quotient of $P \catactl V$. Let $g_1, \cdots, g_n$ be the composition factors of $P$, counting multiplicities. Since $\mathcal{C}$ is pointed, $g_i \catactl V$ is simple for each $i$. This implies that the composition factors of $P \catactl V$ are $g_i \catactl V$ ($i = 1, \cdots, n$), counting multiplicities. Thus $W$ is isomorphic to $g_i \catactl V$ for some $i$. The proof is done.
\end{proof}

We note that the exactness of $\mathcal{M}$ is essential in Lemma~\ref{lem:Morita-in-pointed-FTC-lemma}. Indeed, we assume that $\mathrm{char}(\bfk) \ne 2$ for this example and let $A$ be the algebra generated by $x$ and $g$ subject to $x^2 = 0$, $g^2 = 1$ and $g x = - x g$. Then $\mathcal{M} := \lmod{A}$ is an indecomposable module category over the pointed finite tensor category $\mathcal{C} := \Vect$. The set $\mathrm{Irr}(\mathcal{M})$ consists of two elements, while $\Grp(\mathcal{C})$ is trivial. Hence the action of $\Grp(\mathcal{C})$ on $\mathrm{Irr}(\mathcal{M})$ is not transitive.

\begin{proof}[Proof of Theorem~\ref{thm:Morita-in-pointed-FTC}]
  Let $\mathcal{C}$, $A$ and $B$ be as in the statement.
  The `if' part follows from Lemma~\ref{lem:Morita-eq-matrix-alg}. To prove the converse, we assume that $A$ and $B$ are Morita equivalent. Then, by Lemma~\ref{lem:division-alg-Morita-equiv}, there are an indecomposable exact left $\mathcal{C}$-module category $\mathcal{M}$ and two simple objects $V$ and $W$ of $\mathcal{M}$ such that $\iEnd(V) \cong A$ and $\iEnd(W) \cong B$ as algebras in $\mathcal{C}$. By Lemma~\ref{lem:Morita-in-pointed-FTC-lemma}, there is an invertible object $g \in \mathcal{C}$ such that $W \cong g \catactl V$. Thus we have
  \begin{equation*}
    B \cong \iEnd(W)
    \cong \iEnd(g \catactl V)
    \cong g \otimes \iEnd(V) \otimes g^*
    \cong g \otimes A \otimes g^*,
  \end{equation*}
  where \eqref{eq:internal-End-X-M} was used at the third isomorphism. The proof is done.
\end{proof}

\section{Exact module categories and comodule algebras}
\label{sec:exact-mod-and-comod-alg}

\subsection{Notation and convention}

In this section, {\em an algebra, a coalgebra and a Hopf algebra are always assumed to be finite-dimensional}. Given coalgebras $C$ and $D$, we denote by ${}^C\Mod$, $\Mod^D$ and ${}^C\Mod^D$ the categories of finite-dimensional left $C$-comodules, right $D$-comodules and $C$-$D$-bicomodules, respectively. To be consistent, given algebras $A$ and $B$, we write $\lmod{A}$, $\rmod{B}$ and $\bimod{A}{B}$ as ${}_A\Mod$, $\Mod_B$ and ${}_A\Mod_B$, respectively.

Let $H$ be a Hopf algebra. By our convention that $H$ is finite-dimensional, the categories ${}_H\Mod$ and ${}^H\Mod$ are finite tensor categories. A {\em left $H$-comodule algebra} is nothing but an algebra in ${}^H\Mod$. Given left $H$-comodule algebras $A$ and $B$, the categories ${}_A\mathcal{C}$, $\mathcal{C}_B$ and ${}_A\mathcal{C}_B$ with $\mathcal{C} = {}^H\Mod$ will be denoted by ${}^H_A\Mod$, ${}^H\Mod_B$ and ${}^H_A\Mod_B$, respectively. 

The comultiplication and the counit of a coalgebra $C$ are denoted by $\Delta_C$ and $\varepsilon_C$ or, for short, by $\Delta$ and $\varepsilon$, respectively. We define
\begin{equation*}
  \Grp(C) = \{ c \in C \mid \Delta(c) = c \otimes c, \varepsilon(c) = 1 \}
\end{equation*}
and call an element of this set a {\em grouplike} element. We use the Sweedler notation, such as $\Delta(c) = c_{(1)} \otimes c_{(2)}$ and $\Delta(c_{(1)}) \otimes c_{(2)} = c_{(1)} \otimes c_{(2)} \otimes c_{(3)} = c_{(1)} \otimes \Delta(c_{(2)})$, to express the comultiplication of $c \in C$. The Sweedler notation is also used for a comodule: If $M$ is a left $C$-comodule, then the coaction $\delta_M : M \to C \otimes M$ is expressed as $\delta_M(m) = m_{(-1)} \otimes m_{(0)}$.

\subsection{Terminology on comodule algebras}
\label{subsec:term-comod-alg}

We first introduce some terminology on comodule algebras. Let $H$ be a Hopf algebra (which is finite-dimensional by our convention). We recall that a subspace of a left $H$-comodule is said to be {\em $H$-costable} if it is a subcomodule.

\begin{definition}
  \label{def:H-comod-alg-basics}
  Let $A$ be a left $H$-comodule algebra.
  \begin{enumerate}
  \item $A$ is said to be {\em $H$-simple} (respectively, {\em left $H$-simple}, {\em right $H$-simple}) if it is a non-zero algebra with no non-trivial $H$-costable ideal (respectively, left ideal, right ideal).
  \item $A$ is said to be {\em $H$-indecomposable} if it is non-zero and cannot be written as $A = I \oplus J$ for some non-trivial $H$-costable ideals $I$ and $J$ of $A$.
  \item $A$ is said to {\em have trivial coinvariants} if ${}^{\mathrm{co}\,H}\!A$ is one-dimensional.
    Here, ${}^{\mathrm{co}\,H}\!M$ for $M \in {}^H\Mod$ is the space of left $H$-coinvariants defined by
    \begin{equation*}
      {}^{\mathrm{co}\,H} \! M := \{ m \in M \mid \delta_M(m) = 1 \otimes m \}.
    \end{equation*}
  \end{enumerate}
\end{definition}

By definition, $A$ is $H$-simple, left $H$-simple and right $H$-simple if and only if it is simple as an object of the category ${}^H_A\Mod_A^{}$, ${}^H_A\Mod$ and ${}^H_{}\Mod_A^{}$, respectively. Thus, by Lemma~\ref{lem:exact-alg-left-and-right-simple}, $A$ is left $H$-simple if and only if it is right $H$-simple, if and only if it is a division algebra in ${}^H\Mod$. Also, $A$ is $H$-indecomposable if and only if it is an indecomposable algebra in ${}^H\Mod$ in the sense of the previous section.

Set $\mathcal{C} = {}_H\Mod$.
The category ${}_A\Mod$ is a finite left $\mathcal{C}$-module category with the action $\catactl$ defined as follows: For $X \in \mathcal{C}$ and $M \in {}_A\Mod$, we define $X \catactl M = X \otimes_{\bfk} M$ as a vector space. We make it a left $A$-module by the action given by
\begin{equation*}
  a \cdot (x \otimes m) = a_{(-1)} x \otimes a_{(0)} m
  \quad (a \in A, x \in X, m \in M).
\end{equation*}

\begin{definition}
  \label{def:exact-H-comod-alg}
  An {\em exact left $H$-comodule algebra} is a left $H$-comodule algebra $A$ such that the left ${}_H\Mod$-module category ${}_A\Mod$ is exact.
\end{definition}

This terminology is compatible with Definition~\ref{def:exact-alg-in-FTC}.
Namely, a left $H$-comodule algebra $A$ is exact in the above sense if and only if it is an exact algebra in ${}^H\Mod$ in the sense of Definition~\ref{def:exact-alg-in-FTC}; see Lemma~\ref{lem:H-comod-alg-exactness}.

Andruskiewitsch and Mombelli proved that every indecomposable exact left module category over ${}_H\Mod$ is equivalent to ${}_A\Mod$ for some right $H$-simple left $H$-comodule algebra $A$ \cite[Theorem 3.3]{MR2331768}.
One of the aims of this section is to give a proof of their result with emphasis on the viewpoint of the theory of finite tensor categories. We then give several supplemental results for right $H$-simple left $H$-comodule algebras which will be used in the later sections.

\subsection{Finite module categories over ${}_H\Mod$}

Let $H$ be a Hopf algebra. We first show that every finite left module category over $\mathcal{C} := {}_H\Mod$ is given by a left comodule algebra:

\begin{theorem}[{\it cf}. {\cite[Proposition 1.19]{MR2331768}}]
  \label{thm:finite-mod-cat-over-H-Mod}
  For every finite left $\mathcal{C}$-module category $\mathcal{M}$, there is a left $H$-comodule algebra $A$ such that $\mathcal{M} \approx {}_A\Mod$ as left $\mathcal{C}$-module categories.
\end{theorem}

The proof is essentially the same as that of \cite[Proposition 1.19]{MR2331768}, where $\mathcal{M}$ is assumed to be indecomposable and exact. 

\begin{proof}
  By the Morita theorem, there is an algebra $R$ in $\mathcal{C}$ such that $\mathcal{M} \approx \mathcal{C}_R$. Now we introduce the left $H$-comodule algebra $A$ as follows: As a vector space, $A = R \otimes H$. The multiplication and the left $H$-coaction are given respectively by
  \begin{equation*}
    (r \otimes h) (r' \otimes h') = (h_{(2)} \triangleright r') r \otimes h_{(1)} h'
    \quad \text{and}
    \quad \delta_A(r \otimes h) = h_{(1)} \otimes (r \otimes h_{(2)})
  \end{equation*}
  for $h, h' \in H$ and $r, r' \in R$, where $\triangleright$ is the left action of $H$ on $R$ (namely, $A$ is the smash product $R^{\op} \# H^{\cop}$). The category $\mathcal{C}_R$ is identified with ${}_A\Mod$, and therefore $\mathcal{M} \approx {}_A\Mod$. The proof is done.
\end{proof}

\subsection{Equivariant Eilenberg-Watts theorem}

Let $H$ be a Hopf algebra, and let $A$ and $B$ be left $H$-comodule algebras.
A bimodule $M \in {}_A\Mod_B$ defines a right exact functor $M \otimes_B (-) : {}_B\Mod \to {}_A\Mod$. If $M \in {}^H_A\Mod_B$, then this functor is an (oplax) left module functor over $\mathcal{C} := {}_H\Mod$ with the structure morphism
\begin{align*}
  \xi_{X,W} : M \otimes_B (X \catactl W) & \to X \catactl (M \otimes_B W), \\
  m \otimes_B (x \otimes w) & \mapsto m_{(-1)} x \otimes (m_{(0)} \otimes_B w),
\end{align*}
where $m \in M$, $x \in X \in {}_H\Mod$ and $w \in W \in {}_B\Mod$.

\begin{theorem}
  \label{thm:equivariant-EW}
  For left $H$-comodule algebras $A$ and $B$, the functor
  \begin{equation*}
    \EW(A,B) : {}^H_A\Mod_B^{} \to \Rex_{\mathcal{C}}({}_B\Mod, {}_A\Mod),
    \quad M \mapsto M \otimes_B (-)
  \end{equation*}
  is an equivalence.
\end{theorem}

This theorem, which we call {\em the $H$-equivariant Eilenberg-Watts theorem}, is due to Andruskiewitsch and Mombelli \cite[Proposition 1.23]{MR2331768}. We give a proof of a generalization of this theorem in Appendix~\ref{appendix:equiv-EW}.

Let $A$, $B$ and $C$ be left $H$-comodule algebras.
As a supplement to the above theorem, we note that the functor $\EW(-,-)$ is compatible with the composition of module functors in the sense that the following diagram is commutative:
\begin{equation}
  \label{eq:equivariant-EW-compatibility}
  \begin{tikzcd}[column sep = 64pt]
    {}^H_A\Mod_B^{} \times {}^H_B\Mod_C^{}
    \arrow[r, "{\otimes_B}"]
    \arrow[d, "{\EW(A,B) \times \EW(B,C)}"']
    & {}^H_A\Mod_C^{} \arrow[d, "{\EW(A,C)}"] \\
    \Rex_{\mathcal{C}}({}_B\Mod, {}_A\Mod)
    \times \Rex_{\mathcal{C}}({}_C\Mod, {}_B\Mod)
    \arrow[r, "\text{composition}"]
    & \Rex_{\mathcal{C}}({}_C\Mod, {}_A\Mod).
  \end{tikzcd}
\end{equation}

\subsection{The Morita duality}
\label{subsec:morita-duality}

For a finite tensor category $\mathcal{C}$, we have introduced the 2-category $\ullMod{\mathcal{C}}$ in Subsection \ref{subsec:finite-mod-cat}.
Let $H$ be a Hopf algebra. In this subsection, we deploy the idea of Morita duality of finite tensor categories \cite{MR3242743} and establish a duality between $\ullMod{\mathcal{C}}$ and $\ullMod{\mathcal{D}}$, where $\mathcal{C} = {}_H\Mod$ and $\mathcal{D} = {}^H\Mod$.

We note that both $\mathcal{C}$ and $\mathcal{D}$ act on $\Vect$ through the forgetful functor to $\Vect$.
By the $H$-equivariant Eilenberg-Watts theorem, we have an equivalence
\begin{equation*}
  \EW(\bfk, \bfk) : \mathcal{D} \to \Rex_{\mathcal{C}}(\Vect, \Vect)
\end{equation*}
of tensor categories.
For $\mathcal{M} \in \ullMod{\mathcal{C}}$, we set
\begin{equation*}
  \mathcal{M}^* = \Rex_{\mathcal{C}}(\mathcal{M}, \Vect)
\end{equation*}
and make it a left $\mathcal{D}$-module category by
\begin{equation*}
  \mathcal{D} \times \mathcal{M}^*
  \xrightarrow{\quad \EW(\bfk, \bfk) \times \id \quad}
  \Rex_{\mathcal{C}}(\Vect, \Vect)
  \times \Rex_{\mathcal{C}}(\mathcal{M}, \Vect)
  \xrightarrow{\quad \mathrm{comp} \quad}
  \mathcal{M}^*,
\end{equation*}
where $\mathrm{comp}$ means the composition of module functors. Since the composition corresponds to the tensor product over an algebra on the bimodule side, the action of $\mathcal{D}$ on $\mathcal{M}^*$ is right exact in each variable. Hence $\mathcal{M}^*$ is a finite left $\mathcal{D}$-module category.

\begin{lemma}
  \label{lem:dual-of-A-Mod}
  For a left $H$-comodule algebra $A$, there is an equivalence
  \begin{equation*}
    ({}_A\Mod)^* \approx {}^H\Mod_A
  \end{equation*}
  of left $\mathcal{D}$-module categories.
\end{lemma}
\begin{proof}
  There is an equivalence $\EW(\bfk, A) : {}^H\Mod_A \to ({}_A\Mod)^*$ of linear categories. The commutative diagram~\eqref{eq:equivariant-EW-compatibility} says that $\EW(\bfk,A)$ is in fact an equivalence of left module categories over $\mathcal{D}$.
\end{proof}

Given a 1-morphism $F: \mathcal{M} \to \mathcal{N}$ in $\ullMod{\mathcal{C}}$, we define the functor $F^* : \mathcal{N}^* \to \mathcal{M}^*$ by $F^*(T) = T \circ F$.
It is obvious that $F^*$ commutes with the actions of $\mathcal{D}$.
If we take left $H$-comodule algebras $A$ and $B$ such that $\mathcal{M} \approx {}_A\Mod$ and $\mathcal{N} \approx {}_B\Mod$ and identify $\mathcal{M}^*$ and $\mathcal{N}^*$ with ${}^{H}\Mod_A$ and ${}^{H}\Mod_B$, respectively, then $F^*$ is isomorphic to $(-) \otimes_B M$, where $M \in {}^{H}_B\Mod_A$ is the object corresponding to $F$ via the equivariant Eilenberg-Watts equivalence.
Thus $F^*$ is a 1-morphism in $\ullMod{\mathcal{D}}$.
Also, given a 2-morphism $\xi : (F : \mathcal{M} \to \mathcal{N}) \Rightarrow (G : \mathcal{M} \to \mathcal{N})$ in $\ullMod{\mathcal{C}}$, we define
\begin{equation*}
  \xi^*_T(M) := T(\xi_M) : F^*(T)(M)
  = T F(M) \to T G(M) = G^*(T)(M)
\end{equation*}
for $T \in \mathcal{N}^*$ and $M \in \mathcal{M}$. The family $\{ \xi^*_T(M) \}_{M \in \mathcal{M}}$ of maps is in fact a morphism $\xi^*_T : F^*(T) \to G^*(T)$ in $\mathcal{M}^*$ that is natural in $T \in \mathcal{N}^*$.
By the above discussion, the assignment $\mathcal{M} \mapsto \mathcal{M}^*$ extends to a 2-functor
\begin{equation*}
  (-)^* : \ullMod{\mathcal{C}} \to (\ullMod{\mathcal{D}})^{\op},
\end{equation*}
where $\mathcal{K}^{\op}$ for a 2-category $\mathcal{K}$ means the 2-category obtained from $\mathcal{K}$ by reversing the direction of 1-morphisms but not 2-morphisms.

There is a similar construction for finite left $\mathcal{D}$-module categories.
We note that there is an isomorphism ${}^H\Mod \cong {}_{H^{*\op}}\Mod$ of tensor categories. By applying the above argument to $H^{*\op}$, we obtain a 2-functor
\begin{equation*}
  (-)^* : \ullMod{\mathcal{D}} \to (\ullMod{\mathcal{C}})^{\op},
  \quad \mathcal{N}^* := \Rex_{\mathcal{D}}(\mathcal{N}, \Vect),
\end{equation*}
where $\mathcal{C}$ acts on $\mathcal{N}^*$ through the equivalence $\mathcal{C} \approx \Rex_{\mathcal{D}}(\Vect, \Vect)$.

\begin{theorem}
  \label{thm:Morita-duality}
  For $\mathcal{M} \in \ullMod{\mathcal{C}}$, the canonical functor
  \begin{equation*}
    \phi_{\mathcal{M}} : \mathcal{M} \to \mathcal{M}^{**},
    \quad \phi_{\mathcal{M}}(M)(F) = F(M)
    \quad (M \in \mathcal{M}, F \in \mathcal{M}^*)
  \end{equation*}
  is an equivalence of left $\mathcal{C}$-module categories. The same holds for finite left $\mathcal{D}$-module categories. Hence we have an equivalence
  \begin{equation}
    \label{eq:Morita-duality}
    \ullMod{\mathcal{C}} \approx (\ullMod{\mathcal{D}})^{\op}
  \end{equation}
  of 2-categories.
\end{theorem}
\begin{proof}
  We only give a sketch of the proof here; see Appendix \ref{appendix:morita-duality} for details.
  Since the functor $\phi_{\mathcal{M}}$ is natural in $\mathcal{M}$, we may assume that $\mathcal{M} = {}_A\Mod$ for some left $H$-comodule algebra $A$.
  There are equivalences
  \begin{equation*}
    \mathcal{M} = {}_A\Mod \approx {}_A\mathcal{D}_H
    \approx \Rex_{\mathcal{D}}(\mathcal{D}_A, \mathcal{D}_H)
    \approx (\mathcal{D}_A)^* \approx \mathcal{M}^{**},
  \end{equation*}
  where the first and the third equivalences are given by the fundamental theorem for Hopf modules \eqref{eq:fundamental-thm-of-Hopf-modules}, the second one follows from the Eilenberg-Watts theorem in $\mathcal{D}$ and the last one follows from Lemma~\ref{lem:dual-of-A-Mod}. The proof is completed by verifying that the above equivalence is isomorphic to $\phi_{\mathcal{M}}$.
\end{proof}

We call the 2-equivalence \eqref{eq:Morita-duality} of the above theorem the {\em Morita duality} between the finite tensor categories $\mathcal{C} = {}_H\Mod$ and $\mathcal{D} = {}^H\Mod$.
A bijective correspondence between finite left module categories over $\mathcal{C}$ and those over $\mathcal{D}$ up to equivalence has been pointed out in \cite[Proposition 3.9]{MR2331768}; however, it is not formulated as an equivalence of 2-categories.
The duality yields several useful consequences on module categories and comodule algebras. For example,

\begin{lemma}[{\it cf}. {\cite[Theorem 3.10]{MR2331768}}]
  \label{lem:duality-exactness}
  A finite left $\mathcal{C}$-module category $\mathcal{M}$ is exact if and only if the finite left $\mathcal{D}$-module category $\mathcal{M}^*$ is exact.
\end{lemma}
\begin{proof}
  The Morita duality gives an equivalence $\Rex_{\mathcal{C}}(\mathcal{M}, \mathcal{M}) \approx \Rex_{\mathcal{D}}(\mathcal{M}^*, \mathcal{M}^*)^{\rev}$ of linear monoidal categories, where $(-)^{\rev}$ means that the order of the monoidal product is reversed.
  The exactness of $\mathcal{M}$ is equivalent to the exactness of the tensor product of the left-hand side by Lemma \ref{lem:exact-module-cat-endo}, while the exactness of $\mathcal{M}^*$ is equivalent to the exactness of the tensor product of the right-hand side by the same lemma. Thus $\mathcal{M}$ is exact if and only if $\mathcal{M}^*$ is. The proof is done.
\end{proof}

\begin{lemma}
  \label{lem:H-comod-alg-exactness}
  For a left $H$-comodule algebra $A$, the following are equivalent:
  \begin{enumerate}
  \item $A$ is an exact algebra in ${}^H\Mod$ in the sense of Definition~\ref{def:exact-alg-in-FTC}.
  \item $A$ is an exact left $H$-comodule algebra in the sense of Definition~\ref{def:exact-H-comod-alg}.
  \end{enumerate}
\end{lemma}
\begin{proof}
  Apply Lemma~\ref{lem:duality-exactness} to $\mathcal{M} = {}_A\Mod$ and use Lemma~\ref{lem:dual-of-A-Mod}.
\end{proof}

\begin{lemma}
  \label{lem:H-comod-alg-indec}
  For a left $H$-comodule algebra $A$, the following are equivalent:
  \begin{enumerate}
  \item ${}_A\Mod$ is indecomposable as a left $\mathcal{C}$-module category.
  \item ${}^H\Mod_A$ is indecomposable as a left $\mathcal{D}$-module category.
  \item $A$ is $H$-indecomposable in the sense of Definition~\ref{def:H-comod-alg-basics}.
  \end{enumerate}
\end{lemma}
\begin{proof}
  It is obvious from the duality that $\mathcal{M} \in \ullMod{\mathcal{C}}$ is decomposable if and only if $\mathcal{M}^*$ is.
  The equivalence between (1) and (2) follows from this observation and Lemma~\ref{lem:dual-of-A-Mod}.
  The equivalence between (2) and (3) follows from Lemma \ref{lem:alg-in-C-indecomposability}.
\end{proof}

\subsection{Equivariant Morita equivalence}

Let $H$ be a Hopf algebra.
Following \cite{MR2331768}, we say that two left $H$-comodule algebras $A$ and $B$ are {\em $H$-equivariant Morita equivalent} (or {\em $H$-Morita equivalent} for short) if ${}_A\Mod$ and ${}_B\Mod$ are equivalent as left module categories over ${}_H\Mod$.

\begin{theorem}
  \label{thm:H-Morita}
  For two left $H$-comodule algebras $A$ and $B$, the following are equivalent:
  \begin{enumerate}
  \item $A$ and $B$ are $H$-Morita equivalent.
  \item $A$ and $B$ are Morita equivalent in ${}^H\Mod$ in the sense of Definition~\ref{def:Morita-in-FTC}.
  \end{enumerate}
\end{theorem}
\begin{proof}
  If (1) holds, then we have ${}^H\Mod_A \approx ({}_A\Mod)^* \approx ({}_B\Mod)^* \approx {}^H\Mod_B$ as left module categories over ${}^H\Mod$. Namely, (2) holds. The converse is proved in a similar way by using the Morita duality.
\end{proof}

A {\em pointed Hopf algebra} is a Hopf algebra all of whose simple subcoalgebras are one-dimensional. It is obvious that a Hopf algebra $H$ is pointed if and only if the finite tensor category ${}^H\Mod$ is pointed in the sense of Subsection~\ref{subsec:pointed-FTCs}.
As an application of Theorem~\ref{thm:Morita-in-pointed-FTC}, we prove:

\begin{theorem}
  \label{thm:H-Morita-pointed}
  Let $H$ be a pointed Hopf algebra, and let $A$ and $B$ be right $H$-simple left $H$-comodule algebras. Then $A$ and $B$ are $H$-equivariant Morita equivalent if and only if there is an element $g \in \Grp(H)$ such that $B \cong g A g^{-1}$ as left $H$-comodule algebras. Here, $g A g^{-1}$ is the algebra $A$ with the new left $H$-coaction
  \begin{equation*}
    A \to H \otimes A,
    \quad a \mapsto g a_{(-1)} g^{-1} \otimes a_{(0)}
    \quad (a \in A).
  \end{equation*}
\end{theorem}

This result has been given as \cite[Theorem 4.2]{MR2860429} under the assumption that $\bfk$ is an algebraically closed field of characteristic zero. We give a shorter proof without any assumption on $\bfk$ by using techniques of finite tensor categories.

\begin{proof}
  Since $g A g^{-1} \cong (\bfk g) \otimes A \otimes (\bfk g)^{*}$ as algebras in ${}^H\Mod$ for $g \in \Grp(H)$, the `if' part follows from Lemma~\ref{lem:Morita-eq-matrix-alg}. To prove the converse, we assume that $A$ and $B$ are $H$-equivariant Morita equivalent. Then, by Theorem~\ref{thm:H-Morita}, $A$ and $B$ are Morita equivalent in the finite tensor category $\mathcal{C} := {}^H\Mod$. Hence, by Theorem~\ref{thm:Morita-in-pointed-FTC}, there is a simple object $X \in \mathcal{C}$ such that $B \cong X \otimes A \otimes X^*$. Since $H$ is pointed, $X \cong \bfk g$ for some $g \in \Grp(H)$. For this $g$, the algebra $X \otimes A \otimes X^*$ is isomorphic to $g A g^{-1}$. The proof is done.
\end{proof}

\subsection{Indecomposable exact module categories over ${}_H\Mod$}

Let $H$ be a Hopf algebra.
The following theorem, due to Andruskiewitsch and Mombelli \cite{MR2331768}, is fundamental for the classification of indecomposable exact module categories over the finite tensor category $\mathcal{C} := {}_H\Mod$.

\begin{theorem}[{\cite[Theorem 3.3]{MR2331768}}]
  \label{thm:AM-exact-mod-cat}
  A finite left module category $\mathcal{M}$ over $\mathcal{C}$ is indecomposable and exact if and only if $\mathcal{M} \approx {}_A\Mod$ for some right $H$-simple left $H$-comodule algebra $A$.
\end{theorem}
\begin{proof}
  We first prove the `if' part. Suppose that there exists a right $H$-simple left $H$-comodule algebra $A$ such that $\mathcal{M} \approx {}_A\Mod$. By the fact recalled after Definition~\ref{def:simple-alg-in-FTC}, $A$ is an exact algebra in $\mathcal{D} := {}^H\Mod$. By Lemmas~\ref{lem:H-comod-alg-exactness} and \ref{lem:H-comod-alg-indec}, $\mathcal{M}$ is indecomposable and exact.

  To prove the converse, we assume that $\mathcal{M}$ is indecomposable and exact.
  Then, by the duality, $\mathcal{M}^*$ is an indecomposable exact module category over $\mathcal{D} := {}^H\Mod$.
  By Lemma~\ref{lem:exact-algebra-indec-Morita-equiv-to-right-simple}, there exists a division algebra $A$ in $\mathcal{D}$ such that $\mathcal{M}^* \approx \mathcal{D}_A$ as left $\mathcal{D}$-module categories.
  In the Hopf-algebraic language, $A$ is a right $H$-simple left $H$-comodule algebra.
  Again by the duality, we have equivalences $\mathcal{M} \approx \mathcal{M}^{**} \approx {}_A\Mod$ of left $\mathcal{C}$-module categories.
\end{proof}

The original result \cite[Theorem 3.3]{MR2331768} assumes that the base field $\bfk$ is an algebraically closed field.
If this is the case, the comodule algebra $A$ in the statement has trivial coinvariants by Lemma~\ref{lem:division-alg-is-haploid}, as stated in \cite[Theorem 3.3]{MR2331768}.

By Theorem~\ref{thm:AM-exact-mod-cat}, the classification of indecomposable exact module categories over $\mathcal{C}$ reduces to the classification of right $H$-simple left $H$-comodule algebras up to $H$-Morita equivalence.

\subsection{Skryabin's theorems}
\label{subsec:Skryabin}

Skryabin gave several criteria for relative Hopf modules over a comodule algebra $A$ to be free or projective as $A$-modules \cite{MR2286047}. As these results are essential in this paper, we recall some of his results.
Let $H$ be a Hopf algebra, and let $A$ be an $H$-simple left $H$-comodule algebra (we note that $H$ and $A$ are finite-dimensional by our convention). Then we have:
\begin{enumerate}
\item For every object $M \in {}^H\Mod_A$, there exists an integer $r \ge 1$ such that $M^{\oplus r}$ is free as a right $A$-module. In particular, $M$ is projective as a right $A$-module \cite[Lemma 3.4 and Theorem 3.5]{MR2286047}.
\item The algebra $A$ is Frobenius \cite[Theorem 4.2]{MR2286047}. Hence, by (1), every object of ${}^H\Mod_A$ is also injective as a right $A$-module.
\item If there is a surjective algebra map from $A$ to a division algebra, then every object of ${}^H\Mod_A$ is free as a right $A$-module \cite[Theorem 3.5]{MR2286047}.
\end{enumerate}

A left coideal subalgebra is a subalgebra $L$ of $H$ such that $\Delta(L) \subset H \otimes L$. It is obvious that a left coideal subalgebra of $H$ is a left $H$-comodule algebra by the coaction given by the restriction of the comultiplication of $H$. Skryabin noted that a left coideal subalgebra of $H$ is $H$-simple in the proof of \cite[Theorem 6.1]{MR2286047}. Just a little bit more discussion shows:

\begin{lemma}
  Every left coideal subalgebra of $H$ is right $H$-simple.
\end{lemma}
\begin{proof}
  Let $A$ be a left coideal subalgebra of $H$.
  Then, by \cite[Theorem 6.1]{MR2286047}, every object of ${}^H\Mod_A$ is free over $A$.
  Hence, by considering the rank over $A$, we see that every subobject of $A \in {}^H\Mod_A$ is 0 or $A$. This means that $A$ is right $H$-simple.
\end{proof}

\subsection{Cocycle deformation}
\label{subsec:cocycle-deform}

Let $H$ be a Hopf algebra. A {\em Hopf 2-cocycle} of $H$ is a convolution-invertible linear map $\sigma : H \otimes H \to \bfk$ such that the equations
\begin{equation}
  \label{eq:Hopf-2-cocycle}
  \sigma(x_{(1)}, y_{(1)}) \sigma(x_{(2)} y_{(2)}, z)
  = \sigma(y_{(1)}, z_{(1)}) \sigma(x, y_{(2)} z_{(2)})
\end{equation}
and $\sigma(x, 1) = \varepsilon(x) = \sigma(1, x)$ hold for all elements $x, y, z \in H$ (where $\sigma$ is regarded as a bilinear form on $H$ in a natural way).
Given a 2-cocycle $\sigma$ of $H$, we define a new multiplication $*^{\sigma}$ on $H$ by
\begin{equation*}
  % \label{eq:Hopf-2-cocycle-deformation}
  x *^{\sigma} y = \sigma(x_{(1)}, y_{(1)}) x_{(2)} y_{(2)} \sigma^{-1}(x_{(3)}, y_{(3)})
  \quad (x, y \in H),
\end{equation*}
where $\sigma^{-1}$ is the inverse of $\sigma$ with respect to the convolution product. It is known that the coalgebra $H$ equipped with the new multiplication $*^{\sigma}$ becomes a Hopf algebra \cite[Theorem 1.6]{MR1213985}.
We denote the resulting Hopf algebra by $H^{\sigma}$ and call it {\em the 2-cocycle deformation} of $H$ by $\sigma$.

A left $H$-comodule $M$ is written as ${}_{\sigma}M$ when it is viewed as a left comodule over $H^{\sigma}$. The assignment $M \mapsto {}_{\sigma}M$ gives rise to an isomorphism ${}^H\Mod \to {}^{H^{\sigma}}\Mod$ of tensor categories together with the natural isomorphism given by
\begin{equation}
  \label{eq:Hopf-2-cocycle-tensor-equiv}
  {}_{\sigma}M \otimes {}_{\sigma}N \to {}_{\sigma}(M \otimes N),
  \quad m \otimes n \mapsto \sigma(m_{(-1)}, n_{(-1)}) m_{(0)} \otimes n_{(0)}
\end{equation}
for $M, N \in {}^H\Mod$, $m \in M$ and $n \in N$. Hence the assignment $A \mapsto {}_{\sigma}A$ gives an isomorphism between the category of left $H$-comodule algebras and the category of left $H^{\sigma}$-comodule algebras. Specifically, given a left $H$-comodule algebra $A$, the left $H^{\sigma}$-comodule algebra ${}_{\sigma}A$ is given explicitly as follows: As a left $H$-comodule, we have ${}_{\sigma}A = A$ (we note that $H = H^{\sigma}$ as coalgebras). The multiplication ${}_{\sigma}*$ of ${}_{\sigma}A$ is given by the composition
\begin{equation*}
  {}_{\sigma}*: {}_{\sigma}A \otimes {}_{\sigma}A
  \xrightarrow{\quad \eqref{eq:Hopf-2-cocycle-tensor-equiv} \quad}
  {}_{\sigma}(A \otimes A)
  \xrightarrow{\quad m \quad}
  {}_{\sigma}A,
\end{equation*}
where $m : A \otimes A \to A$ is the multiplication of $A$. Namely,
\begin{equation*}
  a \mathbin{{}_{\sigma}*} b = \sigma(a_{(-1)}, b_{(-1)}) a_{(0)} b_{(0)} \quad (a, b \in A).
\end{equation*}

The isomorphism ${}^H\Mod \cong {}^{H^{\sigma}}\Mod$ of tensor categories induces an equivalence
\begin{equation}
  \label{eq:Hopf-2-cocycle-Morita-equiv}
  \ullMod{{}^{H}\Mod} \approx \ullMod{{}^{H^{\sigma}}\Mod}
\end{equation}
of 2-categories. By the Morita duality, we also have an equivalence
\begin{equation}
  \label{eq:Hopf-2-cocycle-Morita-equiv-2}
  \ullMod{{}_{H}\Mod} \xrightarrow{\ (-)^* \ }
  \ullMod{{}^{H}\Mod}
  \xrightarrow{\ \eqref{eq:Hopf-2-cocycle-Morita-equiv} \ }
  \ullMod{{}^{H^{\sigma}}\Mod} \xrightarrow{\ (-)^* \ }
  \ullMod{{}_{H^{\sigma}}\Mod}
\end{equation}
of 2-categories (we omit the notation $(-)^{\op}$ for the opposite 2-category here). Now let $A$ be a left $H$-comodule algebra. It is easy to see that the 2-equivalence \eqref{eq:Hopf-2-cocycle-Morita-equiv} sends ${}^H\Mod_A$ to ${}^{H^{\sigma}}\Mod_{{}_{\sigma}A}$. Hence, by Lemma~\ref{lem:dual-of-A-Mod}, the 2-equivalence \eqref{eq:Hopf-2-cocycle-Morita-equiv-2} sends ${}_A\Mod$ to ${}_{{}_{\sigma}A}\Mod$. The discussion is schematized as follows:
\begin{equation*}
  \begin{tikzcd}
    {}_A\Mod \arrow[d, mapsto]
    & \ullMod{{}_{H}\Mod} \arrow[rr, "{(-)^*}"]
    \arrow[d, "{\eqref{eq:Hopf-2-cocycle-Morita-equiv-2}}"']
    \arrow[rrd, phantom, "{\circlearrowright}"]
    & & \ullMod{{}^{H}\Mod}
    \arrow[d, "{\eqref{eq:Hopf-2-cocycle-Morita-equiv}}"]
    & {}^H\Mod_A \arrow[d, mapsto] \\
    {}_{{}_{\sigma}A}\Mod
    & \ullMod{{}_{H^{\sigma}}\Mod} \arrow[rr, "{(-)^*}"']
    & & \ullMod{{}^{H^{\sigma}}\Mod}
    & {}^{H^{\sigma}}\Mod_{{}_{\sigma}A}
  \end{tikzcd}
\end{equation*}

We also note that the isomorphism ${}^H\Mod \cong {}^{H^{\sigma}}\Mod$ of tensor categories induces an isomorphism ${}^H\Mod_A \cong {}^{H^{\sigma}}\Mod_{{}_{\sigma}A}$ of categories \cite[Lemma 2.1]{MR2678630}.
Lemmas~\ref{lem:properties-preserved-by-cocycle-deform} and \ref{lem:cocycle-deform-H-Morita} below are obvious from the above observation:

\begin{lemma}
  \label{lem:properties-preserved-by-cocycle-deform}
  Let $P(H, A)$ be one of the following propositions:
  \begin{enumerate}
  \item $A$ is an $H$-indecomposable left $H$-comodule algebra.
  \item $A$ is a right $H$-simple left $H$-comodule algebra.
  \item $A$ is an exact left $H$-comodule algebra.
  \end{enumerate}
  For a Hopf algebra $H$, a left $H$-comodule algebra $A$ and a Hopf 2-cocycle $\sigma$ of $H$, the proposition $P(H, A)$ is equivalent to $P(H^{\sigma}, {}_{\sigma}A)$.
\end{lemma}

\begin{lemma}
  \label{lem:cocycle-deform-H-Morita}
  Two left $H$-comodule algebras $A$ and $B$ are $H$-Morita equivalent if and only if ${}_{\sigma}A$ and ${}_{\sigma}B$ are $H^{\sigma}$-Morita equivalent.
\end{lemma}

Although it is not directly related to the main purpose of this paper, the following consequence of Lemma~\ref{lem:properties-preserved-by-cocycle-deform} is noteworthy:

\begin{theorem}
  \label{thm:cocycle-deform-semisimplicity}
  Let $\sigma$ be a Hopf 2-cocycle of $H$, and let $A$ be a left $H$-comodule algebra. We assume that both $H$ and $H^{\sigma}$ are semisimple. Then ${}_{\sigma}A$ is semisimple if and only if $A$ is.
\end{theorem}
\begin{proof} % [Proof of Theorem~\ref{thm:cocycle-deform-semisimplicity}]
  In general, a finite module category $\mathcal{M}$ over a semisimple finite tensor category is exact if and only if $\mathcal{M}$ is semisimple. Thus we have
  \begin{gather*}
    \text{${}_{\sigma}A$ is a semisimple algebra} \iff
    \text{${}_{\sigma}A$ is an exact left $H^{\sigma}$-comodule algebra} \\
    \iff \text{$A$ is an exact left $H$-comodule algebra}
    \iff \text{$A$ is a semisimple algebra},
  \end{gather*}
  where we have used Lemma~\ref{lem:properties-preserved-by-cocycle-deform} at the second equivalence.
\end{proof}

We note some sufficient conditions for $H$ ensuring that $H^{\sigma}$ is semisimple. If $\Bbbk$ has characteristic zero, then $H$ is semisimple if and only if $H^{\sigma}$ is by the Larson-Radford theorem \cite{MR957441}. In arbitrary characteristic, it is known that the class of finite-dimensional semisimple Hopf algebras $H$ with $H^*$ unimodular is closed under 2-cocycle deformation (the dual of \cite[Corollary 3.6]{AEGN02}). Since a semisimple Hopf algebra is unimodular, it implies that $H$ is semisimple and cosemisimple if and only if $H^{\sigma}$ is.

Theorem~\ref{thm:cocycle-deform-semisimplicity} fails without the semisimplicity of $H$; see Remark~\ref{rem:A4-semisimplicity}.

\subsection{Examples: Group algebras}

At the end of this section, we recall from \cite{MR3242743} how right $H$-simple left $H$-comodule algebras are given in the case where $H$ is the group Hopf algebra of a finite group.

Let $G$ be a group. For a non-negative integer $n$, we define $\mathrm{C}^n(G)$ to be the set of all maps $f$ from $G^n$ to $\bfk^{\times}$ such that $f(x_1, \cdots, x_n) = 1$ whenever one of the $x_i$ is the identity element of $G$, and make it an abelian group by the pointwise multiplication. For $f_i \in \mathrm{C}^i(G)$ ($i = 1, 2$), we define $\partial_i(f_i) \in \mathrm{C}^{i+1}(G)$ by $\partial_1(f_1)(a, b) = f_1(a) f_1(a b)^{-1} f_1(b)$ and
\begin{equation*}
  \partial_2(f_2)(a, b, c) = f_2(a, b) f_2(a, b c)^{-1} f_2(a b, c) f_2(b, c)^{-1}
\end{equation*}
for $a, b, c \in G$. The maps $\partial_1$ and $\partial_2$ are group homomorphisms. We set
\begin{equation*}
  \mathrm{Z}^2(G) = \Ker(\partial_2), \quad
  \mathrm{B}^2(G) = \Img(\partial_1) \quad \text{and} \quad
  \mathrm{H}^2(G) = \mathrm{Z}^2(G)/\mathrm{B}^2(G).
\end{equation*}

The group $\mathrm{H}^2(G)$ is nothing but the second cohomology group of $G$ with coefficients in $\bfk^{\times}$, where $G$ acts on $\bfk^{\times}$ trivially. The group $\mathrm{Z}^2(G)$ is identified with the set of Hopf 2-cocycles of the group Hopf algebra $\bfk G$. The left $\bfk G$-comodule algebra ${}_{\psi}\bfk G$ for $\psi \in \mathrm{Z}^2(G)$ is called the {\em twisted group algebra} of $G$ by $\psi$.
It is easy to see that the isomorphism class of the comodule algebra ${}_{\psi}\bfk G$ depends on the class of $\psi$ in the cohomology group $\mathrm{H}^2(G)$.

\begin{theorem}[{\cite[Corollary 7.12.20]{MR3242743}}]
  \label{thm:exact-H-comod-over-group-alg}
  Suppose that $\bfk$ is algebraically closed.
  Let $G$ be a finite group, and let $H = \bfk G$ be the group Hopf algebra.
  Then every right $H$-simple left $H$-comodule algebra is isomorphic to ${}_{\psi}\bfk F$ for some subgroup $F$ of $G$ and a 2-cocycle $\psi \in \mathrm{Z}^2(F)$.
\end{theorem}

\section{Exact comodule algebras over pointed Hopf algebras}
\label{sec:exact-comod-alg-over-pointed}

\subsection{Filtered and graded vector spaces}

Mombelli \cite{MR2678630} gave useful results on exact comodule algebras over pointed Hopf algebras. In this section, we give a proof of some of his results in a slightly generalized form and then introduce a strategy for classifying right $H$-simple left $H$-comodule algebras up to $H$-Morita equivalence in the case where $H$ is a finite-dimensional pointed Hopf algebra whose group of grouplike elements satisfies some technical assumptions. For this purpose, we begin by reviewing basic results on filtered or graded Hopf algebras and comodule algebras over them.

Unless otherwise noted, filtrations and gradings on vector spaces are assumed to be indexed by the monoid $\mathbb{Z}_{\ge 0}$ of non-negative integers.
Thus a {\em filtered vector space} is a vector space $V$ equipped with a family $\{ V_{[i]} \}_{i \ge 0}$ of subspaces of $V$ such that $V_{[0]} \subset V_{[1]} \subset \cdots$ and $V = \bigcup_{i \ge 0} V_{[i]}$. We denote by $\FiltVect$ the category of filtered vector spaces and linear maps respecting the filtrations. If $V$ and $W$ are filtered vector spaces, then their tensor product $V \otimes W$ is also filtered by $(V \otimes W)_{[n]} = \sum_{i = 0}^n V_{[i]} \otimes W_{[n-i]}$. The category $\FiltVect$ is a symmetric monoidal category with respect to this operation. A filtered algebra, a filtered coalgebra and a filtered Hopf algebra are an algebra, a coalgebra and a Hopf algebra in $\FiltVect$, respectively.

Graded vector spaces and linear maps respecting the grading also form a symmetric monoidal category, which we denote by $\GrVect$. A graded algebra, a graded coalgebra and a graded Hopf algebra are an algebra, a coalgebra and a Hopf algebra in $\GrVect$, respectively. For a filtered vector space $V$, the associated graded vector space $\gr(V)$ is defined by $\gr(V) = \bigoplus_{i \ge 0} \gr_i(V)$, where $\gr_i(V) = V_{[i]}/V_{[i-1]}$ ($i \in \mathbb{Z}_{\ge 0}$) with the convention $V_{[-1]} = 0$.
The assignment $V \mapsto \gr(V)$ extends to a $\bfk$-linear functor from $\FiltVect$ to $\GrVect$.

Now let $V, W \in \FiltVect$. For $i, j \in \mathbb{Z}_{\ge 0}$, there is a well-defined linear map
\begin{equation*}
  \phi_{i j} : \gr_i(V) \otimes \gr_j(W) \to \gr_{i + j}(V \otimes W)
\end{equation*}
given by $\phi_{i j}((v + V_{[i-1]}) \otimes (w + W_{[j-1]})) = v \otimes w + (V \otimes W)_{[i+j-1]}$ for $v \in V_{[i]}$ and $w \in W_{[j]}$. By combining the linear maps $\phi_{i j}$, we obtain a linear map
\begin{equation*}
  \gr(V) \otimes \gr(W) = \bigoplus_{i, j = 0}^{\infty} \gr_i(V) \otimes \gr_j(W)
  \to\bigoplus_{k = 0}^{\infty} \gr_k(V \otimes W) = \gr(V \otimes W),
\end{equation*}
which is in fact an isomorphism natural in $V, W \in \FiltVect$. Moreover, this makes the functor $\gr : \FiltVect \to \GrVect$ a symmetric monoidal functor. An important consequence of this observation is that the functor $\gr$ makes a filtered X into a graded X, where X is an algebra, a coalgebra, a Hopf algebra, or other algebraic object defined by commutative diagrams in a symmetric monoidal category.

\subsection{Filtered and graded comodule algebras}

Let $C$ be a filtered coalgebra with filtration $\{ C_{[i]} \}_{i \ge 0}$. A {\em filtered left $C$-comodule} is a filtered vector space endowed with a structure of a left $C$-comodule such that the coaction respects the filtration. We note that any (non-filtered) $C$-comodule admits a canonical filtration:

\begin{lemma}[{\it cf}. {\cite[Lemma 4.1]{MR2678630}}]
  \label{lem:comodule-filtration}
  For a left $C$-comodule $M$, we define
  \begin{equation}
    \label{eq:induced-comodule-filtration}
    M_{[n]} = \delta_M^{-1}(C_{[n]} \otimes M)
    \quad (n \in \mathbb{Z}_{\ge 0}),
  \end{equation}
  where $\delta_M : M \to C \otimes M$ is the coaction. Then the filtration $\{ M_{[n]} \}_{n \ge 0}$ makes $M$ a filtered left $C$-comodule.
\end{lemma}
\begin{proof}
  It is obvious that $M = \bigcup_{n \ge 0} M_{[n]}$.
  We fix elements $n \in \mathbb{Z}_{\ge 0}$ and $m \in M_{[n]}$ with $m \ne 0$, and aim to show $\delta_M(m) \in \sum_{a + b = n} C_{[a]} \otimes M_{[b]}$.
  For that, we fix a basis $\{ x_{\lambda} \}_{\lambda \in \Lambda}$ of $C_{[n]}$ with the property that, for each integer $a$ with $0 \le a < n$, there is a subset $\Lambda_a$ of $\Lambda$ such that $\{ x_{\lambda} \}_{\lambda \in \Lambda_a}$ is a basis of $C_{[a]}$.
  There are linear maps $x_{\lambda}^* : C \to \bfk$ such that $\langle x_{\lambda}^*, x_{\mu} \rangle = \delta_{\lambda, \mu}$ for all $\lambda, \mu \in \Lambda$, where $\delta$ is the Kronecker delta.
  Since $\delta_M(m) \in C_{[n]} \otimes M$, we can write it as $\delta_M(m) = \sum_{j = 1}^r x_{\lambda_j} \otimes m_j$ for some elements $\lambda_j \in \Lambda$ and $m_j \in M$ ($j = 1, \cdots, r$) with the $\lambda_j$ distinct. Now let $n_j$ be the smallest non-negative integer satisfying $x_{\lambda_j} \in C_{[n_j]}$. Then $x_{\lambda_j}^*$ vanishes on $C_{[a]}$ for $a < n_j$ since $C_{[a]}$ is spanned by $x_{\lambda}$ ($\lambda \in \Lambda_a$) and $\lambda_j \not \in \Lambda_a$. Since $n_{j} \le n$, we have
  \begin{equation*}
    \langle x_{\lambda_i}^*, (x_{\lambda_{j}})_{(1)} \rangle (x_{\lambda_{j}})_{(2)}
    \in \sum_{a = 0}^{n_j} x_{\lambda_i}^*(C_{[a]}) C_{[n_{j} - a]}
    \subset \sum_{a \ge n_i} C_{[n_{j} - a]} \subset C_{[n-n_{i}]}
  \end{equation*}
  with the convention $C_{[r]} = 0$ for $r < 0$.
  By the coassociativity of $\delta_M$, we have
  \begin{align*}
    \delta_M(m_i)
    & = (x_{\lambda_i}^* \otimes \delta_M) \delta_M(m)
      = (x_{\lambda_i}^* \otimes \id_C \otimes \id_M) (\Delta \otimes \id_M) \delta_M(m) \\
    & = \sum_{j = 1}^r \langle x_{\lambda_i}^*, (x_{\lambda_{j}})_{(1)} \rangle
      (x_{\lambda_{j}})_{(2)} \otimes m_j \in C_{[n - n_{i}]} \otimes M.
  \end{align*}
  This means $m_i \in M_{[n - n_i]}$. Hence,
  \begin{equation*}
    \delta_M(m)
    = \sum_{i = 1}^r x_{\lambda_i} \otimes m_i
    \in \sum_{i = 1}^r C_{[n_i]} \otimes M_{[n-n_i]}
    \subset \sum_{a + b = n} C_{[a]} \otimes M_{[b]}.
    \qedhere
  \end{equation*}
\end{proof}

Let $H$ be a filtered Hopf algebra. A {\em filtered left $H$-comodule algebra} is a filtered algebra $A$ endowed with a structure of a filtered left $H$-comodule such that the coaction $\delta_A : A \to H \otimes A$ is a morphism of filtered algebras.

\begin{lemma}[{\it cf}. {\cite[Lemma 4.1]{MR2678630}}]
  Let $H$ be a filtered Hopf algebra with filtration $\{ H_{[n]} \}_{n \ge 0}$, and let $A$ be a left $H$-comodule algebra. Then $A$ is a filtered left $H$-comodule algebra in the above sense by the filtration \eqref{eq:induced-comodule-filtration} with $M = A$.
\end{lemma}
\begin{proof}
  It is easy to verify $A_{[i]} \cdot A_{[j]} \subset A_{[i+j]}$ for all $i, j \in \mathbb{Z}_{\ge 0}$ by the assumption that $\delta_A$ is an algebra map. Lemma~\ref{lem:comodule-filtration} completes the proof.
\end{proof}

The definition of a {\em graded left comodule algebra} over a graded Hopf algebra is obtained by replacing the word `filtered' with `graded' in the definition of filtered left comodule algebra. If $A$ is a filtered left comodule algebra over a filtered Hopf algebra $H$, then $\gr(A)$ is a graded left comodule algebra over $\gr(H)$.

\subsection{The coradical and the Loewy filtration}

For a coalgebra $C$, the {\em coradical} of $C$ is defined to be the sum of all simple subcoalgebras of $C$. The {\em coradical filtration} of $C$ is defined inductively by $C_n = \Delta^{-1}(C_0 \otimes C + C \otimes C_{n-1})$ for $n \ge 1$, where $C_0$ is the coradical of $C$. It is well-known that the coradical filtration makes $C$ a filtered coalgebra \cite[Chapter 5]{MR1243637}. We denote by $\grc(C)$ the associated graded coalgebra of $C$ with respect to the coradical filtration.

There is a closely related filtration on a comodule. Given a left $C$-comodule $M$, we define subspaces $M_0 \subset M_1 \subset M_2 \subset \cdots$ by $M_0 = \socle(M)$ and $M_{n}/M_{n-1} = \socle(M/M_{n-1})$ for $n \ge 1$, where $\socle(-)$ denotes the socle, that is, the sum of simple subcomodules.
We call $\{ M_n \}_{n \ge 0}$ the {\em Loewy filtration} on $M$ and denote by $\grL(M)$ the associated graded vector space.
Following \cite[Section 1]{MR1953053}, we have
\begin{equation*}
  M_n = \delta_M^{-1}(C_n \otimes M)
  \quad \text{and} \quad
  \delta_M(M_n) \subset \sum_{i = 0}^n C_{n-i} \otimes M_i
\end{equation*}
for all $n \in \mathbb{Z}_{\ge 0}$. The latter means that $M$ is a filtered left $C$-comodule with respect to the Loewy filtration. Since the functor $\gr$ is symmetric monoidal, $\grL(M)$ is a left comodule over $\grc(C)$, as has been noted in \cite[Section 1]{MR1953053}.

The operations $\grc$ and $\grL$ are idempotent: For a coalgebra $C$, the coradical filtration of $\grc(C)$ is identical to the natural filtration associated to the grading of $\grc(C)$ and thus $\grc(\grc(C)) \cong \grc(C)$ as graded coalgebras. A similar statement holds for comodules.

A graded coalgebra $C = \bigoplus_{n = 0}^{\infty} C(n)$ is said to be {\em coradically graded} if the $n$-th layer of its coradical filtration $\{ C_n \}_{n \ge 0}$ is given by $C_n = C(0) \oplus \dotsb \oplus C(n)$ for all $n$. We assume that $C$ is coradically graded. Let $M$ be a left $C$-comodule with Loewy filtration $\{ M_n \}_{n \ge 0}$, and set $M(n) = M_n/M_{n-1}$ for $n \ge 0$. Then the associated graded vector space $\grL(M) = \bigoplus_{n \ge 0} M(n)$ is also a left $C$-comodule since $\grc(C) \cong C$.
Now let $f, g \in \Grp(C)$, $x \in C(1)$, $v_f \in M(0)$ and $m \in M(1)$.
We assume that they satisfy the following equations:
\begin{equation}
  \label{eq:lifting-0-assumptions}
  \Delta(x) = x \otimes f + g \otimes x,
  \quad \delta_{\gr}(v_f) = f \otimes v_f,
  \quad \delta_{\gr}(m) = x \otimes v_f + g \otimes m,
\end{equation}
where $\delta_{\gr}$ is the coaction of $\grL(M)$. As pointed out in the proof of \cite[Lemma 5.5]{MR2678630}, the element $m$ can be lifted to $M$ in the following sense:

\begin{lemma}
  \label{lem:lifting-0}
  Under the above assumption, there exists an element $\tilde{m} \in M_1$ such that $\delta_M(\tilde{m}) = x \otimes v_f + g \otimes \tilde{m}$ and $\pi_M(\tilde{m}) = m$, where $\pi_M : M_1 \to M(1)$ is the projection.
\end{lemma}
\begin{proof} 
  Let $\pi_C : C \to C(0)$ denote the projection. Since $C$ is a graded coalgebra, $\pi_C$ is a coalgebra map. If we regard $M_1$ as a left $C(0)$-comodule by $\pi_C$, then $M_0$ is a subcomodule of $M_1$ and $\pi_M : M_1 \to M(1)$ is $C(0)$-colinear. Since $C(0)$ is cosemisimple, $\pi_M$ has a $C(0)$-colinear section, say $s : M_1/M_0 \to M_1$.

  We show that $\tilde{m} := s(m)$ meets the requirements.
  Since $\grL(M)$ is a graded comodule, equation \eqref{eq:lifting-0-assumptions} implies that $x \otimes v_f \in C(1) \otimes M_0$. Since the section $s$ is $C(0)$-colinear, we also have
  \begin{equation}
    \label{eq:lem-lifting-0-proof-eq-1}
    (\pi_C \otimes \id_M)\delta_M(\tilde{m})
    = (\pi_C \otimes s)\delta_{\gr}(m)
    = g \otimes m.
  \end{equation}
  We recall that $\delta_{\gr}|_{M(1)}$ is induced by the composition of $\delta_M|_{M_1}$ (whose image is contained in $C_1 \otimes M_0 + C_0 \otimes M_1$) and the canonical linear map
  \begin{equation}
    \label{eq:lem-lifting-0-proof-eq-2}
    C_1 \otimes M_0 + C_0 \otimes M_1
    \twoheadrightarrow \frac{C_1 \otimes M_0 + C_0 \otimes M_1}{C_0 \otimes M_0}
    \cong \!\!\! \begin{array}{l} (C(1) \otimes M(0)) \\ \quad \mbox{} \oplus (C(0) \otimes M(1)). \end{array}
  \end{equation}
  Now we set $r := \delta_M(\tilde{m}) - (x \otimes v_f + g \otimes \tilde{m})$. By the definition of $\delta_{\gr}|_{M(1)}$, the element $r$ belongs to the kernel of \eqref{eq:lem-lifting-0-proof-eq-2}, that is, $C_0 \otimes M_0$. Since $\pi_C$ is identity on $C(0)$, we have
  \begin{equation*}
    r = (\pi_C \otimes \id_M)(r)
    = (\pi_C \otimes \id_M)\delta_M(\tilde{m}) - (\pi_C(x) \otimes v_f + \pi_C(g) \otimes \tilde{m}) = 0
  \end{equation*}
  by \eqref{eq:lem-lifting-0-proof-eq-1} and $\pi_C(x) = 0$. The proof is done.
\end{proof}

\subsection{$H$-simplicity of filtered comodule algebras}

For the rest of this section, we assume that all algebras are finite-dimensional, as in the previous section. We now give the following useful criterion for the $H$-simplicity of a filtered comodule algebra:

\begin{theorem}
  \label{thm:H-simplicity-comod-alg}
  Let $H$ be a Hopf algebra with filtration $\{ H_{[i]} \}_{i \ge 0}$, and let $A$ be a non-zero filtered left $H$-comodule algebra with the filtration given by \eqref{eq:induced-comodule-filtration} with $M = A$. Then the following are equivalent:
  \begin{enumerate}
  \item $A_{[0]}$ is right $H_{[0]}$-simple.
  \item $A$ is right $H$-simple.
  \item $\mathrm{gr}(A)$ is right $\mathrm{gr}(H)$-simple.
  \end{enumerate}
\end{theorem}

The filtration on $H$ is not necessarily the coradical filtration.
Mombelli proved this theorem under the assumption that $H_{[0]}$ is semisimple and $\bfk$ is an algebraically closed field of characteristic zero \cite[Proposition 4.4 and Corollary 4.5]{MR2678630}. Linchenko's result \cite[Theorem 3.1]{MR1995055} is crucial in the proof given in \cite{MR2678630}. In this paper, we extend Mombelli's result by using a refinement of \cite[Theorem 3.1]{MR1995055} due to Skryabin \cite{MR2832259}.

Before we give a proof of this theorem, we explain Skryabin's result that we will use. Let $H$ be a Hopf algebra, and let $A$ be a left $H$-comodule algebra. We note that $A$ can be viewed as a left module algebra over $U := H^{*\op}$. We define $\Jac^H(A)$ to be the sum of all $H$-costable ideals of $A$ contained in the Jacobson radical $\Jac(A)$ of $A$. Since $\Jac(A)$ is the largest nilpotent ideal of $A$, $\Jac^H(A)$ is the largest $H$-costable nilpotent ideal of $A$. Hence $A$ is $U$-semiprime (in the sense used in \cite{MR2832259}) if and only if $\Jac^H(A) = 0$. Applying Skryabin's result \cite[Theorem 1.1]{MR2832259} on $U$-semiprime $U$-module algebras to $A$, we have that $A$ is a finite product of $H$-simple algebras if and only if $\Jac^H(A) = 0$.

\begin{proof}[Proof of Theorem~\ref{thm:H-simplicity-comod-alg}]
  We first show that (1) implies (2).
  We assume that (1) holds and let $J$ be a non-zero $H$-costable right ideal of $A$.
  Since the coradical of $H$ is contained in $H_{[0]}$ \cite[Lemma 5.3.4]{MR1243637}, $\mathrm{soc}(A)$ is contained in $A_{[0]}$ (where $\mathrm{soc}(-)$ means the socle as a left $H$-comodule). Hence we have
  \begin{equation*}
    0 \ne \mathrm{soc}(J) = J \cap \mathrm{soc}(A) \subset J \cap A_{[0]}
  \end{equation*}
  and, in particular, $J \cap A_{[0]} \ne 0$.
  Since $J \cap A_{[0]}$ is an $H_{[0]}$-costable right ideal of $A_{[0]}$, we have $J \cap A_{[0]} = A_{[0]}$ by the assumption. Hence $1 \in J$. Therefore $J = A$.

  Next, we show that (2) implies (1).
  We assume that $A$ is right $H$-simple.
  Then the $H_{[0]}$-Jacobson radical $J := \Jac^{H_{[0]}}(A_{[0]})$ is zero.
  Indeed, since $J$ is $H_{[0]}$-costable, $J A$ is an $H$-costable right ideal of $A$.
  By the assumption that $A$ is right $H$-simple, $J A$ is either $0$ or $A$.
  If the latter is the case, then the Nakayama lemma implies $A = 0$, which is a contradiction.
  Thus $J A = 0$. Therefore $J = 0$.

  By Skryabin's result \cite[Theorem 1.1]{MR2832259} mentioned above, $A_{[0]}$ is decomposed into the finite product of $H_{[0]}$-simple algebras, as $A_{[0]} = R_1 \times \dotsb \times R_m$. Then $A$ is decomposed as the direct sum of non-zero $H$-costable right ideals, as $A = R_1 A \oplus \dotsb \oplus R_m A$. Since $A$ is right $H$-simple, we have $m = 1$, that is, $A_{[0]}$ is $H_{[0]}$-simple.

  We now prove that $A_{[0]}$ is in fact right $H_{[0]}$-simple.
  Since $A_{[0]}$ is $H_{[0]}$-simple as we have just proved, $A_{[0]}$ is also $H$-simple. Hence, by Skryabin's result mentioned in Subsection~\ref{subsec:Skryabin}, $A_{[0]}$ is injective as a left $A_{[0]}$-module.
  Hence the inclusion map $A_{[0]} \hookrightarrow A$ splits as a morphism of left $A_{[0]}$-modules. In other words, $A$ has a left $A_{[0]}$-submodule $B$ such that $A = A_{[0]} \oplus B$.
  The rest of the proof is basically the same as the argument of Mombelli \cite[Section 4]{MR2678630}.
  Let $I$ be a non-zero $H_{[0]}$-costable right ideal of $A_{[0]}$. Since $I$ is a right ideal of $A_{[0]}$, we have $I A_{[0]} = I$. Since $B$ is a left $A_{[0]}$-submodule, we have $I B \subset B$. Since $A$ is right $H$-simple, we have $I A = A$. Hence we have $A_{[0]} \oplus B = A = I A = I A_{[0]} \oplus I B \subset I \oplus B$, which yields $A_{[0]} = I$. Therefore $A_{[0]}$ is right $H_{[0]}$-simple.

  We have proved (1) $\Leftrightarrow$ (2).
  Since the 0-th components of $\mathrm{gr}(A)$ and $\mathrm{gr}(H)$ are $A_{[0]}$ and $H_{[0]}$, respectively, the equivalence (1) $\Leftrightarrow$ (3) follows from (1) $\Leftrightarrow$ (2). The proof is done.
\end{proof}

\subsection{Strategy of the classification}
\label{subsec:classification-strategy}

Mombelli \cite{MR2678630,MR3264686} and Mombelli and Garc\'ia Iglesias \cite{MR2860429} proposed methods for classifying right $H$-simple left $H$-comodule algebras for a pointed Hopf algebra $H$. Here we give a modification of their strategy to give a complete list of indecomposable exact module categories over $\lmod{u_q(\mathfrak{sl}_2)}$ in Section~\ref{sec:exact-comod-alg-uqsl2}.

We first recall an observation related to Mombelli's results \cite[Proposition 5.3 and Lemma 5.5]{MR3264686}, which plays a central role in our classification strategy of exact comodule algebras. More precisely, his results show that, under suitable hypotheses, a graded $H$-simple left $H$-comodule algebra can be realized as a homogeneous left coideal subalgebra of $H$. An essential part of his criterion comes from the following observation:

\begin{lemma}
  \label{lem:coideal-sub-criterion}
  Let $H$ be a Hopf algebra, and let $L$ be a left $H$-comodule algebra.
  Suppose that $L$ is $H$-simple and there is an algebra map $\alpha : L \to \bfk$. Then
  \begin{equation*}
    \phi_{\alpha} := (\id_H \otimes \alpha) \circ \delta_L : L \to H
  \end{equation*}
  is an injective homomorphism of left $H$-comodule algebras, and thus $L$ can be regarded as a left coideal subalgebra of $H$.
\end{lemma}
\begin{proof}
  We have $\Delta \phi_{\alpha}(a)
  = a_{(-2)} \otimes a_{(-1)} \alpha(a_{(0)})
  = (\id_H \otimes \phi_{\alpha}) \delta_L(a)$ for all $a \in L$, which means that $\phi_{\alpha}$ is $H$-colinear.
  Since $\phi_{\alpha}$ is written as a composition of algebra maps, $\phi_{\alpha}$ is an algebra map.
  We have proved that $\phi_{\alpha}$ is a homomorphism of left $H$-comodule algebras.
  This implies that $\Ker(\phi_{\alpha})$ is an $H$-costable ideal of $L$ such that $\Ker(\phi_{\alpha}) \subsetneq L$.
  By the assumption that $L$ is $H$-simple, $\Ker(\phi_{\alpha})$ is zero, that is, $\phi_{\alpha}$ is injective. The proof is done.
\end{proof}

The above lemma has a graded variant:

\begin{lemma}[{\it cf.} {\cite[Lemma 5.5]{MR3264686}}]
  \label{lem:coideal-sub-criterion-graded}
  Let $H = \bigoplus_{n \ge 0} H(n)$ be a graded Hopf algebra, and let $L = \bigoplus_{n \ge 0} L(n)$ be a graded left $H$-comodule algebra. If $L$ is $H$-simple and $L(0)$ has a one-dimensional representation, then there is an injective homomorphism $L \to H$ of graded left $H$-comodule algebras, and hence $L$ can be regarded as a homogeneous left coideal subalgebra of $H$.
\end{lemma}
\begin{proof}
  We choose an algebra map $\alpha' : L(0) \to \bfk$ and define $\alpha : L \to \bfk$ to be the composition of $\alpha'$ and the projection $L \to L(0)$.
  Then we have an injective homomorphism $\phi_{\alpha} : L \to H$ of left $H$-comodule algebras by Lemma~\ref{lem:coideal-sub-criterion}.
  It is obvious from the definition that $\phi_{\alpha}$ preserves the grading. The proof is done.
\end{proof}

Now let $U$ be a pointed Hopf algebra and consider:

\begin{assumption}
  \label{assumptions-for-classification-strategy}
  The following three conditions are satisfied:
  \begin{enumerate}
  \item \label{item:assumption-base-field}
    The base field $\Bbbk$ is algebraically closed.
  \item \label{item:assumption-cocycle-deform}
    There is a Hopf 2-cocycle $\sigma$ on $H := \grc(U)$ such that $U \cong H^{\sigma}$.
  \item \label{item:assumption-cocycle-vanishing}
    The cohomology group $\mathrm{H}^2(F)$ vanishes for all subgroups $F < \Grp(U)$.
  \end{enumerate}
\end{assumption}

The assumption (\ref{item:assumption-cocycle-deform}) has been verified for many pointed Hopf algebras (see, {\it e.g.}, \cite{MR3133699} and references therein). The assumption (\ref{item:assumption-cocycle-vanishing}) is restrictive but imposed to avoid the technical difficulty of dealing with twisted group algebras.
Although we will not mention them in this paper, some ideas for remedying the assumption (\ref{item:assumption-cocycle-vanishing}) were discussed in \cite{MR2989520,MR3264686,MR2860429}.

Now we assume that Assumption \ref{assumptions-for-classification-strategy} holds.
If $\mathcal{A}$ is a right $U$-simple left $U$-comodule algebra, then $\mathcal{A} = {}_{\sigma}\mathcal{L}$ for some right $H$-simple left $H$-comodule algebra $\mathcal{L}$ by Lemma~\ref{lem:properties-preserved-by-cocycle-deform}, where $H = \grc(U)$. We note that the 0-th component of $H$ is the group algebra of $\Grp(U)$. By Theorems~\ref{thm:exact-H-comod-over-group-alg} and~\ref{thm:H-simplicity-comod-alg}, and Assumption~\ref{assumptions-for-classification-strategy}, we find that the 0-th layer of the Loewy filtration of $\mathcal{L}$ is isomorphic to $\bfk F$ for some $F < \Grp(U)$. Hence, by Lemma~\ref{lem:coideal-sub-criterion-graded}, $\grL(\mathcal{L})$ is a homogeneous left coideal subalgebra of $H$. In view of the above discussion, we now have the following strategy to classify right $U$-simple left $U$-comodule algebras up to $U$-Morita equivalence:
\begin{description}
  \setlength{\itemsep}{5pt}
\item[Step 1] Classify all graded left coideal subalgebras of $H$.
  By Skryabin's result (see Subsection~\ref{subsec:Skryabin}), all of them are right $H$-simple.
\item[Step 2] For each coideal subalgebra $L$ obtained in Step 1, we view it as a left $\grc(H)$-comodule algebra via the canonical isomorphism $\grc(H) \cong H$, and then classify all left $H$-comodule algebras $\mathcal{L}$ such that $\grL(\mathcal{L}) \cong L$ as graded $\grc(H)$-comodule algebras. By Theorem~\ref{thm:H-simplicity-comod-alg}, all of them are right $H$-simple.
\item[Step 3] Now every indecomposable exact $H$-comodule algebra is $H$-Morita equivalent to one of the comodule algebras given in Step 2. Determine whether they are $H$-Morita equivalent.
\item[Step 4] For each representative $\mathcal{L}$ of an $H$-Morita equivalence class of left $H$-comodule algebras obtained in Step 2, compute ${}_{\sigma}\mathcal{L}$.
  By Lemma~\ref{lem:properties-preserved-by-cocycle-deform}, all of them are right $U$-simple.
  By Lemma~\ref{lem:cocycle-deform-H-Morita}, no two distinct representatives are $U$-Morita equivalent.
\end{description}

In Section \ref{sec:exact-comod-alg-uqsl2}, we will apply the above strategy for $U = u_q(\mathfrak{sl}_2)$.
Even Step 1 is non-trivial in general; however, a classification of coideal subalgebras of $\grc(u_q(\mathfrak{sl}_2))$ is already known \cite{2024arXiv241010064S}.
In our experience, it is hard to find left $U$-comodule algebras $\mathcal{L}$ such that $\grL(\mathcal{L}) \cong L$ for some graded left coideal subalgebras $L$ of $H$.
This is the reason why we obtain left $U$-comodule algebras in two steps, Steps 2 and 4.

For Step 2, we will use Lemma~\ref{lem:lifting-skew-primitive-type-element} below. The setting is as follows: Let $H$ be a coradically graded pointed Hopf algebra, and let $A$ be a left $H$-comodule algebra such that $A_0 = \bfk F$ for some subgroup $F$ of $\Grp(H)$. Let $g \in \Grp(H)$ be a grouplike element, let $x \in H$ and $y \in A_1/A_0$ be non-zero elements, and let $\chi : F \to \bfk^{\times}$ be a group homomorphism. We assume that $|F|$ is non-zero in $\bfk$, the element $g$ commutes with all elements of $F$, and the following equations are satisfied:
\begin{gather*}
  \Delta(x) = x \otimes 1 + g \otimes x,
  \quad \delta_{\grL(A)}(y) = x \otimes 1 + g \otimes y, \\
  a x = \chi(a) x a,
  \quad a y = \chi(a) y a
  \quad (a \in F).
\end{gather*}

\begin{lemma}[{\cite[Lemma 5.5]{MR2678630}}]
  \label{lem:lifting-skew-primitive-type-element}
  Under the above assumptions, there exists an element $\tilde{y} \in A_1$ such that
  \begin{equation}
    \label{eq:lifting-skew-primitive-type-element}
    \delta_A(\tilde{y}) = x \otimes 1 + g \otimes \tilde{y},
    \quad a \tilde{y} = \chi(a) \tilde{y} a \quad (a \in F),
    \quad \pi(\tilde{y}) = y,
  \end{equation}
  where $\pi : A_1 \to A_1 / A_0$ is the projection.
\end{lemma}
\begin{proof}
  We provide a detailed proof, as the one given in \cite{MR2678630} appears to require some additional clarification.
  We make $A_1$ a left $\bfk F$-module by the action given by $a \triangleright v = a v a^{-1}$ for $a \in F$ and $v \in A_1$. Then $A_0$ is a submodule of $A_1$, and thus $A_1 / A_0$ has a canonical structure of a left $\bfk F$-module. Moreover, $\bfk y$ is a one-dimensional submodule of $A_1/A_0$. By Lemma~\ref{lem:lifting-0}, there is an element $y_1 \in A_1$ such that $\delta_A(y_1) = x \otimes 1 + g \otimes y_1$ and $\pi(y_1) = y$. It is straightforward to verify that the element $\tilde{y} := |F|^{-1} \sum_{a \in F} \chi(a)^{-1} a \triangleright y_1$ satisfies \eqref{eq:lifting-skew-primitive-type-element}.
\end{proof}

\section{Exact comodule algebras over $u_q(\mathfrak{sl}_2)$}
\label{sec:exact-comod-alg-uqsl2}

\subsection{The Hopf algebras $u_q(\mathfrak{sl}_2)$ and $\gr(u_q(\mathfrak{sl}_2))$}

Throughout this section, $\bfk$ is an algebraically closed field of characteristic zero. We also fix an odd integer $N > 1$ and a root of unity $q \in \bfk$ of order $N$.
For an indeterminate $t$ and integers $m$ and $r$ with $m \ge r \ge 0$, we define
\begin{equation*}
  (m)_t = \sum_{i = 0}^{m-1} t^i,
  \quad (m)_t! = \prod_{i = 1}^{m} (i)_t
  \quad \text{and} \quad
  \binom{m}{r}_{\!\!t} = \frac{(m)_t!}{(r)_t! (m-r)_t!}
\end{equation*}
with the convention $(0)_t = 0$ and $(0)_t! = 1$. Since they are in fact polynomials in $t$, we may evaluate them at any element of $\bfk$. The following variant of the binomial formula is used frequently: If $A$ and $B$ are elements of the same algebra subject to the relation $B A = \lambda A B$ for some $\lambda \in \bfk$, then we have
\begin{equation*}
  (A + B)^m = \sum_{r = 0}^m \binom{m}{r}_{\!\!\lambda} A^r B^{m-r}
\end{equation*}
for all integers $m \ge 0$. In particular, we have $(A + B)^m = A^m + B^m$ when $\lambda$ is a root of unity of order $m \ge 1$. In accordance with the terminology commonly used in the literature, we refer to these identities as the {\em $q$-binomial formula}, even when the parameter $\lambda$ is not $q$.

The Hopf algebra $u_q := u_q(\mathfrak{sl}_2)$ is defined as follows: As an algebra, it is generated by $E$, $F$ and $K$ subject to the relations $K E = q^2 E K$, $K F = q^{-2} F K$, $K^N = 1$, $E^N = F^N = 0$ and $E F - F E = (K-K^{-1})/(q-q^{-1})$. The comultiplication $\Delta$, the counit $\varepsilon$ and the antipode $S$ are given by
\begin{gather*}
  \Delta(E) = E \otimes K + 1 \otimes E,
  \quad \Delta(F) = F \otimes 1 + K^{-1} \otimes F,
  \quad \Delta(K) = K \otimes K, \\
  \varepsilon(K) = 1, \ \varepsilon(E) = \varepsilon(F) = 0,
  \ S(K) = K^{-1},
  \ S(E) = -EK^{-1},
  \ S(F) = -KF
\end{gather*}
on the generators. For consistency with the theory of Nichols algebras, we introduce the new generator $\tilde{E} := (q-q^{-1}) K^{-1} E$. By using the Poincar\'e-Birkhoff-Witt basis of $u_q$, we see that the set
\begin{equation}
  \label{eq:uq-modified-PBW-basis}
  \{ \tilde{E}{}^a F^b K^c \mid a, b, c = 0, 1, \cdots, N - 1 \}
\end{equation}
is a basis of $u_q$. Moreover, we have
\begin{equation*}
  \tilde{E}{}^N = 0,
  \ K \tilde{E} = q^2 \tilde{E} K,
  \ \tilde{E} F - q^2 F \tilde{E} = 1 - K^{-2},
  \ \Delta(\tilde{E}) = \tilde{E} \otimes 1 + K^{-1} \otimes \tilde{E}.
\end{equation*}

It is well-known that $u_q$ is a pointed Hopf algebra whose coradical is the subalgebra generated by $K$, and the coradical filtration of $u_q$ agrees with the algebra filtration given by $\deg(E) = \deg(F) = 1$ and $\deg(K) = 0$. We always understand $\gr(u_q)$ as the associated graded Hopf algebra of $u_q$ with respect to the coradical filtration. The Hopf algebra $\gr(u_q)$ is generated by $E$, $F$ and $K$ subject to the same relations as $u_q$ but with the last one replaced with $E F - F E = 0$. For notational reasons, we denote by $x$, $y$ and $g$ the elements $\tilde{E}$, $F$ and $K$ of $\gr(u_q)$, respectively. Then the Hopf algebra $\gr(u_q)$ can alternatively be defined as the algebra generated by $x$, $y$ and $g$ subject to the relations
\begin{equation}
  \label{eq:gr-uq-sl2-relations}
  x^N = y^N = 0,
  \quad g^N = 1,
  \quad g x = q^2 x g,
  \quad g y = q^{-2} y g,
  \quad x y = q^2 y x
\end{equation}
endowed with the comultiplication determined by
\begin{equation}
  \label{eq:gr-uq-sl2-comultiplication}
  \Delta(x) = x \otimes 1 + g^{-1} \otimes x, \quad
  \Delta(y) = y \otimes 1 + g^{-1} \otimes y, \quad
  \Delta(g) = g \otimes g.
\end{equation}

\begin{lemma}
  \label{lem:gr-uq-sl2-cocycle-deform}
  There is a Hopf 2-cocycle $\sigma$ on $\gr(u_q)$ given by
  \begin{equation}
    \label{eq:gr-uq-sl2-cocycle-deform}
    \sigma(x^{i_1} y^{j_1} g^{k_1}, x^{i_2} y^{j_2} g^{k_2}) =
    \delta_{i_1, j_2} \delta_{j_1, 0} \delta_{i_2, 0} (i_1)_{q^2}! \, q^{-2 i_1 k_1}
  \end{equation}
  for $i_1, j_1, k_1, i_2, j_2, k_2 \in \{ 0, 1, \cdots, N - 1 \}$.
  The Hopf algebra $\gr(u_q)^{\sigma}$ is isomorphic to $u_q$ via the algebra map
  \begin{equation}
    \label{eq:gr-uq-sl2-cocycle-deform-2}
    u_q \to \gr(u_q)^{\sigma}, \quad \tilde{E} \mapsto x, \quad F \mapsto y, \quad K \mapsto g.
  \end{equation}
\end{lemma}

The Hopf 2-cocycle $\sigma$ of this lemma is obtained by a general method for constructing Hopf 2-cocycles on pointed Hopf algebras given in \cite{2010arXiv1010.4976G}. For the reader's convenience, we provide the details in Appendix~\ref{appendix:detail-Hopf-2-cocycle}.

\subsection{Step 1. List graded coideal subalgebras of $\gr(u_q(\mathfrak{sl}_2))$}

It is well-known that $\mathrm{H}^2(C) = 0$ for a finite cyclic group $C$. Since $\Grp(u_q)$ is the cyclic group of order $N$ (generated by $K$), Assumption \ref{assumptions-for-classification-strategy} is fulfilled in our case. Now we start deploying the classification strategy explained in Subsection~\ref{subsec:classification-strategy}.

The first step is to list all the graded left coideal subalgebras of $\gr(u_q)$. We have already classified right coideal subalgebras of $\gr(u_q)$ in \cite{2024arXiv241010064S}. Since left coideal subalgebras and right coideal subalgebras of a Hopf algebra are in bijection via the antipode, we know from \cite[Example 3.5]{2024arXiv241010064S} that a homogeneous left coideal subalgebra of $\gr(u_q)$ is one of the subalgebras
\begin{equation}
  \label{eq:homogeneous-CSA}
  \begin{gathered}
    L_0(r) := \langle g^{N/r} \rangle,
    \quad L_1(r) := \langle g^{N/r}, x \rangle,
    \quad L_2(r) := \langle g^{N/r}, y \rangle, \\
    \quad L_3(r) := \langle g^{N/r}, x, y \rangle
    \quad \text{or}
    \quad L_4(\alpha, \beta) := \langle \alpha x + \beta y \rangle
  \end{gathered}
\end{equation}
of $\gr(u_q)$, where $r$ is a positive divisor of $N$, and $\alpha$ and $\beta$ are elements of $\bfk$ with at least one of $\alpha$ and $\beta$ non-zero.

The above list has duplicates. For example, we have $L_1(1) = L_4(1, 0)$, $L_2(1) = L_4(0, 1)$ and $L_4(\alpha, \beta) = L_4(\alpha', \beta')$ if $\alpha \beta' = \alpha' \beta$. We do not need to worry about these duplicates now, as they will be resolved when we discuss $\gr(u_q)$-Morita equivalence at a later stage of the classification.

\subsection{Step 2. Find liftings as comodule algebras over $\gr(u_q(\mathfrak{sl}_2))$}

Let $L$ be a graded left coideal subalgebra of $H := \gr(u_q)$. Following the terminology used in the study of pointed Hopf algebras, by a {\em lifting} of $L$, we mean a left $H$-comodule algebra $\mathcal{L}$ such that $\gr(\mathcal{L}) \cong L$ as graded left $H$-comodule algebras, where $\gr$ is taken with respect to the Loewy filtration. In this subsection, for each coideal subalgebra in the list \eqref{eq:homogeneous-CSA}, we find its liftings. We first introduce several families of left $H$-comodule algebras:

\begin{definition}
  \label{def:list-of-liftings}
  For a positive divisor $r$ of $N$ and parameters $\alpha, \beta, \xi, \zeta, \eta \in \bfk$ with $(\alpha, \beta) \ne (0,0)$, we introduce the following algebras:
  \begin{enumerate}
    \setcounter{enumi}{-1}
  \item The algebra $\mathscr{L}_0(r)$ is generated by $G$ subject to
    \begin{equation*}
      G^r = 1.
    \end{equation*}
  \item The algebra $\mathscr{L}_1(r; \xi)$ is generated by $G$ and $X$ subject to
    \begin{equation*}
      G^r = 1, \quad X^N = \xi, \quad G X = q^{2N/r} X G.
    \end{equation*}
  \item The algebra $\mathscr{L}_2(r; \zeta)$ is generated by $G$ and $Y$ subject to
    \begin{equation*}
      G^r = 1, \quad
      Y^N = \zeta, \quad G Y = q^{-2N/r} Y G.
    \end{equation*}
  \item The algebra $\mathscr{L}_3(r; \xi, \zeta)$ is generated by $G$, $X$ and $Y$ subject to the relations for $\mathscr{L}_1(r; \xi)$ and $\mathscr{L}_2(r; \zeta)$, and
    \begin{equation*}
      X Y - q^2 Y X = 0.
    \end{equation*}
  \item[($3'$)] The algebra $\mathscr{L}_3(N; \xi, \zeta, \eta)$ is generated by $G$, $X$ and $Y$ subject to the same relations as $\mathscr{L}_3(N; \xi, \zeta)$ but with the last one replaced with
    \begin{equation*}
      X Y - q^2 Y X = - \eta G^{-2}.
    \end{equation*}
  \item The algebra $\mathscr{L}_4(\alpha, \beta; \xi)$ is generated by $W$ subject to
    \begin{equation*}
      W^N = \xi.
    \end{equation*}
  \end{enumerate}
  The left $H$-comodule structures of these algebras are given by
  \begin{equation}
    \label{eq:gr-uq-coaction}
    \begin{gathered}
      \delta(X) = x \otimes 1 + g^{-1} \otimes X,
      \quad \delta(Y) = y \otimes 1 + g^{-1} \otimes Y, \\
      \delta(W) = (\alpha x + \beta y) \otimes 1 + g^{-1} \otimes W,
      \quad \delta(G) = g^{N/r} \otimes G
    \end{gathered}
  \end{equation}
  on the generators, where $r$ is taken to be $N$ for $\mathscr{L}_3(N; \xi, \zeta, \eta)$.
\end{definition}

With the help of the $q$-binomial formula, one can verify that \eqref{eq:gr-uq-coaction} indeed makes them left $H$-comodule algebras.
By Theorem~\ref{thm:H-simplicity-comod-alg}, all the left $H$-comodule algebras in Definition~\ref{def:list-of-liftings} are right $H$-simple.

\begin{remark}
  \label{rem:Mombelli-list-gr-uq-comod-alg}
  Mombelli presented a list of right $H$-simple left $H$-comodule algebras up to $H$-Morita equivalence in \cite[\S8.4]{MR2678630}. With his notation, we have
  \begin{gather*}
    \mathscr{L}_0(d) \cong \bfk C_d,
    \quad \mathscr{L}_1(d; \xi) \cong \mathcal{A}_0(d, \xi),
    \quad \mathscr{L}_2(d; \xi) \cong \mathcal{A}_1(d, \xi), \\
    \mathscr{L}_4(\alpha, \beta; \xi)
    \cong \mathcal{A}(\xi \beta^{-N}, \alpha \beta^{-1})
    \quad (\beta \ne 0)
    \quad \text{and}
    \quad \mathscr{L}_4(\alpha, 0; \xi) \cong \mathcal{A}_0(1, \alpha^{-N} \xi).
  \end{gather*}
  The algebras of type $\mathscr{L}_3$ do not appear in Mombelli's list. In Step 4 (Subsection~\ref{subsec:Step-4}), we will compute the cocycle deformation ${}_{\sigma}\mathcal{L}$ for each left $H$-comodule algebra $\mathcal{L}$ of Definition \ref{def:list-of-liftings}. The result will show
  \begin{equation*}
    {}_{\sigma}(\mathscr{L}_3(r; \xi, \zeta)) \cong \mathcal{B}(r, \zeta, \xi)
    \quad \text{and} \quad
    {}_{\sigma}(\mathscr{L}_3(N; \xi, \zeta, \eta))
    \cong \mathcal{C}(N, \zeta, \xi, -\eta, N-2)
  \end{equation*}
  with the notation of \cite[\S8.4]{MR2678630}. In fact, the algebras of type $\mathcal{B}$ and $\mathcal{C}$ in \cite[\S8.4]{MR2678630} are not well-defined as left $H$-comodule algebras, but are well-defined as left $u_q$-comodule algebras.
\end{remark}

\begin{remark}
  The minus sign of the right-hand side of one of the defining relations of the algebra $\mathscr{L}_3(N; \xi, \zeta, \eta)$, $X Y - q^2 Y X = -\eta G^{-2}$, is attached so that the cocycle deformation of $\mathscr{L}_3(N; 0, 0, 1)$ becomes $u_q(\mathfrak{sl}_2)$.
\end{remark}

Our task is to show that the above list exhausts all liftings.
We first examine the coideal subalgebra $L_3(N)$ ($=H$), which is in fact the most complicated case.

\begin{lemma}
  \label{lem:lifting-L3N}
  A lifting of $L_3(N)$ is isomorphic to $\mathscr{L}_3(N; \xi, \zeta, \eta)$ for some $\xi, \zeta, \eta \in \bfk$.
\end{lemma}
\begin{proof}
  Let $\mathcal{L}$ be a lifting of $L_3(N)$, and let $\mathcal{L}_i$ ($i = 0, 1, \cdots$) be the $i$-th layer of the Loewy filtration of $\mathcal{L}$. By definition, we may, and do, identify
  \begin{equation*}
    \mathcal{L}_{i}/\mathcal{L}_{i-1}
    = \mathrm{span}_{\bfk}\{ g^{t} x^{u} y^{v} \mid 0 \le t, u, v < N, u + v = i \}.
  \end{equation*}
  In particular, $\mathcal{L}_0$ is the subalgebra of $L_3(N)$ generated by $g$. To avoid confusion, we denote by $\tilde{G}$ the element of $\mathcal{L}_0$ corresponding to $g \in L_3(N)$. Let $\pi : \mathcal{L}_1 \to \mathcal{L}_1/\mathcal{L}_0$ be the projection. By Lemma~\ref{lem:lifting-skew-primitive-type-element}, there are elements $\tilde{X}$ and $\tilde{Y}$ of $\mathcal{L}$ such that
  \begin{gather*}
    \delta_{\mathcal{L}}(\tilde{X}) = x \otimes 1 + g^{-1} \otimes \tilde{X},
    \quad \tilde{G} \tilde{X} = q^{2} \tilde{X} \tilde{G},
    \quad \pi(\tilde{X}) = x, \\
    \delta_{\mathcal{L}}(\tilde{Y}) = y \otimes 1 + g^{-1} \otimes \tilde{Y},
    \quad \tilde{G} \tilde{Y} = q^{-2} \tilde{Y} \tilde{G},
    \quad \pi(\tilde{Y}) = y.
  \end{gather*}
  By the $q$-binomial formula, we have
  \begin{equation*}
    \delta_{\mathcal{L}}(\tilde{X}^N) = (x \otimes 1 + g^{-1} \otimes \tilde{X})^N
    = (x \otimes 1)^N + (g^{-1} \otimes \tilde{X})^N = 1 \otimes \tilde{X}^N.
  \end{equation*}
  This means that $\tilde{X}^N$ belongs to the space of coinvariants of $\mathcal{L}$, which is spanned by the unit of $\mathcal{L}$. Hence $\tilde{X}^N = \xi$ for some $\xi \in \bfk$. By the same argument, we have $\tilde{Y}^N = \zeta$ for some $\zeta \in \bfk$. Now we consider the subspace
  \begin{equation}
    \label{eq:lifting-L3N}
    V := \{ v \in \mathcal{L} \mid \delta_{\mathcal{L}}(v) = g^{-2} \otimes v \}
  \end{equation}
  of $\mathcal{L}_0$. Since $\mathcal{L}_0$ is the subalgebra generated by $\tilde{G}$, it is easy to see that $V$ is spanned by $\tilde{G}{}^{-2}$.
  The element $z := \tilde{X} \tilde{Y} - q^2 \tilde{Y} \tilde{X}$ belongs to $V$. Indeed,
  \begin{equation*}
    \delta_{\mathcal{L}}(z) = (x y - q^2 y x) \otimes 1 + g^{-2} \otimes (\tilde{X}\tilde{Y} - q^2 \tilde{Y} \tilde{X}) = g^{-2} \otimes z.
  \end{equation*}
  Thus $z = -\eta \tilde{G}{}^{-2}$ for some $\eta \in \bfk$. We now obtain a homomorphism
  \begin{equation*}
    \phi : \mathscr{L}_3(N; \xi, \zeta, \eta) \to \mathcal{L},
    \quad X \mapsto \tilde{X},
    \quad Y \mapsto \tilde{Y},
    \quad G \mapsto \tilde{G}
  \end{equation*}
  of left $\gr(u_q)$-comodule algebras. Since $\mathscr{L}_3(N; \xi, \zeta, \eta)$ is right $\gr(u_q)$-simple, $\phi$ is injective. Since the source and the target of $\phi$ have the same dimension, $\phi$ is in fact an isomorphism. The proof is done.
\end{proof}

If $\mathcal{L}$ is a lifting of $L_3(r)$ with $r < N$, then we may identify
\begin{equation*}
  \mathcal{L}_i/\mathcal{L}_{i-1}
  = \mathrm{span}\{ (g^{N/r})^t x^{u} y^{v} \mid 0 \le t < r; 0 \le u, v < N; u + v = i \}
\end{equation*}
and take $\tilde{G} \in \mathcal{L}_0$ to be the element corresponding to $g^{N/r} \in H$ as in the above proof of Lemma~\ref{lem:lifting-L3N}. The classification of liftings of $L_3(r)$ is completed in almost the same way as above. The difference from the case where $r = N$ is that the subspace $V$, defined by \eqref{eq:lifting-L3N}, is zero when $r < N$ since $\delta_{\mathcal{L}}(\tilde{G}) = g^{N/r} \otimes \tilde{G}$. Due to this, there is no room for the parameter $\eta$ to appear, and the result is as follows:

\begin{lemma}
  Let $r$ be a positive divisor of $N$ with $r < N$. Then a lifting of $L_3(r)$ is isomorphic to $\mathscr{L}_3(r; \xi, \zeta)$ for some $\xi, \zeta \in \bfk$.
\end{lemma}

Lemmas~\ref{lem:Step-2-lifting-L0}--\ref{lem:Step-2-lifting-L4} below are also proved in a similar way.

\begin{lemma}
  \label{lem:Step-2-lifting-L0}
  All liftings of $L_0(r)$ are isomorphic to $\mathscr{L}_0(r)$.
\end{lemma}

\begin{lemma}
  \label{lem:Step-2-lifting-L1}
  A lifting of $L_1(r)$ is isomorphic to $\mathscr{L}_1(r; \xi)$ for some $\xi \in \bfk$.
\end{lemma}

\begin{lemma}
  \label{lem:Step-2-lifting-L2}
  A lifting of $L_2(r)$ is isomorphic to $\mathscr{L}_2(r; \zeta)$ for some $\zeta \in \bfk$.
\end{lemma}

\begin{lemma}
  \label{lem:Step-2-lifting-L4}
  A lifting of $L_4(\alpha, \beta)$ is isomorphic to $\mathscr{L}_4(\alpha, \beta; \xi)$ for some $\xi \in \bfk$.
\end{lemma}

\subsection{Step 3. Detect $\gr(u_q(\mathfrak{sl}_2))$-Morita equivalence}
\label{subsec:Step-3}

As in the previous subsection, we write $H = \gr(u_q)$. Here we accomplish Step 3 of the strategy, that is, detect $H$-Morita equivalence between left $H$-comodule algebras obtained in the previous step. We first fix parameters $\xi, \zeta, \eta, \xi', \zeta', \eta' \in \bfk$ and prove:

\begin{claim}
  \label{claim:H-Morita-L3-N}
  $\mathscr{L}_{3}(N; \xi, \zeta, \eta)$ and $\mathscr{L}_{3}(N; \xi', \zeta', \eta')$ are $H$-Morita equivalent if and only if there exists an integer $k$ such that $(\xi', \zeta', \eta') = (\xi, \zeta, \eta q^{2k})$.
\end{claim}
\begin{proof}
  We set $\mathcal{L} = \mathscr{L}_{3}(N; \xi, \zeta, \eta)$ and $\mathcal{L}' = \mathscr{L}_{3}(N; \xi', \zeta', \eta')$. For $m \in \mathbb{Z}$, there is an isomorphism $\mathcal{L} \to g^{m} \mathcal{L} g^{-m}$ of left $H$-comodule algebras given by sending the generators $G$, $X$ and $Y$ of $\mathcal{L}$ to $G$, $q^{-2m} X$ and $q^{2m} Y$, respectively. Thus, by Theorem~\ref{thm:H-Morita-pointed}, $\mathcal{L}$ and $\mathcal{L}'$ are $H$-Morita equivalent if and only if they are isomorphic as left $H$-comodule algebras.

  We discuss when $\mathcal{L} \cong \mathcal{L}'$ as left $H$-comodule algebras. The first layer of the Loewy filtration of $\mathcal{L}$ is decomposed as $\mathcal{L}_1 = V_0 \oplus \dotsb \oplus V_{N - 1}$, where $V_i$ is the subspace of $\mathcal{L}$ spanned by $G^i$, $G^i X$ and $G^i Y$. Each $V_i$ is a subcomodule of $\mathcal{L}$. Given left $H$-comodules $V$ and $W$, we denote by ${}^H\Hom(V, W)$ the space of left $H$-comodule maps from $V$ to $W$. It is straightforward to verify
  \begin{gather*}
    {}^H\Hom(V_i, V_i) = \bfk \, \id_{V_i},
    \quad {}^H\Hom(V_i, V_{i - 1}) = \{ f_{a,b}^{(i)} \mid a, b \in \bfk \}
  \end{gather*}
  and ${}^H\Hom(V_i, V_j) = 0$ ($j \ne i, i - 1$) for all $i, j = 0, 1, \cdots, N - 1$ with the convention $V_{-1} = V_{N-1}$. Here, $f_{a, b}^{(i)} : V_i \to V_{i - 1}$ ($a, b \in \bfk$) is the left $H$-comodule map sending $G^i$, $G^i X$ and $G^i Y$ to $0$, $a G^{i-1}$ and $b G^{i-1}$, respectively.

  Now we suppose that there is an isomorphism $\phi : \mathcal{L}' \to \mathcal{L}$ of left $H$-comodule algebras. To avoid confusion, we denote the generators $G$, $X$ and $Y$ of $\mathcal{L}'$ by $G'$, $X'$ and $Y'$, respectively, and define $V_i' \subset \mathcal{L}'$ in the same way as $V_i$.
  By considering the isomorphism $\socle(\mathcal{L}) \cong \socle(\mathcal{L}')$ of left $H$-comodules induced by $\phi$, we have $\phi(G') = a G$ for some $a \in \bfk$. The isomorphism also induces a morphism $V_0' \to \mathcal{L}_1$ of left $H$-comodules. We note that $V_0'$ is identified with $V_0$ as a left $H$-comodule. By the above computation of Hom-spaces between $V_i$'s, we have
  \begin{equation*}
    {}^H\Hom(V_0', \mathcal{L}_1)
    \cong \bigoplus_{i = 0}^{N - 1} {}^H\Hom(V_0', V_i)
    = {}^H\Hom(V_0', V_0) \oplus {}^H\Hom(V_0', V_{N-1}).
  \end{equation*}
  This implies that $\phi(X') = b_1 X + b_2 G^{-1}$ and $\phi(Y') = c_1 Y + c_2 G^{-1}$ for some $b_1, b_2, c_1, c_2 \in \bfk$.
  Since both $G$ and $G'$ are of order $N$, we have $a^N = 1$. Thus $a = q^{k}$ for some $k \in \mathbb{Z}$. We also have
  \begin{gather*}
    x \otimes 1 + g^{-1} \otimes (b_1 X + b_2 G^{-1})
    = (\id_H \otimes \phi) \delta_{\mathcal{L}'}(X')
    = \delta_{\mathcal{L}} \phi(X') \\
    = x \otimes b_1 + g^{-1} \otimes (b_1 X + b_2 G^{-1}),
  \end{gather*}
  which implies $b_1 = 1$. Moreover, since
  \begin{equation*}
    q^2 X + b_2 G^{-1}
    = \phi(G') \phi(X') \phi(G'{}^{-1})
    = q^2 \phi(X') = q^2 X + q^2 b_2 G^{-1},
  \end{equation*}
  we have $b_2 = 0$. In a similar way, we obtain $c_1 = 1$ and $c_2 = 0$. Summarizing, $\phi$ is given by $\phi(G') = q^{k} G$, $\phi(X') = X$ and $\phi(Y') = Y$ for some $k \in \mathbb{Z}$. Hence,
  \begin{gather*}
    \xi = X^N = \phi(X')^N = \phi(\xi') = \xi', \quad
    \zeta = Y^N = \phi(Y')^N = \phi(\zeta') = \zeta', \\
    - \eta G^{-2} = X Y - q^2 Y X
    = \phi(X' Y' - q^2 Y' X')
    = \phi(- \eta' G'^{-2})
    = - \eta' q^{-2k} G^{-2}.
  \end{gather*}
  In conclusion, we have $(\xi', \zeta', \eta') = (\xi, \zeta, \eta q^{2k})$.
  Hence the `only if' part is proved. The converse also holds since there is an isomorphism
  \begin{equation*}
    \mathscr{L}_{3}(N; \xi, \zeta, \eta q^{2k})
    \to \mathscr{L}_{3}(N; \xi, \zeta, \eta),
    \quad G \mapsto q^{k} G, \quad X \mapsto X, \quad Y \mapsto Y
  \end{equation*}
  of left $H$-comodule algebras. The proof is done.
\end{proof}

By the same argument as in Claim~\ref{claim:H-Morita-L3-N}, we have:

\begin{claim}
  \label{claim:H-Morita-L3-r}
  $\mathscr{L}_{3}(r; \xi, \zeta)$ and $\mathscr{L}_{3}(r'; \xi', \zeta')$ are $H$-Morita equivalent if and only if $(r', \xi', \zeta') = (r, \xi, \zeta)$.
\end{claim}

For notational convenience, we set $\mathscr{L}_3(r; \xi, \zeta, 0) := \mathscr{L}_3(r; \xi, \zeta)$ even for a positive divisor $r$ of $N$ with $r < N$. Claims \ref{claim:H-Morita-L3-N} and~\ref{claim:H-Morita-L3-r} are combined into the following lemma:

\begin{lemma}
  \label{lem:H-Morita-L3}
  Let $r$ and $r'$ be positive divisors of $N$, and let $\xi, \xi', \zeta, \zeta', \eta, \eta' \in \bfk$. We assume that $\eta = 0$ if $r < N$, and also assume that $\eta' = 0$ if $r' < N$. Then the left $H$-comodule algebras $\mathscr{L}_3(r; \xi, \zeta, \eta)$ and $\mathscr{L}_3(r'; \xi', \zeta', \eta')$ are $H$-Morita equivalent if and only if there exists $k \in \mathbb{Z}$ such that
  \begin{equation*}
    (r', \xi', \zeta', \eta') = (r, \xi, \zeta, q^{2k} \eta).
  \end{equation*}
\end{lemma}
\begin{proof}
  The case where $r = r' = N$ and the case where $r, r' < N$ have already been discussed in Claims \ref{claim:H-Morita-L3-N} and~\ref{claim:H-Morita-L3-r}, respectively. The claim of this lemma is easily verified in other cases just by comparing the dimensions of the algebras.
\end{proof}

We omit proofs of Lemmas~\ref{lem:H-Morita-L0}--\ref{lem:H-Morita-L2} below since they are proved in the same way as Claim~\ref{claim:H-Morita-L3-N}. We use parameters introduced in Lemma~\ref{lem:H-Morita-L3}.

\begin{lemma}
  \label{lem:H-Morita-L0}
  $\mathscr{L}_0(r)$ and $\mathscr{L}_0(r')$ are $H$-Morita equivalent if and only if $r = r'$.
\end{lemma}

\begin{lemma}
  \label{lem:H-Morita-L1}
  $\mathscr{L}_1(r; \xi)$ and $\mathscr{L}_1(r'; \xi')$ are $H$-Morita equivalent if and only if
  \begin{equation*}
    (r', \xi') = (r, \xi).
  \end{equation*}
\end{lemma}

\begin{lemma}
  \label{lem:H-Morita-L2}
  $\mathscr{L}_2(r; \zeta)$ and $\mathscr{L}_2(r'; \zeta')$ are $H$-Morita equivalent if and only if
  \begin{equation*}
    (r', \zeta') = (r, \zeta).
  \end{equation*}
\end{lemma}

We choose parameters $\alpha, \beta, \xi, \alpha', \beta', \xi' \in \bfk$ with $(\alpha, \beta), (\alpha', \beta') \ne (0,0)$ and discuss when $\mathscr{L}_4(\alpha, \beta; \xi)$ and $\mathscr{L}_4(\alpha', \beta', \xi')$ are $H$-Morita equivalent. The result looks a little different from the previous cases, and is as follows:

\begin{lemma}
  \label{lem:H-Morita-L4}
  $\mathscr{L}_4(\alpha, \beta; \xi)$ and $\mathscr{L}_4(\alpha', \beta'; \xi')$ are $H$-Morita equivalent if and only if there exist an integer $k$ and a non-zero element $\lambda \in \bfk$ such that
  \begin{equation*}
    (\alpha', \beta', \xi') = (\lambda q^{2k} \alpha, \lambda q^{-2k} \beta, \lambda^N \xi).
  \end{equation*}
\end{lemma}
\begin{proof}
  We set $\mathcal{L} = \mathscr{L}_{4}(\alpha, \beta; \xi)$ and $\mathcal{L}' = \mathscr{L}_{4}(\alpha', \beta'; \xi')$. For $k \in \mathbb{Z}$, there is an isomorphism $g^k \mathcal{L} g^{-k} \cong \mathscr{L}_{4}(q^{2k} \alpha, q^{-2k} \beta; \xi)$ of left $H$-comodule algebras.
  Thus, in view of Theorem~\ref{thm:H-Morita-pointed}, it suffices to show
  \begin{equation}
    \label{eq:H-Morita-L4-eq-1}
    \text{$\mathcal{L} \cong \mathcal{L}'$ as left $H$-comodule algebras}
    \iff \text{$(\alpha', \beta', \xi') = (\lambda \alpha, \lambda \beta, \lambda^N \xi)$}.
  \end{equation}
  We now prove \eqref{eq:H-Morita-L4-eq-1}. To avoid confusion, we denote the generator $W$ of $\mathcal{L}'$ by $W'$. Suppose that there is an isomorphism $\phi : \mathcal{L}' \to \mathcal{L}$. Then, by an argument using the Loewy filtration as in the proof of Claim~\ref{claim:H-Morita-L3-N}, we have $\phi(W') = \lambda W$ for some $\lambda \in \bfk$. Since $\phi$ is an isomorphism, $\lambda \ne 0$. We also have
  \begin{gather*}
    \alpha' x \otimes 1 + \beta' y \otimes 1 + g^{-1} \otimes \lambda W
    = (\id \otimes \phi) \delta_{\mathcal{L}'}(W') \\
    = \delta_\mathcal{L} \phi(W')
    = \lambda \alpha x \otimes 1 + \lambda \beta y \otimes 1 + g^{-1} \otimes \lambda W.
  \end{gather*}
  By comparing the coefficients of $x \otimes 1$ and $y \otimes 1$, we obtain $\alpha' = \lambda \alpha$ and $\beta' = \lambda \beta$, respectively. Furthermore, we have
  \begin{equation*}
    \xi' = \phi(W'{}^N)
    = \phi(W')^N = (\lambda W)^N = \lambda^N \xi.
  \end{equation*}
  The `only if' part ($\Rightarrow$) of \eqref{eq:H-Morita-L4-eq-1} has been verified. The converse also holds since there is an isomorphism
  \begin{equation*}
    \mathscr{L}_{4}(\lambda \alpha, \lambda \beta; \lambda^N \xi)
    \to \mathscr{L}_{4}(\alpha, \beta; \xi),
    \quad W \mapsto \lambda W
  \end{equation*}
  of left $H$-comodule algebras. The proof is done.
\end{proof}

To complete the classification, we define $d(L)$ for a left $H$-comodule algebra $L$ to be the pair $(m/r, r)$, where $m = \dim_{\bfk}(L)$ and $r = \dim_{\bfk}(\socle(L))$. Then we have
\begin{equation}
  \label{eq:H-Morita-invariant-d}
  \begin{gathered}
    d(\mathscr{L}_0(r)) = (1, r), \quad
    d(\mathscr{L}_1(r; \xi)) = (N, r), \quad
    d(\mathscr{L}_2(r; \zeta)) = (N, r), \\
    d(\mathscr{L}_3(r; \xi, \zeta, \eta)) = (N^2, r), \quad
    d(\mathscr{L}_4(\alpha, \beta; \xi)) = (N, 1).
  \end{gathered}
\end{equation}Theorem~\ref{thm:H-Morita-pointed} implies that $d(L)$ is an $H$-Morita invariant for right $H$-simple left $H$-comodule algebras $L$. Thus, for example, we know from \eqref{eq:H-Morita-invariant-d} that $\mathscr{L}_0(r)$ and $\mathscr{L}_1(r'; \xi)$ are not $H$-Morita equivalent. A missing piece to complete our task is the following case:

\begin{lemma}
  \label{lem:H-Morita-L1-L2}
  $\mathscr{L}_1(r; \xi)$ and $\mathscr{L}_2(r'; \zeta')$ are not $H$-Morita equivalent.
\end{lemma}
\begin{proof}
  We write $\mathcal{L} = \mathscr{L}_1(r; \xi)$ and $\mathcal{L}' = \mathscr{L}_2(r'; \xi')$.
  Since $g^m \mathcal{L} g^{-m} \cong \mathcal{L}$ for any $m \in \mathbb{Z}$, they are $H$-Morita equivalent if and only if $\mathcal{L} \cong \mathcal{L}'$. Given a left $H$-comodule $V$, we denote by $C(V)$ the coefficient coalgebra of $V$. Namely,
  \begin{equation*}
    C(V) = \mathrm{span}_{\bfk}
    \{ (\id_H \otimes f) \delta_V(v) \mid v \in V, f \in V^* \}.
  \end{equation*}
  It is easy to see that $C(\mathcal{L})$ contains $x$, while $C(\mathcal{L}')$ does not. Hence $\mathcal{L}$ and $\mathcal{L}'$ are not isomorphic even as left $H$-comodules. The proof is done.
\end{proof}

Here is the conclusion of this subsection: Given a positive integer $m$, we denote by $\mathrm{Div}(m)$ the set of positive divisors of $m$.
Noting that $\mathscr{L}_1(1; \xi)$ and $\mathscr{L}_2(1; \xi)$ are isomorphic to $\mathscr{L}_4(1, 0; \xi)$ and $\mathscr{L}_4(0, 1; \xi)$, respectively, we define
\begin{align*}
  \mathscr{F}_0
  & = \{ \mathscr{L}_0(r) \mid r \in \mathrm{Div}(N) \}, \\
  \mathscr{F}_1
  & = \{ \mathscr{L}_1(r; \xi) \mid r \in \mathrm{Div}(N), r \ne 1; \xi \in \bfk \}, \\
  \mathscr{F}_2
  & = \{ \mathscr{L}_2(r; \xi) \mid r \in \mathrm{Div}(N), r \ne 1; \xi \in \bfk \}, \\
  \mathscr{F}_3
  & = \{ \mathscr{L}_3(r; \xi, \zeta, \eta)
    \mid r \in \mathrm{Div}(N); \xi, \zeta, \eta \in \bfk; \text{$\eta = 0$ if $r < N$} \}, \\
  \mathscr{F}_4
  & = \{ \mathscr{L}_4(\alpha, \beta; \xi)
    \mid \alpha, \beta, \xi \in \bfk, (\alpha, \beta) \ne (0,0) \}
\end{align*}
and call these sets {\em families}.

\begin{theorem}
  Every indecomposable exact left module category over $\lmod{H}$ is equivalent to $\lmod{A}$ for some $A \in \mathscr{F}_0 \cup \mathscr{F}_1 \cup \mathscr{F}_2 \cup \mathscr{F}_3 \cup \mathscr{F}_4$. By \eqref{eq:H-Morita-invariant-d} and Lemma~\ref{lem:H-Morita-L1-L2}, two left $H$-comodule algebras belonging to different families are not $H$-Morita equivalent. For two left $H$-comodule algebras belonging to the same family, whether they are $H$-Morita equivalent is determined by Lemmas \ref{lem:H-Morita-L3}--\ref{lem:H-Morita-L4}.
\end{theorem}

\subsection{Step 4. Compute cocycle deformations}
\label{subsec:Step-4}

We compute deformations of $\gr(u_q)$-comodule algebras obtained in Step 2 by the Hopf 2-cocycle given in Lemma~\ref{lem:gr-uq-sl2-cocycle-deform} to obtain $u_q$-comodule algebras.
For each comodule algebra $\mathcal{L}$ in Definition \ref{def:list-of-liftings}, we denote the cocycle deformation ${}_{\sigma}\mathcal{L}$ by the same symbol but with $\mathscr{L}$ replaced with $\mathscr{A}$. For example,
\begin{equation*}
  \mathscr{A}_3(N; \xi, \zeta, \eta) := {}_{\sigma}(\mathscr{L}_3(N; \xi, \zeta, \eta)).
\end{equation*}

\subsubsection{The cocycle deformation of $\mathscr{L}_3(N; \xi, \zeta, \eta)$}

It turns out that every comodule algebra $\mathcal{L}$ in Definition \ref{def:list-of-liftings} is a comodule subalgebra of $\mathscr{L}_3(N; \xi, \zeta, \eta)$ for some parameters $\xi$, $\zeta$ and $\eta$, and therefore the cocycle deformation ${}_{\sigma}\mathcal{L}$ is a comodule subalgebra of $\mathscr{A}_3(N; \xi, \zeta, \eta)$. Thus we begin by proving:

\begin{lemma}
  \label{lem:cocycle-deform-L3}
  In $\mathscr{A}_3(N; \xi, \zeta, \eta)$, the generators of $\mathscr{L}_3(N; \xi, \zeta, \eta)$ satisfy
  \begin{equation}
    \label{eq:cocycle-deform-L3-relations}
    \begin{gathered}
      G^N = 1, \quad X^N = \xi, \quad Y^N = \zeta, \quad G X = q^2 X G, \\
      G Y = q^{-2} Y G, \quad X Y - q^2 Y X = 1 - \eta G^{-2}.
    \end{gathered}
  \end{equation}
  The left $u_q$-comodule algebra $\mathscr{A}_3(N; \xi, \zeta, \eta)$ can also be defined to be the algebra generated by $G$, $X$ and $Y$ subject to the relations \eqref{eq:cocycle-deform-L3-relations} endowed with the left coaction $\delta$ determined by
  \begin{equation}
    \label{eq:cocycle-deform-L3-coaction}
    \delta(X) = \tilde{E} \otimes 1 + K^{-1} \otimes X,
    \ \delta(Y) = F \otimes 1 + K^{-1} \otimes Y,
    \ \delta(G) = K \otimes G.
  \end{equation}
\end{lemma}
\begin{proof}
  We set $\mathcal{L} = \mathscr{L}_3(N; \xi, \zeta, \eta)$.
  We define $\mathcal{A}$ to be the algebra generated by $X$, $Y$ and $G$ subject to the relations \eqref{eq:cocycle-deform-L3-relations}, and make it a left $u_q$-comodule by \eqref{eq:cocycle-deform-L3-coaction}. Our aim is to show that $\mathcal{A} \cong {}_{\sigma}\mathcal{L}$ as left $u_q$-comodule algebras. Let $*$ denote the twisted multiplication of ${}_{\sigma}\mathcal{L}$. We note that $\delta_{\mathcal{L}}(G)$ and $\delta_{\mathcal{L}}(X)$ belong to $\langle g, x \rangle \otimes \mathcal{L}$. If $a$ and $b$ are elements of $\mathcal{L}$ such that $\delta_{\mathcal{L}}(a), \delta_{\mathcal{L}}(b) \in \langle g, x \rangle \otimes \mathcal{L}$, then we have
  \begin{equation*}
    a * b = \sigma(a_{(-1)}, b_{(-1)}) a_{(0)} b_{(0)}
    = \varepsilon(a_{(-1)}) \varepsilon(b_{(-1)}) a_{(0)} b_{(0)} = a b
  \end{equation*}
  since $\sigma$ is equal to $\varepsilon \otimes \varepsilon$ on $\langle g, x \rangle^{\otimes 2}$. By this observation, we have $G * X = G X$, $X * G = X G$, $G^{*m} = G^m$ and $X^{*m} = X^m$ for any $m \in \mathbb{Z}_{\ge 0}$, where $(-)^{*m}$ denotes the $m$-th power with respect to $*$. Since $\sigma$ is equal to $\varepsilon \otimes \varepsilon$ on $\langle g, y \rangle^{\otimes 2}$, we also have $G * Y = G Y$, $Y * G = Y G$ and $Y^{* m} = Y^m$. Finally, we have
\begin{align*}
  X * Y & = \sigma(x, y) 1 + \sigma(x, g^{-1}) Y + \sigma(g^{-1}, y) X + \sigma(g^{-1}, g^{-1}) X Y = X Y + 1, \\
  Y * X & = \sigma(y, x) 1 + \sigma(y, g^{-1}) X + \sigma(g^{-1}, x) Y + \sigma(g^{-1}, g^{-1}) Y X = Y X.
\end{align*}
By the above computation, we have
\begin{gather*}
  G * X = q^2 X * G, \quad G * Y = q^{-2} Y * G, \quad G^{*N} = 1, \\
  X^{*N} = \xi, \quad Y^{*N} = \zeta, \quad X * Y - q^2 Y * X = 1 - \eta G^{-1} * G^{-1}.
\end{gather*}
Thus we have a well-defined homomorphism
\begin{equation}
  \label{eq:cocycle-deform-L3-proof-1}
  \mathcal{A} \to {}_{\sigma}\mathcal{L}, \quad G \mapsto G, \quad X \mapsto X, \quad Y \mapsto Y
\end{equation}
of left $u_q$-comodule algebras. By Theorem~\ref{thm:H-simplicity-comod-alg}, $\mathcal{A}$ is right $u_q$-simple. Thus \eqref{eq:cocycle-deform-L3-proof-1} is injective. Since $\mathcal{A}$ and ${}_{\sigma}\mathcal{L}$ have the same dimension, we conclude that~\eqref{eq:cocycle-deform-L3-proof-1} is an isomorphism. The proof is done.
\end{proof}

\subsubsection{The cocycle deformation of $\mathscr{L}_i$ for $i = 0, 1, 2, 3$}

Let $r$ be a positive divisor of $N$.
The left $\gr(u_q)$-comodule algebras $\mathscr{L}_0(r)$, $\mathscr{L}_1(r; \xi)$, $\mathscr{L}_2(r; \zeta)$ and $\mathscr{L}_3(r; \xi, \zeta)$ are embedded into $\mathscr{L}_3(N; \xi, \zeta, 0)$ by the algebra map sending the generators $G$, $X$ and $Y$ to $G^{N/r}$, $X$ and $Y$, respectively. By Lemma~\ref{lem:cocycle-deform-L3}, it is easy to obtain the following description of their cocycle deformations:

\begin{lemma}
  \label{lem:cocycle-deform-L0123}
  As algebras, we have
  \begin{equation*}
    \mathscr{A}_0(r) = \mathscr{L}_0(r), \quad
    \mathscr{A}_1(r; \xi) = \mathscr{L}_1(r; \xi) \quad \text{and} \quad
    \mathscr{A}_2(r; \zeta) = \mathscr{L}_2(r; \zeta).
  \end{equation*}
  The algebra $\mathscr{A}_3(r; \xi, \zeta)$ is generated by $G$, $X$ and $Y$ subject to the relations
  \begin{gather*}
    G^r = 1, \quad X^N = \xi, \quad Y^N = \zeta, \quad G X = q^{2N/r} X G, \\
    G Y = q^{-2N/r} Y G, \quad X Y - q^2 Y X = 1.
  \end{gather*}
  The left $u_q$-comodule structures are determined by
  \begin{equation*}
    \delta(X) = \tilde{E} \otimes 1 + K^{-1} \otimes X,
    \quad \delta(Y) = F \otimes 1 + K^{-1} \otimes Y,
    \quad \delta(G) = K^{N/r} \otimes G
  \end{equation*}
  on the generators.
\end{lemma}

\subsubsection{The cocycle deformation of $\mathscr{L}_4(\alpha, \beta; \xi)$}
\label{subsec:cocycle-deform-L4}

We fix parameters $\alpha, \beta, \xi \in \bfk$ with $(\alpha, \beta) \ne (0, 0)$. Here, we show that the cocycle deformation of $\mathscr{L}_4(\alpha, \beta; \xi)$ is generated by a single element and compute the minimal polynomial of the generator.

For an integer $n \ge 0$ and indeterminates $z$ and $a$, the {\em Dickson polynomial} $D_n(z, a)$ is the two-variable polynomial defined recursively by $D_0(z,a) = 2$, $D_1(z,a) = z$ and $D_n(z, a) = z D_{n-1}(z,a) - a D_{n-2}(z,a)$ for $n \ge 2$. Explicitly, it is given by
\begin{equation}
  \label{eq:Dickson-poly-formula}
  D_n(z, a) = \sum_{i = 0}^{\lfloor n/2 \rfloor} \frac{n}{n-i} \binom{n-i}{i} (-a)^i z^{n-2i}
  \quad (n = 1, 2, \cdots),
\end{equation}
where $\lfloor - \rfloor$ is the floor function. The Dickson polynomial is crucial for our analysis. Before proceeding to the main discussion, we provide the following lemma:

\begin{lemma}
  For a root of unity $\omega \in \bfk$ of order $n$, we have
  \begin{equation}
    \label{eq:Dickson-poly-1}
    D_n(z, u v) - u^n - v^n
    = \prod_{i = 0}^{n - 1} ( z - (u \omega^i + v \omega^{-i}))
  \end{equation}
  in the algebra $\bfk[u, v, z]$ of polynomials with variables $u$, $v$ and $z$.
\end{lemma}
\begin{proof}
  It is well-known that the Dickson polynomial satisfies
  \begin{equation}
    \label{eq:Dickson-poly-2}
    D_n(s + t, s t) = s^n + t^n
  \end{equation}
  in the algebra $\bfk[s, t]$. We regard both sides of \eqref{eq:Dickson-poly-1} as polynomials in $z$ with coefficients in $\bfk[u,v]$. It follows from \eqref{eq:Dickson-poly-2} that both sides of equation \eqref{eq:Dickson-poly-1} have the common distinct roots $u \omega^i + v \omega^{-i}$ ($i = 0, 1, \cdots, n - 1$). Comparing the leading coefficients, we conclude that \eqref{eq:Dickson-poly-1} holds.
\end{proof}

The minimal polynomial of the generator of $\mathscr{A}_4(\alpha, \beta; \xi)$ is expressed by the Dickson polynomial. More precisely, in this subsection, we prove:

\begin{lemma}
  \label{lem:cocycle-deform-L4}
  The left $u_q$-comodule algebra $\mathscr{A}_4(\alpha, \beta; \xi)$ is generated by a single element $W$ subject to the relation $\phi_{\alpha, \beta, \xi}(W) = 0$, where
  \begin{equation}
    \label{eq:cocycle-deform-L4-min-pol}
    \phi_{\alpha, \beta, \xi}(t) = D_{N}(t, \alpha \beta (1-q^2)^{-1}) - \xi.
  \end{equation}
  The left coaction of $u_q$ on $\mathscr{A}_4(\alpha, \beta; \xi)$ is given by
  \begin{equation}
    \label{eq:cocycle-deform-L4-generator-coaction}
    \delta(W) = (\alpha \tilde{E} + \beta F) \otimes 1 + K^{-1} \otimes W
  \end{equation}
  on the generator.
\end{lemma}

If $\alpha = 0$, then there is an isomorphism $\mathscr{L}_4(0, \beta; \xi) \to \mathscr{L}_2(1; \beta^{-N} \xi)$ of left $\gr(u_q)$-comodule algebras sending the generator $W$ to $\beta Y$. Similarly, if $\beta = 0$, then there is an isomorphism $\mathscr{L}_4(\alpha, 0; \xi) \cong \mathcal{L}_1(1; \alpha^{-N} \xi)$. Hence, when either $\alpha$ or $\beta$ is zero, Lemma \ref{lem:cocycle-deform-L4} follows from Lemma~\ref{lem:cocycle-deform-L0123}.

In what follows, we assume that both $\alpha$ and $\beta$ are non-zero. We first embed $\mathscr{A}_4(\alpha, \beta; \xi)$ into $u_q$ as a left $u_q$-comodule subalgebra. For this purpose, we note that there are elements $u$ and $v$ of $\bfk$ satisfying
\begin{equation}
  \label{eq:cocycle-deform-L4-def-u-v}
  u^N + v^N = \xi \quad \text{and} \quad
  u v = \alpha \beta (1 - q^2)^{-1}.
\end{equation}
Indeed, let $s$ and $t$ be the solutions of the quadratic equation $T^2 - \xi T + \zeta^N = 0$, where $\zeta = \alpha \beta (1 - q^2)^{-1}$, and let $v$ and $w$ be $N$-th roots of $s$ and $t$, respectively. Then, since $(v w)^N = s t = \zeta^N$, we have $v w = q^k \zeta$ for some $k \in \mathbb{Z}$. The elements $u := q^{-k} w$ and $v$ of $\bfk$ satisfy \eqref{eq:cocycle-deform-L4-def-u-v}.

\begin{claim}
  \label{claim:A4-embedding-into-uq}
  We choose $u, v \in \bfk$ satisfying~\eqref{eq:cocycle-deform-L4-def-u-v}. Then there is an injective homomorphism $f : \mathscr{A}_4(\alpha, \beta; \xi) \to u_q$ of left $u_q$-comodule algebras such that
  \begin{equation}
    f(W) = \alpha \tilde{E} + \beta F + (u + v) K^{-1}.
  \end{equation}
\end{claim}
\begin{proof}
  There is an injective homomorphism
  \begin{equation*}
    \mathscr{L}_4(\alpha, \beta; \xi) \to \mathscr{L}_3(1; (u/\alpha)^N, (v/\beta)^N),
    \quad W \mapsto \alpha X + \beta Y
  \end{equation*}
  of left $\gr(u_q)$-comodule algebras. We denote by
  \begin{equation*}
    f_1 : \mathscr{A}_4(\alpha, \beta; \xi) \to \mathscr{A}_3(1; (u/\alpha)^N, (v/\beta)^N)
  \end{equation*}
  the induced homomorphism of left $u_q$-comodule algebras. We have already known that the target of $f_1$ is generated by $X$ and $Y$ subject to the relations $X^N = (u/\alpha)^N$, $Y^N = (v/\beta)^N$ and $X Y - q^2 Y X = 1$. Thus there is the following algebra map:
  \begin{equation*}
    f_2 : \mathscr{A}_3(1; (u/\alpha)^N, (v/\beta)^N) \to \bfk,
    \quad X \mapsto u/\alpha,
    \quad Y \mapsto v/\beta.
  \end{equation*}
  By Lemma~\ref{lem:coideal-sub-criterion}, we obtain an injective homomorphism
  \begin{equation*}
    f = (\id_{u_q} \otimes f_2 f_1) \circ \delta_{\mathscr{A}_4(\alpha, \beta; \xi)}:
    \mathscr{A}_4(\alpha, \beta; \xi) \to u_q
  \end{equation*}
  of left $u_q$-comodule algebras, which meets the requirements.
\end{proof}

By the embedding given by Claim~\ref{claim:A4-embedding-into-uq}, we prove:

\begin{claim}
  \label{claim:A4-generator}
  $\mathscr{A}_4(\alpha, \beta; \xi)$ is generated by $W$ as an algebra.
\end{claim}
\begin{proof}
  Let $\mathcal{A}'$ be the subalgebra of $\mathcal{A} = \mathscr{A}_4(\alpha, \beta; \xi)$ generated by $W$.
  We choose $u, v \in \bfk$ satisfying \eqref{eq:cocycle-deform-L4-def-u-v} and set $W' = \alpha \tilde{E} + \beta F + (u + v) K^{-1}$.
  By Claim~\ref{claim:A4-embedding-into-uq}, $\mathcal{A}'$ is isomorphic to the subalgebra of $u_q$ generated by $W'$.
  When we introduce the lexicographic order $\tilde{E} > F > K$ for the basis \eqref{eq:uq-modified-PBW-basis} of $u_q$, the element $W'{}^i$ for $0 \le i < N$ is equal to $\alpha^i \tilde{E}{}^i$ plus lower terms. Thus the set $\{ W'{}^i \mid 0 \le i < N \}$ is linearly independent. Hence we have $N \le \dim_{\bfk} \langle W' \rangle = \dim_{\bfk} \mathcal{A}' \le \dim_{\bfk} \mathcal{A} = N$, which implies $\mathcal{A}' = \mathcal{A}$. The proof is done.
\end{proof}

By Claims~\ref{claim:A4-embedding-into-uq} and~\ref{claim:A4-generator}, the minimal polynomial of $W \in \mathscr{A}_4(\alpha, \beta; \xi)$ is the same as that of $\alpha \tilde{E} + \beta F + (u + v) K^{-1}$, where $u, v \in \bfk$ satisfy \eqref{eq:cocycle-deform-L4-def-u-v}. Generalizing the problem slightly, we determine the minimal polynomial of $W_{\alpha, \beta, \gamma} := \alpha \tilde{E} + \beta F + \gamma K^{-1}$ for $\gamma \in \bfk$ as follows:

\begin{claim}
  \label{claim:A4-min-pol}
  The monic minimal polynomial of $W_{\alpha, \beta, \gamma}$ is
  \begin{equation}
    \label{eq:claim:A4-min-pol}
    D_N(t, \alpha \beta (1-q^2)^{-1}) - D_N(\gamma, \alpha \beta (1 - q^2)^{-1}).
  \end{equation}
\end{claim}
\begin{proof}
  This lemma is proved by the same idea as in \cite{2024arXiv241010064S}.
  Let $\Psi(t)$ be the monic minimal polynomial of $W_{\alpha, \beta, \gamma}$, and let $\Phi(t)$ be the polynomial \eqref{eq:claim:A4-min-pol}. Our goal is to show $\Psi(t) = \Phi(t)$.

  It is obvious that $\deg \Phi(t) = N$.
  The degree of $\Psi(t)$ is also $N$.
  Indeed, we choose $u, v \in \bfk$ satisfying $u + v = \gamma$ and $u v = \alpha \beta (1 - q^2)^{-1}$, and set $\xi = u^N + v^N$. Then, since $u$ and $v$ satisfy~\eqref{eq:cocycle-deform-L4-def-u-v}, Claim~\ref{claim:A4-embedding-into-uq} gives an injective homomorphism $f : \mathscr{A}_4(\alpha, \beta; \xi) \to u_q$ sending $W$ to $W_{\alpha, \beta, \gamma}$. Claim~\ref{claim:A4-generator} implies that the degree of the minimal polynomial of $W_{\alpha, \beta, \gamma}$ is equal to the dimension of $\mathscr{A}_4(\alpha, \beta; \xi)$, which is $N$.

  Now we set $C := \alpha \tilde{E} + u K^{-1}$ and $D := \beta F + v K^{-1}$. Then we have $W_{\alpha, \beta, \gamma} = C + D$, $C^N = u^N$, $D^N = v^N$ and $C D - q^2 D C = \alpha \beta$. Moreover, the latter three equations are defining relations of the subalgebra $\langle C, D \rangle$ as it can be verified in a similar way as \cite[Theorem 4.6]{2024arXiv241010064S}.
  Hence we have algebra maps $\chi_k : \langle C, D \rangle \to \bfk$ ($k \in \mathbb{Z}$) such that $\chi_k(C) = q^{2 k} u$ and $\chi_k(D) = q^{-2 k} v$. By restricting $\chi_k$ to the subalgebra $\langle W_{\alpha, \beta, \gamma} \rangle$, we obtain the following algebra maps:
  \begin{equation}
    \label{eq:A4-one-dim-reps}
    \langle W_{\alpha, \beta, \gamma} \rangle \to \bfk,
    \quad W_{\alpha, \beta, \gamma} \mapsto \mu_k := u q^{2 k} + v q^{-2 k}
    \quad (k \in \mathbb{Z}).
  \end{equation}

  It is obvious that $\mu_i = \mu_j$ if $i \equiv j \pmod{N}$.
  We suppose that there exist integers $i$ and $j$ such that $i \not \equiv j \pmod{N}$ and $\mu_i = \mu_j$. Then we have
  \begin{equation*}
    u - q^{-2i-2j} v = (q^{2i} - q^{2j})^{-1} (\mu_i - \mu_j) = 0
  \end{equation*}
  and therefore $u^N = v^N$. Since $u + v = \gamma$ and $u v = \alpha \beta (1 - q^2)^{-1}$, we have
  \begin{gather*}
    D_N(\gamma, \alpha \beta (1-q^2)^{-1})^2
    = D_N(u + v, u v)^2
    \mathbin{\mathop{=}^{\eqref{eq:Dickson-poly-2}}}
    (u^N + v^N)^2 \\
    = u^{2N} + 2 u^N v^N + v^{2N} = 4 u^N v^N
    = 4 \alpha^N \beta^N (1-q^2)^{-N}.
  \end{gather*}

  Now we verify the claim under the assumption that {\em the equation}
  \begin{equation}
    \label{eq:A4-multiple-roots-condition}
    D_N(\gamma, \alpha \beta (1-q^2)^{-1})^2
    - 4 \alpha^N \beta^N (1-q^2)^{-N} = 0
  \end{equation}
  {\em does not hold}. The fact that \eqref{eq:A4-one-dim-reps} is a well-defined algebra map implies that the minimal polynomial of $W_{\alpha, \beta, \gamma}$ has $\mu_k$ ($k \in \mathbb{Z}$) as its roots. The discussion of the previous paragraph shows that $\mu_0, \cdots, \mu_{N - 1}$ are distinct. Since $\Psi(T)$ is a monic polynomial of degree $N$, we have
  \begin{equation}
    \label{eq:A4-min-pol-decomposition}
    \Psi(t)
    = \prod_{k = 0}^{N-1} (t - \mu_k)
    = \prod_{k = 0}^{N-1} (t - (u q^k + v q^{-k})) = \Phi(t),
  \end{equation}
  where the last equality follows from \eqref{eq:Dickson-poly-1}.

  Finally, we discuss the general case where the equation \eqref{eq:A4-multiple-roots-condition} may hold. Using a finite-dimensional faithful representation of $u_q$, we regard $u_q$ as a subalgebra of a matrix algebra of sufficiently large degree. Then the equation
  \begin{equation}
    \label{eq:A4-min-pol-proof-3}
    \Phi(W_{\alpha, \beta, \gamma}) = 0
  \end{equation}
  is viewed as a system of polynomial equations with variables $\alpha$, $\beta$ and $\gamma$. By the explicit formula \eqref{eq:Dickson-poly-formula} of the Dickson polynomial, the left hand side of \eqref{eq:A4-multiple-roots-condition} is a polynomial of degree $2 N$ in $\gamma$ and, in particular, it is a non-zero polynomial in $\alpha$, $\beta$ and $\gamma$.
  We have already verified \eqref{eq:A4-min-pol-proof-3} provided that $\alpha$, $\beta$ and $\gamma$ do not satisfy the polynomial equation \eqref{eq:A4-multiple-roots-condition}.
  Thus the equation \eqref{eq:A4-min-pol-proof-3} actually holds for any $\alpha, \beta, \gamma \in \bfk$. Hence $\Psi(t)$ divides $\Phi(t)$. Since both $\Psi(t)$ and $\Phi(t)$ are monic polynomials of the same degree, we conclude $\Psi(t) = \Phi(t)$. The proof is done.
\end{proof}

\begin{proof}[Proof of Lemma~\ref{lem:cocycle-deform-L4}]
  As we have explained after Claim~\ref{claim:A4-generator}, the minimal polynomial of $W \in \mathscr{A}_4(\alpha, \beta; \xi)$ is equal to that of $\alpha \tilde{E} + \beta F + (u + v) K^{-1}$, where $u, v \in \bfk$ satisfy \eqref{eq:cocycle-deform-L4-def-u-v}. Claim~\ref{claim:A4-min-pol} completes the proof.
\end{proof}

\begin{remark}
  \label{rem:A4-introduction}
  Since $\tilde{E} = (q-q^{-1}) K^{-1}E$, the left $u_q$-comodule algebra $\mathscr{A}$ in the Introduction is identified $\mathscr{A}_4(\tilde{\alpha}, \beta; \xi)$ with $\tilde{\alpha} = (q-q^{-1})^{-1} \alpha$. Therefore the monic minimal polynomial of $W \in \mathscr{A}$ is $\phi_{\tilde{\alpha}, \beta, \xi}(T)$, as we have announced in the Introduction.
\end{remark}

\begin{remark}
  \label{rem:A4-semisimplicity}
  Equation \eqref{eq:A4-min-pol-decomposition} in the proof of Claim \ref{claim:A4-min-pol} gives the decomposition of $\phi_{\alpha, \beta, \xi}(T)$. Namely, fixing $u, v \in \bfk$ satisfying \eqref{eq:cocycle-deform-L4-def-u-v}, we have
  \begin{equation*}
    \phi_{\alpha, \beta, \xi}(T) = \prod_{k = 0}^{N-1} (T - (u q^k + v q^{-k})).
  \end{equation*}
  By an argument similar to the proof of Claim \ref{claim:A4-min-pol}, we see that $\phi_{\alpha, \beta, \xi}(t)$ has multiple roots if and only if the following equation holds:
  \begin{equation}
    \label{eq:A4-non-semisimple-condition}
    \xi^2 = 4 \alpha^N \beta^N (1-q^2)^{-N}.
  \end{equation}
  This means that $\mathscr{A}_4(\alpha, \beta; \xi)$ is semisimple if and only if \eqref{eq:A4-non-semisimple-condition} does not hold. On the other hand, $\mathscr{L}_4(\alpha, \beta; \xi)$ is semisimple if and only if $\xi \ne 0$. Thus we have examples of semisimple left $\gr(u_q)$-comodule algebras whose cocycle deformations are not semisimple ({\it cf}. Theorem~\ref{thm:cocycle-deform-semisimplicity}).
\end{remark}

\subsection{Conclusion}

Thus, we present the following conclusion based on our findings.
In Step 4 (Subsection \ref{subsec:Step-4}), we have obtained the following families of left $u_q$-comodule algebras: Let $r \in \mathrm{Div}(N)$ and $\alpha, \beta, \eta, \xi, \zeta \in \bfk$ with $(\alpha, \beta) \ne (0, 0)$ be parameters. Then,
\begin{enumerate}
  \setcounter{enumi}{-1}
\item The algebra $\mathscr{A}_0(r)$ is generated by $G$ subject to
  \begin{equation*}
    G^r = 1.
  \end{equation*}
\item The algebra $\mathscr{A}_1(r; \xi)$ is generated by $G$ and $X$ subject to
  \begin{equation*}
    G^r = 1, \quad X^N = \xi, \quad G X = q^{2N/r} X G.
  \end{equation*}
\item The algebra $\mathscr{A}_2(r; \zeta)$ is generated by $G$ and $Y$ subject to
  \begin{equation*}
    G^r = 1, \quad Y^N = \zeta, \quad G Y = q^{-2N/r} Y G.
  \end{equation*}
\item The algebra $\mathscr{A}_3(r; \xi, \zeta)$ is generated by $G$, $X$ and $Y$ subject to the relations for $\mathscr{A}_1(r; \xi)$ and $\mathscr{A}_2(r; \zeta)$, and
  \begin{equation*}
    X Y - q^2 Y X = 1.
  \end{equation*}
\item[$(3')$] The algebra $\mathscr{A}_3(N; \xi, \zeta, \eta)$ is generated by $G$, $X$ and $Y$ subject to the same relations as $\mathscr{A}_3(N; \xi, \zeta)$ but with the last one replaced with
  \begin{equation*}
    X Y - q^2 Y X = 1 - \eta G^{-2}.
  \end{equation*}
\item The algebra $\mathscr{A}_4(\alpha, \beta; \xi)$ is generated by $W$ subject to $\phi_{\alpha, \beta, \xi}(W) = 0$, where $\phi_{\alpha, \beta, \xi}(T)$ is the polynomial defined by \eqref{eq:cocycle-deform-L4-min-pol}.
\end{enumerate}
The left $u_q$-coaction $\delta$ of each of them is given by
\begin{equation*}
  % \label{eq:uq-coaction}
  \begin{gathered}
    \delta(X) = \tilde{E} \otimes 1 + K^{-1} \otimes X,
    \quad \delta(Y) = F \otimes 1 + K^{-1} \otimes Y, \\
    \delta(W) = (\alpha \tilde{E} + \beta F) \otimes 1 + K^{-1} \otimes W,
    \quad \delta(G) = K^{N/r} \otimes G
  \end{gathered}
\end{equation*}
on the generators, where $\tilde{E} = (q-q^{-1}) K^{-1} E$, and $r$ is taken to be $N$ for the case of the algebra $\mathscr{A}_3(N; \xi, \zeta, \eta)$. For notational convenience, we set
\begin{equation*}
  \mathscr{A}_3(r; \xi, \zeta, 0) := \mathscr{A}_3(r; \xi, \zeta)
\end{equation*}
for $r \in \mathrm{Div}(N)$ with $r < N$. Now we define
\begin{align*}
  \mathscr{F}_0
  & = \{ \mathscr{A}_0(r) \mid r \in \mathrm{Div}(N) \}, \\
  \mathscr{F}_1
  & = \{ \mathscr{A}_1(r; \xi) \mid r \in \mathrm{Div}(N), r \ne 1; \xi \in \bfk \}, \\
  \mathscr{F}_2
  & = \{ \mathscr{A}_2(r; \xi) \mid r \in \mathrm{Div}(N), r \ne 1; \xi \in \bfk \}, \\
  \mathscr{F}_3
  & = \{ \mathscr{A}_3(r; \xi, \zeta, \eta)
    \mid r \in \mathrm{Div}(N); \xi, \zeta, \eta \in \bfk; \text{$\eta = 0$ if $r < N$} \}, \\
  \mathscr{F}_4
  & = \{ \mathscr{A}_4(\alpha, \beta; \xi)
    \mid \alpha, \beta, \xi \in \bfk, (\alpha, \beta) \ne (0,0) \},
\end{align*}
call these sets {\em families}, and let $\mathscr{F}$ be the union of the families:
\begin{equation*}
  \mathscr{F} := \mathscr{F}_0 \cup \mathscr{F}_1 \cup \mathscr{F}_2 \cup \mathscr{F}_3 \cup \mathscr{F}_4.
\end{equation*}
By Lemma~\ref{lem:cocycle-deform-H-Morita} and the conclusion of Step 3 (Subsection~\ref{subsec:Step-3}), we have:

\begin{theorem}
  \label{thm:conclusion}
  An indecomposable exact module category over $\Rep(u_q)$ is equivalent to $\Rep(A)$ for some $A \in \mathscr{F}$. Two elements of $\mathscr{F}$ belonging to different families are not $u_q$-Morita equivalent.
  Regarding $u_q$-Morita equivalence of left $u_q$-comodule algebras of the same family, we have:
  \begin{enumerate}
    \setcounter{enumi}{-1}
  \item $\mathscr{A}_0(r), \mathscr{A}_0(r') \in \mathscr{F}_0$ are $u_q$-Morita equivalent if and only if
    \begin{equation*}
      r = r'.
    \end{equation*}
  \item $\mathscr{A}_1(r; \xi), \mathscr{A}_1(r'; \xi') \in \mathscr{F}_1$ are $u_q$-Morita equivalent if and only if
    \begin{equation*}
      (r, \xi) = (r', \xi').
    \end{equation*}
  \item $\mathscr{A}_2(r; \zeta), \mathscr{A}_2(r'; \zeta') \in \mathscr{F}_2$ are $u_q$-Morita equivalent if and only if
    \begin{equation*}
      (r, \zeta) = (r', \zeta').
    \end{equation*}
  \item $\mathscr{A}_3(r; \xi, \zeta, \eta), \mathscr{A}_3(r'; \xi', \zeta', \eta') \in \mathscr{F}_3$ are $u_q$-Morita equivalent if and only if there exists an integer $k \in \mathbb{Z}$ such that
    \begin{equation*}
      (r', \xi', \zeta', \eta') = (r, \xi, \zeta, q^{2k} \eta).
    \end{equation*}
  \item $\mathscr{A}_4(\alpha, \beta; \xi), \mathscr{A}_4(\alpha', \beta'; \xi') \in \mathscr{F}_4$ are $u_q$-Morita equivalent if and only if there exist $k \in \mathbb{Z}$ and $\lambda \in \bfk^{\times}$ such that
    \begin{equation*}
      (\alpha', \beta', \xi') = (\lambda q^{2k} \alpha, \lambda q^{-2k} \beta, \lambda^N \xi).
    \end{equation*}
  \end{enumerate}
\end{theorem}

\subsection{Remarks}

(1) In recent developments of finite tensor categories and their modules, the effectiveness of the techniques of the relative Serre functor is recognized; see, {\it e.g.}, \cite{MR3435098,MR4042867,MR4586249,MR4624433}. As natural generalizations of the notions of pivotal finite tensor categories, pivotal exact module categories are introduced in terms of the relative Serre functor \cite{MR3435098,2022arXiv220707031F,MR4624433}. In a forthcoming paper, we will compute the relative Serre functors of indecomposable exact module categories over $\Rep(u_q(\mathfrak{sl}_2))$. Their pivotality will also be addressed.

(2) Simple objects of indecomposable exact module categories over $\Rep(u_q(\mathfrak{sl}_2))$ will also be classified in our forthcoming paper. Together with (1), this leads to the classification of division algebras, Frobenius division algebras and symmetric Frobenius division algebras in $\Rep(u_q(\mathfrak{sl}_2))$ in terms of internal endomorphism algebras.

(3) It is known that the group of tensor autoequivalences of $\Rep(u_q(\mathfrak{sl}_2))$ is isomorphic to the projective special linear group $PSL_2$ of degree 2 \cite{MR3530496,MR3890207,MR3937297}. Hence $PSL_2$ acts on the set of equivalence classes of indecomposable exact module categories over $\Rep(u_q(\mathfrak{sl}_2))$. Cris Negron posed the question of whether the number of the orbits is finite. This question will also be addressed in a forthcoming paper.

(4) In this paper, we have only considered the case where $q$ is a root of unity of odd order. The even order case may also be addressed in a similar manner. An interesting relevant problem is the classification of indecomposable exact module categories over $\Rep(u_q^{\phi}(\mathfrak{sl}_2))$, where $u_q^{\phi}(\mathfrak{sl}_2)$ is the modification of $u_q(\mathfrak{sl}_2)$ at a root of unity $q$ of even order discussed in \cite{MR3608158,2018arXiv180902116G,MR4082225,MR4227163} in connection with logarithmic conformal field theory. Since $u_q^{\phi}(\mathfrak{sl}_2)$ is no longer a Hopf algebra (but a quasi-Hopf algebra), most of our methods cannot be applied to $u_q^{\phi}(\mathfrak{sl}_2)$ at least in a direct way. Nevertheless, we hope that our results will provide a useful model for studying the case of $u_q^{\phi}(\mathfrak{sl}_2)$.

\appendix

\section{Equivariant Eilenberg-Watts theorem}
\label{appendix:equiv-EW}

\newcommand{\AllMod}{\tilde{\Mod}}

The base field $\bfk$ is arbitrary.
In Appendix \ref{appendix:equiv-EW}, we give a proof of the equivariant Eilenberg-Watts theorem (Theorem~\ref{thm:equivariant-EW}) due to Andruskiewitsch and Mombelli in a generalized form. They only considered finite-dimensional Hopf algebras in \cite{MR2331768}, however, we point out that a crucial part of the argument works for any bialgebras. For the clarification, we first introduce some terminology:

\begin{definition}
  A {\em $\bfk$-linear monoidal category} is a $\bfk$-linear category $\mathcal{C}$ endowed with a structure of a monoidal category $(\mathcal{C}, \otimes, \unitobj)$ such that the monoidal product $\otimes$ is $\bfk$-bilinear. Given a $\bfk$-linear monoidal category $\mathcal{C}$, a {\em $\bfk$-linear left $\mathcal{C}$-module category} is a $\bfk$-linear category $\mathcal{M}$ endowed with a structure of a left $\mathcal{C}$-module category such that the action $\catactl : \mathcal{C} \times \mathcal{M} \to \mathcal{M}$ is $\bfk$-bilinear.
\end{definition}

\begin{definition}
  Let $\mathcal{C}$ be a monoidal category, and let $\mathcal{M}$ and $\mathcal{N}$ be left $\mathcal{C}$-module categories. An {\em oplax left $\mathcal{C}$-module functor} from $\mathcal{M}$ to $\mathcal{N}$ is a functor $F: \mathcal{M} \to \mathcal{N}$ equipped with a natural transformation
  \begin{equation*}
    \xi^{(F)}_{X,M} : F(X \catactl M) \to X \catactl F(M)
    \quad (X \in \mathcal{C}, M \in \mathcal{M})
  \end{equation*}
  satisfying $\xi^{(F)}_{X \otimes Y,M} = (\id_X \catactl \xi^{(F)}_{Y,M}) \circ \xi^{(F)}_{X, Y \catactl M}$ and $\xi^{(F)}_{\unitobj, M} = \id_{F(M)}$ for all objects $X, Y \in \mathcal{C}$ and $M \in \mathcal{M}$. Given two left $\mathcal{C}$-module functors $F$ and $G$ from $\mathcal{M}$ to $\mathcal{N}$, a {\em morphism} from $F$ to $G$ is a natural transformation $\alpha : F \to G$ satisfying
  \begin{equation*}
    (\id_X \catactl \alpha_M) \circ \xi^{(F)}_{X,M}
    = \xi^{(G)}_{X,M} \circ \alpha_{X \catactl M}
  \end{equation*}
  for all $X \in \mathcal{C}$ and $M \in \mathcal{M}$.
\end{definition}

\begin{definition}
  A {\em strong left $\mathcal{C}$-module functor} is an oplax left $\mathcal{C}$-module functor with an invertible structure morphism.
\end{definition}

Unlike in Section~\ref{sec:exact-comod-alg-over-pointed}, (co)algebras and (co)modules are not assumed to be finite-dimensional in this section. We have used the symbol $\Mod$ for categories of finite-dimensional modules or comodules.
We use the same notation as in Section~\ref{sec:exact-comod-alg-over-pointed}, but with $\Mod$ replaced with $\AllMod$, for the corresponding categories without a restriction on dimensions. For example, ${}_A\AllMod$ for an algebra $A$ is the category of all left $A$-modules.

Now let $H$ be a bialgebra. Then the category $\mathcal{C} := {}_H\AllMod$ is a $\bfk$-linear monoidal category. If $A$ is a left $H$-comodule algebra, then ${}_A\AllMod$ is a $\bfk$-linear left module category over $\mathcal{C}$ in the same way as in the case of finite-dimensional Hopf algebras.

Given $\bfk$-linear left $\mathcal{C}$-module categories $\mathcal{M}$ and $\mathcal{N}$, we denote by $\mathscr{F}_{\mathcal{C}}^{\mathrm{oplax}}(\mathcal{M}, \mathcal{N})$ the category of $\bfk$-linear oplax left $\mathcal{C}$-module functors from $\mathcal{M}$ to $\mathcal{N}$ whose underlying functor admits a right adjoint. If $A$ and $B$ are left $H$-comodule algebras and $P$ is an object of the category ${}^H_A\AllMod_B^{}$ of $A$-$B$-bimodules in ${}^H\AllMod$, then the functor $F := P \otimes_B (-)$ is an object of $\mathscr{F}_{\mathcal{C}}^{\mathrm{oplax}}({}_B\AllMod, {}_A\AllMod)$ together with the natural transformation
\begin{equation}
  \label{eq:Appendix-oplax-module-structure}
  \begin{gathered}
    F(X \catactl M) = P \otimes_B (X \catactl M)
    \to X \catactl (P \otimes_B M) = X \catactl F(M), \\
    p \otimes_B (x \otimes m)
    \mapsto p_{(-1)} x \otimes (p_{(0)} \otimes_B m),
  \end{gathered}
\end{equation}
where $p \in P$, $x \in X \in \mathcal{C}$ and $m \in M \in {}_B\AllMod$.

\begin{theorem}
  \label{thm:Appendix-equiv-EW}
  For left $H$-comodule algebras $A$ and $B$, the functor
  \begin{equation}
    \label{eq:Appendix-equiv-EW}
    {}_A^{H}\AllMod_B^{} \to   \mathscr{F}_{\mathcal{C}}^{\mathrm{oplax}}({}_B\AllMod, {}_A\AllMod),
    \quad M \mapsto M \otimes_B (-)
  \end{equation}
  is an equivalence.
\end{theorem}
\begin{proof}
  Although our argument is the same as the proof of \cite[Proposition 1.23]{MR2331768}, we present the details to demonstrate that we do not require an antipode of $H$.
  We fix $P \in {}_A\AllMod_B$ and let $F: {}_B\AllMod \to {}_A\AllMod$ be the functor given by tensoring $P$ over $B$. We first show that there is a bijection between the set of linear maps $\delta : P \to H \otimes P$ making $P$ an object of ${}_A^{H}\AllMod_B^{}$ and the set of natural transformations
  \begin{equation*}
    \xi_{X,M} : F(X \otimes M) \to X \otimes F(M) \quad (X \in \mathcal{C}, M \in {}_B\AllMod)
  \end{equation*}
  making $F$ an oplax left $\mathcal{C}$-module functor. A construction of $\xi$ from $\delta$ has been given above. The inverse construction is as follows: Suppose that we are given a natural transformation $\xi$ making $F$ an oplax left $\mathcal{C}$-module functor. Under the canonical identification $P \otimes_B B \cong P$, we define
  \begin{equation*}
    \delta_{\xi}(p) = \xi_{H,B}(p \otimes_B (1_H \otimes 1_B))
    \in H \otimes (P \otimes_B B) = H \otimes P
  \end{equation*}
  for $p \in P$ and write it as $\delta_{\xi}(p) = p_{(-1)} \otimes p_{(0)}$ although we have not yet proved that $\delta_{\xi}$ is indeed a left coaction of $H$. For all $X \in \mathcal{C}$ and $M \in {}_B\AllMod$, the natural transformation $\xi_{X,M}$ is given by the same formula as \eqref{eq:Appendix-oplax-module-structure}. Indeed, for $x \in X$, there is a unique left $H$-module map $\phi_x : H \to X$ such that $\phi_x(1_H) = x$. For $m \in M$, there is also a unique left $B$-module map $\psi_m : B \to M$ such that $\psi_m(1_B) = m$. By the naturality of $\xi$, we have
  \begin{align*}
    \xi_{X,M}(p \otimes_B (x \otimes m))
    & = \xi_{X,M}(p \otimes_B (\phi_x(1_H) \otimes \psi_m(1_B))) \\
    & = (\phi_x \otimes (\id_P \otimes_B \psi_m)) \xi_{H,B}(p \otimes_B (1_H \otimes 1_B)) \\
    & = (\phi_x \otimes (\id_P \otimes_B \psi_m)) (p_{(-1)} \otimes p_{(0)} \otimes_B 1_B) \\
    & = p_{(-1)} x \otimes (p_{(0)} \otimes_B m)
  \end{align*}
  for $p \in P$, $x \in X$ and $m \in M$. Now it is straightforward to show that $\delta_{\xi}$ makes $P$ an object of ${}^H_A\AllMod_B^{}$ by translating the axioms of oplax left $\mathcal{C}$-module functors. It is obvious that the construction $\xi \mapsto \delta_{\xi}$ is the inverse of the construction of an oplax left $\mathcal{C}$-module structure from a left coaction of $H$ mentioned at the beginning of the proof.

  We shall complete the proof of this theorem. By the Eilenberg-Watts theorem, a $\bfk$-linear functor ${}_B\AllMod \to {}_A\AllMod$ admitting a right adjoint is isomorphic to $P \otimes_B (-)$ for some $A$-$B$-bimodule $P$. The above discussion means that the functor~\eqref{eq:Appendix-equiv-EW} is essentially surjective. By the definition of morphisms of oplax left $\mathcal{C}$-module functors, it is easy to check that \eqref{eq:Appendix-equiv-EW} is fully faithful. The proof is done.
\end{proof}

An {\em opantipode} of a bialgebra $H$ is a linear map $\varsigma : H \to H$ such that
\begin{equation*}
  \varsigma(h_{(2)}) h_{(1)} = \varepsilon(h) 1_H = h_{(2)} \varsigma(h_{(1)})
\end{equation*}
for all $h \in H$. If $H$ has a bijective antipode, then the inverse of the antipode of $H$ is an opantipode of $H$.

\begin{theorem}
  \label{thm:Appendix-equiv-EW-opantipode}
  If the bialgebra $H$ has an opantipode $\varsigma$, then we have
  \begin{equation*}
    \mathscr{F}_{\mathcal{C}}^{\mathrm{strong}}({}_B\AllMod, {}_A\AllMod)
    = \mathscr{F}_{\mathcal{C}}^{\mathrm{oplax}}({}_B\AllMod, {}_A\AllMod),
  \end{equation*}
  where the left-hand side is the full subcategory of the right-hand side consisting of strong left $\mathcal{C}$-module functors.
\end{theorem}
\begin{proof}
  Indeed, the natural transformation \eqref{eq:Appendix-oplax-module-structure} has the inverse
  \begin{equation*}
    x \otimes (p \otimes_B m)
    \mapsto p_{(0)} \otimes_B (\varsigma(p_{(-1)}) x \otimes m).
    \qedhere
  \end{equation*}
\end{proof}

\begin{proof}[Proof of Theorem~\ref{thm:equivariant-EW}]
  Let $H$ be a finite-dimensional bialgebra, and let $A$ and $B$ be finite-dimensional left $H$-comodule algebras. We recall the fact that a $\bfk$-linear functor between finite abelian categories is right exact if and only if it has a right adjoint. By the same argument as Theorem~\ref{thm:Appendix-equiv-EW}, we have an equivalence
  \begin{equation}
    \label{eq:Appendix-equivariant-EW-fd}
    {}_A^H\Mod_B^{} \approx \Rex_{\mathcal{C}}^{\mathrm{oplax}}({}_B\Mod, {}_A\Mod),
    \quad M \mapsto M \otimes_B (-)
  \end{equation}
  of $\bfk$-linear categories, where $\mathcal{C}$ is ${}_H\Mod$ and the right hand side is the category of $\bfk$-linear right exact oplax left $\mathcal{C}$-module functors from ${}_B\Mod$ to ${}_A\Mod$.

  Now we suppose that $H$ is a finite-dimensional Hopf algebra. Then, by the theorem of Larson and Sweedler, the antipode of $H$ is invertible. By the same argument as Theorem~\ref{thm:Appendix-equiv-EW-opantipode}, we now have that the target of \eqref{eq:Appendix-equivariant-EW-fd} is equal to the category of $\bfk$-linear right exact strong left $\mathcal{C}$-module functors. The proof is done.
\end{proof}

\section{Morita duality between ${}_H\Mod$ and ${}^H\Mod$}
\label{appendix:morita-duality}

The base field $\bfk$ is arbitrary.
Let $H$ be a finite-dimensional Hopf algebra.
The aim of Appendix \ref{appendix:morita-duality} is to complete the proof of the Morita duality between ${}_H\Mod$ and ${}^H\Mod$ (Theorem~\ref{thm:Morita-duality}). A key ingredient is {\em the fundamental theorem for Hopf modules} \cite[\S1.9]{MR1243637}, which states that the functor $\Phi : \Vect \to {}^H\Mod_H$ given by $M \mapsto H \otimes M$ is an equivalence. Here, the left coaction and the right action of $H$ on the vector space $\Phi(M)$ are given by
\begin{equation*}
  h \otimes m \mapsto h_{(1)} \otimes h_{(2)} \otimes m \quad \text{and} \quad
  (h \otimes m) \cdot h' = h h' \otimes m,
\end{equation*}
respectively, for $h, h' \in H$ and $m \in M$.

We need a slightly generalized version of the fundamental theorem involving a left $H$-comodule algebra. We note that the source and the target of $\Phi$ are left module categories over $\mathcal{D} := {}^H\Mod$. The functor $\Phi$ is in fact an equivalence of left $\mathcal{D}$-module categories by the natural transformation
\begin{equation*}
  \xi_{X,M} : X \otimes \Phi(M) \to \Phi(X \otimes M),
  \quad x \otimes (h \otimes m) \mapsto x_{(-1)} h \otimes (x_{(0)} \otimes m),
\end{equation*}
where $x \in X \in \mathcal{D}$, $m \in M \in \Vect$ and $h \in H$.

Given a finite left $\mathcal{D}$-module category $\mathcal{M}$ and an algebra $A$ in $\mathcal{D}$, a {\em left $A$-module in $\mathcal{M}$} is an object $M \in \mathcal{M}$ together with a morphism $A \otimes M \to M$ in $\mathcal{M}$ satisfying the same axioms as a left $A$-module in $\mathcal{D}$. It is obvious that an equivalence of left $\mathcal{D}$-module categories induces an equivalence between the categories of $A$-modules in them. Since the categories of left $A$-modules in $\Vect$ and in ${}^H\Mod_H$ are the same as ${}_A\Mod$ and ${}^H_A\Mod_H^{}$, respectively, the functor $\Phi$ induces an equivalence ${}_A\Mod \approx {}^H_A \Mod_H^{}$ of linear categories (as is well-known; see, {\it e.g.}, \cite{MR2238882}). For later use, we denote the equivalence so obtained by
\begin{equation}
  \label{eq:fundamental-thm-of-Hopf-modules}
  \HM(A) : {}_A\Mod \to {}^H_A\Mod_H^{},
  \quad M \mapsto \Phi(M).
\end{equation}
By definition, the left action of $A$ on $\Phi(M)$ is given by
\begin{equation*}
  A \otimes \Phi(M)
  \xrightarrow{\quad \xi_{A,M} \quad} \Phi(A \otimes M)
  \xrightarrow{\quad \Phi(a_M) \quad} \Phi(M),
\end{equation*}
where $a_M : A \otimes M \to M$ is the action of $A$ on $M$. Namely,
\begin{equation*}
  a \cdot (h \otimes m) = a_{(-1)} h \otimes a_{(0)} m
  \quad (a \in A, h \in H, m \in M).
\end{equation*}

\begin{proof}[Proof of Theorem~\ref{thm:Morita-duality}]
  Let $H$ be a Hopf algebra, and let $\mathcal{M}$ be a finite left module category over $\mathcal{C} := {}_H\Mod$. Since $\phi_{\mathcal{M}}$ is natural in $\mathcal{M}$, it suffices to consider the case where $\mathcal{M} = {}_A\Mod$ for some left $H$-comodule algebra $A$. We write $\mathcal{D} = {}^H\Mod$. Then there are equivalences
  \begin{gather*}
    \mathcal{M} = {}_A\Mod
    \xrightarrow{\quad \HM(A) \quad}
    {}^H_A\Mod_H^{} \xrightarrow{\quad \widetilde{\EW}_{\mathcal{D}}(A, H) \quad}
    \Rex_{\mathcal{D}}({}^H\Mod_A, {}^H\Mod_H) \\
    \xrightarrow{\quad \Rex_{\mathcal{D}}(\EW(\bfk, A), \HM(\bfk))^{-1} \quad}
    \Rex_{\mathcal{D}}(\Rex_{\mathcal{C}}({}_A\Mod, \Vect), \Vect) = \mathcal{M}^{**}
  \end{gather*}
  of categories, where the functor
  \begin{equation*}
    \widetilde{\EW}_{\mathcal{D}}(A, H) : {}_A\mathcal{D}_H \to \Rex_{\mathcal{D}}(\mathcal{D}_A, \mathcal{D}_H),
    \quad M \mapsto (-) \otimes_A M
  \end{equation*}
  is the Eilenberg-Watts equivalence in $\mathcal{D}$.
  We prove this theorem by showing that the equivalence $\mathcal{M} \to \mathcal{M}^{**}$ obtained by composing the above equivalences is isomorphic to $\phi_{\mathcal{M}}$. To achieve this, it suffices to show that there is an isomorphism
  \begin{equation*}
    \widetilde{\EW}_{\mathcal{D}}(A, H) \circ \HM(A)
    \cong \Rex_{\mathcal{D}}(\EW(\bfk, A), \HM(\bfk)) \circ \phi_{\mathcal{M}}.
  \end{equation*}
  Now let $L$ and $R$ be the left- and the right-hand sides, respectively.
  For $M \in \mathcal{M}$ and $N \in {}^H\Mod_A$, we have $L(M)(N) = N \otimes_A (H \otimes M)$, where the left coaction and the right action of $H$ are given by
  \begin{gather*}
    n \otimes_A (h \otimes m) \mapsto n_{(-1)} h_{(1)} \otimes (n_{(0)} \otimes_A (h_{(2)} \otimes m)) \\
    \text{and} \quad
    (n \otimes_A (h \otimes m)) \cdot h' = n \otimes_A (h h' \otimes m),
  \end{gather*}
  respectively, for $n \in N$, $h, h' \in H$ and $m \in M$. We also have
  \begin{equation*}
    R(M)(N)
    = (\HM(\bfk) \circ \phi_{\mathcal{M}}(M) \circ \EW(\bfk, A))(N)
    = H \otimes (N \otimes_A M),
  \end{equation*}
  where the left coaction and the right action of $H$ are given by
  \begin{equation*}
    h \otimes x \mapsto h_{(1)} \otimes h_{(2)} \otimes x
    \quad \text{and} \quad
    (h \otimes x) \cdot h' = h h' \otimes x,
  \end{equation*}
  respectively, for $h, h' \in H$ and $x \in N \otimes_A M$. Now we define
  \begin{equation*}
    f_{M,N} : L(M)(N) \to R(M)(N),
    \quad n \otimes_A (h \otimes m)
    \mapsto n_{(-1)} h \otimes (n_{(0)} \otimes_A m)
  \end{equation*}
  for $n \in N$, $h \in H$ and $m \in M$. It is straightforward to verify that $f_{M,N}$ is an isomorphism and natural in $M$ and $N$. Hence we have $L \cong R$ as functors.
\end{proof}

\section{A Hopf 2-cocycle turning $\gr(u_q(\mathfrak{sl}_2))$ into $u_q(\mathfrak{sl}_2)$}
\label{appendix:detail-Hopf-2-cocycle}

The base field $\bfk$ is assumed to be algebraically closed and of characteristic zero. We fix an odd integer $N > 1$ and a root of unity $q \in \bfk$ of order $N$.
In Appendix \ref{appendix:detail-Hopf-2-cocycle}, we give an explicit Hopf 2-cocycle $\sigma$ of $H := \gr(u_q(\mathfrak{sl}_2))$ using the method of Grunenfelder and Mastnak \cite{2010arXiv1010.4976G}. We then show that the Hopf 2-cocycle $\sigma$ deforms the Hopf algebra $H$ into $u_q(\mathfrak{sl}_2)$ as stated in Lemma \ref{lem:gr-uq-sl2-cocycle-deform}.

For simplicity of notation, we set $\mathbb{I} = \{ 0, 1, \cdots, N - 1 \}$.
As we have observed, the Hopf algebra $H$ is generated by the homogeneous elements $x$, $y$ and $g$ of degrees 1, 1 and 0, respectively, subject to the relations~\eqref{eq:gr-uq-sl2-relations}.
The comultiplication is given by \eqref{eq:gr-uq-sl2-comultiplication} on the generators. By the $q$-binomial formula, we have $\Delta(x^i y^j g^k)$
\begin{align*}
  & = \sum_{r = 0}^i \sum_{s = 0}^j \binom{i}{r}_{\!\!q^2} \binom{j}{s}_{\!\!q^2}
  (g^{-1} \otimes x)^{r} (x \otimes 1)^{i - r}
  (y \otimes 1)^{j-s} (g^{-1} \otimes y)^{s} (g^k \otimes g^k) \\
  & = \sum_{r = 0}^i \sum_{s = 0}^j \binom{i}{r}_{\!\!q^2} \binom{j}{s}_{\!\!q^2} q^{-2r(i - r) + 2 r(j - s)} x^{i-r} y^{j-s} g^{-r-s+k} \otimes x^r y^s g^k
\end{align*}
for all $i, j, k \in \mathbb{I}$.
To construct the Hopf 2-cocycle $\sigma$, we consider two algebra maps $\rho_1$ and $\rho_2$ from $H$ to the matrix algebra of degree 2 given by
\begin{gather*}
  \rho_1(g) = \begin{pmatrix} 1 & 0 \\ 0 & q^{-2} \end{pmatrix}, \quad
  \rho_1(x) = \begin{pmatrix} 0 & 1 \\ 0 & 0 \end{pmatrix}, \quad
  \rho_1(y) = \begin{pmatrix} 0 & 0 \\ 0 & 0 \end{pmatrix}, \\
  \rho_2(g) = \begin{pmatrix} 1 & 0 \\ 0 & q^{-2} \end{pmatrix}, \quad
  \rho_2(x) = \begin{pmatrix} 0 & 0 \\ 0 & 0 \end{pmatrix}, \quad
  \rho_2(y) = \begin{pmatrix} 0 & 0 \\ 1 & 0 \end{pmatrix}.
\end{gather*}
One can define $\alpha, \xi_1, \xi_2 \in H^*$ so that the equations
\begin{equation*}
  \rho_1(h) = \begin{pmatrix} \varepsilon(h) & \xi_1(h) \\ 0 & \alpha(h) \end{pmatrix},
  \quad
  \rho_2(h) = \begin{pmatrix} \varepsilon(h) & 0 \\ \xi_2(h) & \alpha(h) \end{pmatrix}
\end{equation*}
hold for all $h \in H$. Since $\rho_1$ and $\rho_2$ are algebra maps, we have
\begin{equation}
  \label{eq:appendix-xi-1-xi-2-comult}
  \Delta(\alpha) = \alpha \otimes \alpha, \quad
  \Delta(\xi_1) = \xi_1 \otimes \alpha + \varepsilon \otimes \xi_1, \quad
  \Delta(\xi_2) = \xi_2 \otimes \varepsilon + \alpha \otimes \xi_2,
\end{equation}
where $\Delta$ is the comultiplication of $H^*$. 
It is easy to verify
\begin{equation}
  \label{eq:appendix-alpha-xi-1-xi-2}
  \begin{aligned}
    \langle \xi_1, x^i y^j g^k \rangle & = \delta_{i,1} \delta_{j,0} q^{-2k},
    & \langle \alpha, x^i y^j g^k \rangle = \delta_{i,0} \delta_{j,0} q^{-2k}, \\
    \langle \xi_2, x^i y^j g^k \rangle & = \delta_{i,0} \delta_{j,1}
  \end{aligned}
\end{equation}
for $i, j, k \in \mathbb{I}$, where $\delta$ means the Kronecker delta. By induction on $m \in \mathbb{Z}_{\ge 0}$, one can also verify that the $m$-th powers of $\xi_1$ and $\xi_2$ with respect to the convolution product are given by
\begin{align}
  \label{eq:appendix-xi-1-power}
  \langle \xi_1^m, x^i y^j g^k \rangle
  & = \delta_{i,m} \delta_{j,0} (m)_{q^2}! \, q^{-2km}, \\
  \label{eq:appendix-xi-2-power}
  \langle \xi_2^m, x^i y^j g^k \rangle
  & = \delta_{i,0} \delta_{j,m} (m)_{q^2}!
  \quad (i, j, k \in \mathbb{I}).                                       
\end{align}
In particular, we have $\xi_1^N = \xi_2^N = 0$. Now we define
\begin{equation}
  \label{eq:appendix-gr-uq-sl2-cocycle-def}
  \sigma = \exp_{q^2}(\xi_1 \otimes \xi_2),
  \quad \text{where }
  \exp_{q^2}(X) = \sum_{r = 0}^{N-1} \frac{X^r}{(r)_{q^2}!}.
\end{equation}

\begin{lemma}
  The element $\sigma \in H^* \otimes H^*$, viewed as a bilinear form on $H$, is equal to the bilinear form mentioned in Lemma~\ref{lem:gr-uq-sl2-cocycle-deform}.
\end{lemma}
\begin{proof}
  For $i_1, j_1, k_1, i_2, j_2, k_2 \in \mathbb{I}$, we have
  \begin{align*}
    \sigma(x^{i_1} y^{j_1} g^{k_1}, x^{i_2} y^{j_2} g^{k_2})
    & = \sum_{r = 0}^{N-1} \frac{1}{(r)_{q^2}!}
      \langle \xi_1^r, x^{i_1} y^{j_1} g^{k_1} \rangle
      \langle \xi_2^r, x^{i_2} y^{j_2} g^{k_2} \rangle \\
    {}^{\eqref{eq:appendix-xi-1-power}, \eqref{eq:appendix-xi-2-power}}
    & = \sum_{r = 0}^{N-1} \frac{1}{(r)_{q^2}!}
      \delta_{i_1, r} \delta_{j_1, 0} (i_1)_{q^2}! \, q^{-2 i_1 k_1} \cdot
      \delta_{i_2, 0} \delta_{j_2, r} (j_2)_{q^2}! \\
    & = \delta_{i_1, j_2} \delta_{j_1, 0} \delta_{i_2, 0} (i_1)_{q^2}! \, q^{-2 i_1 k_1}.
      \qedhere
  \end{align*}
\end{proof}

Equations \eqref{eq:appendix-xi-1-xi-2-comult} and \eqref{eq:appendix-alpha-xi-1-xi-2} say that $\xi_1$ and $\xi_2$ are $d_1 * \chi_2$ and $d_2$, respectively, with the notation of \cite[Example 5.2]{2010arXiv1010.4976G}. Thus, by the general argument given in \cite{2010arXiv1010.4976G}, the bilinear form $\sigma$ is a Hopf 2-cocycle of $H$. Below, for the reader's convenience, we give a direct proof of this fact. As the first step, we note:

\begin{lemma}
  The following equations hold:
  \begin{equation}
    \label{eq:appendix-gr-uq-sl2-dual-relations}
    \alpha \xi_1 = q^2 \xi_1 \alpha, \quad
    \alpha \xi_2 = q^2 \xi_2 \alpha, \quad
    \xi_1 \xi_2 = \xi_2 \xi_1.
  \end{equation}
\end{lemma}
\begin{proof}
  We introduce a grading on $H^*$ by $(H^*)_n = (H_n)^*$ for $n \ge 0$.
  Given $\beta, \gamma \in \Grp(H^*)$, we define $P_1(\beta, \gamma)$ to be the space of homogeneous elements $\xi \in (H^*)_1$ satisfying $\Delta(\xi) = \xi \otimes \beta + \gamma \otimes \xi$. If $\xi \in P_1(\beta, \gamma)$, then we have
  \begin{equation*}
    \langle \xi, x g^i \rangle
    = \langle \xi, x \rangle \langle \beta, g^i \rangle
    + \langle \gamma, x \rangle \langle \xi, g^i \rangle
    = \langle \beta, g^i \rangle \langle \xi, x \rangle,
    \quad
    \langle \xi, y g^i \rangle
    = \langle \gamma, g^i \rangle \langle \xi, y \rangle
  \end{equation*}
  for all $i \in \mathbb{I}$. This means that $\xi \in P_1(\beta, \gamma)$ is completely determined by the scalars $\langle \xi, x \rangle$ and $\langle \xi, y \rangle$. Both $\alpha \xi_1$ and $\xi_1 \alpha$ belong to $P_1(\alpha^2, \alpha)$. Since $\langle \alpha \xi_1, x \rangle = q^2$, $\langle \xi_1 \alpha, x \rangle = 1$ and $\langle \alpha \xi_1, y \rangle = \langle \xi_1 \alpha, y \rangle = 0$, we have $\alpha \xi_1 = q^2 \xi_1 \alpha$.

  We have verified the first equation in \eqref{eq:appendix-gr-uq-sl2-dual-relations}. The second one is proved in a similar way. To prove the third one, we note that $\zeta := \xi_1 \xi_2 - \xi_2 \xi_1$ satisfies
  \begin{align*}
    \Delta(\zeta)
    & = \xi_1 \xi_2 \otimes \alpha + \xi_1 \alpha \otimes \alpha \xi_2 + \xi_2 \otimes \xi_1 + \alpha \otimes \xi_1 \xi_2 \\
    & \qquad - (\xi_2 \xi_1 \otimes \alpha + \xi_2 \otimes \xi_1 + \alpha \xi_1 \otimes \xi_2 \alpha + \alpha \otimes \xi_2 \xi_1) \\
    & = \zeta \otimes \alpha + \alpha \otimes \zeta,
  \end{align*}
  where the second equality follows from $\alpha \xi_i = q^2 \xi_i \alpha$ ($i = 1, 2$).
  Hence $\alpha^{-1} \zeta$ is a primitive element. Since a finite-dimensional Hopf algebra over a field of characteristic zero cannot have a non-zero primitive element, we have $\zeta = 0$, that is, the third equation of \eqref{eq:appendix-gr-uq-sl2-dual-relations} holds.
\end{proof}

\begin{lemma}
  \label{lem:appendix-sigma-inverse}
  $\sigma$ is invertible with respect to the convolution product.
\end{lemma}
\begin{proof}
  Since $\xi_1^N = \xi_2^N = 0$, the element $\nu := \varepsilon \otimes \varepsilon - \sigma$ satisfies $\nu^N = 0$. Thus the inverse of $\sigma$ is given by $\sigma^{-1} = (\varepsilon \otimes \varepsilon - \nu)^{-1} = \varepsilon \otimes \varepsilon + \nu + \nu^2 + \cdots + \nu^{N-1}$.
\end{proof}

\begin{lemma}
  The bilinear form $\sigma$ is a Hopf 2-cocycle of $H$.
\end{lemma}
\begin{proof}
  We have already proved that $\sigma$ is invertible with respect to the convolution product.
  It is easy to see that the equations $\sigma(1, h) = \varepsilon(h) = \sigma(h, 1)$ hold for all $h \in H$. We shall show that $\sigma$ satisfies \eqref{eq:Hopf-2-cocycle}. By the definition of the multiplication and the comultiplication of $H^*$, the equation \eqref{eq:Hopf-2-cocycle} is equivalent to
  \begin{equation*}
    (\varepsilon \otimes \sigma) \cdot (\id_{H^*} \otimes \Delta_{H^*})(\sigma)
    = (\sigma \otimes \varepsilon) \cdot (\Delta_{H^*} \otimes \id_{H^*})(\sigma).
  \end{equation*}
  This is verified as follows:
  $(\varepsilon \otimes \sigma) \cdot (\id_{H^*} \otimes \Delta_{H^*})(\sigma)$
  \begin{align*}
    & = \exp_{q^2}(\varepsilon \otimes \xi_1 \otimes \xi_2)
      \cdot \exp_{q^2}(\xi_1 \otimes \xi_2 \otimes \varepsilon
      + \xi_1 \otimes \alpha \otimes \xi_2) \\
    & = \exp_{q^2}(\varepsilon \otimes \xi_1 \otimes \xi_2)
      \cdot \exp_{q^2}(\xi_1 \otimes \xi_2 \otimes \varepsilon)
      \cdot \exp_{q^2}(\xi_1 \otimes \alpha \otimes \xi_2) \\
    & = \exp_{q^2}(\xi_1 \otimes \xi_2 \otimes \varepsilon)
      \cdot \exp_{q^2}(\varepsilon \otimes \xi_1 \otimes \xi_2)
      \cdot \exp_{q^2}(\xi_1 \otimes \alpha \otimes \xi_2) \\
    & = \exp_{q^2}(\xi_1 \otimes \xi_2 \otimes \varepsilon)
      \cdot \exp_{q^2}(\varepsilon \otimes \xi_1 \otimes \xi_2
      + \xi_1 \otimes \alpha \otimes \xi_2) \\
    & = (\sigma \otimes \varepsilon)
      \cdot (\Delta_{H^*} \otimes \id_{H^*})(\sigma).
  \end{align*}
  Here, the third equation holds since $\xi_1 \otimes \xi_2 \otimes \varepsilon$ and $\varepsilon \otimes \xi_1 \otimes \xi_2$ commute, and we have used the following fact in the second and fourth equalities: If $A$ and $B$ are elements of the same algebra satisfying $B A = q^2 A B$ and $A^i B^{N-i} = 0$ for all $i = 0, 1, \cdots, N$, then the equation $\exp_{q^2}(A + B) = \exp_{q^2}(A) \cdot \exp_{q^2}(B)$ holds \cite[Lemma 3.2]{MR2267580}. The proof is done.
\end{proof}

\begin{proof}[Proof of Lemma~\ref{lem:gr-uq-sl2-cocycle-deform}]
  Let $\sigma$ be the Hopf 2-cocycle of $H := \gr(u_q(\mathfrak{sl}_2))$ constructed above. By the proof of Lemma~\ref{lem:appendix-sigma-inverse}, the inverse of $\sigma$ with respect to the convolution product can be written as
  \begin{equation*}
    \sigma^{-1} = \varepsilon \otimes \varepsilon - \xi_1 \otimes \xi_2 + c_2 (\xi_1^2 \otimes \xi_2^2) + \cdots + c_{N-1} (\xi_1^{N-1} \otimes \xi_2^{N-1})
  \end{equation*}
  for some $c_2, \cdots, c_{N-1} \in \bfk$. Let $*$ denote the multiplication of $H^{\sigma}$. We consider the Hopf subalgebra $A = \langle g, x \rangle$ of $H$. Since $\xi_1 \otimes \xi_2$ vanishes on $A^{\otimes 2}$, $\sigma$ and $\sigma^{-1}$ are equal to $\varepsilon \otimes \varepsilon$ on $A^{\otimes 2}$. Thus, for $a, b \in A$, we have
  \begin{align*}
    a * b
    & = \sigma(a_{(1)}, b_{(1)}) a_{(2)} b_{(2)} \sigma^{-1}(a_{(3)}, b_{(3)}) \\
    & = \varepsilon(a_{(1)}) \varepsilon(b_{(1)})
      a_{(2)} b_{(2)} \varepsilon(a_{(3)})\varepsilon(b_{(3)})
      = a b.
  \end{align*}
  In particular, we have $g * x = g x$, $x * g = x g$, $x^{* m} = x^{m}$ and $g^{* m} = g^m$ for $m \in \mathbb{Z}_{\ge 0}$, where $(-)^{*m}$ denotes the $m$-th power with respect to $*$. Since $\xi_1 \otimes \xi_2$ also vanishes on the Hopf subalgebra $\langle g, y \rangle^{\otimes 2}$, we have $g * y = g y$, $y * g = y g$ and $y^{*m} = y^m$. Finally, a tedious but straightforward computation shows
  \begin{equation*}
    x * y = 1 + x y - g^{-2} \quad \text{and} \quad y * x = y x.
  \end{equation*}
  Hence, $x * y - q^2 y * x = 1 - g^{-2}$.
  The discussion so far shows that
  \begin{equation}
    \tag{\ref{eq:gr-uq-sl2-cocycle-deform-2}}
    u_q(\mathfrak{sl}_2) \to \gr(u_q(\mathfrak{sl}_2))^{\sigma},
    \quad \tilde{E} \mapsto x, \quad F \mapsto y, \quad K \mapsto g
  \end{equation}
  is a well-defined algebra map. It is easy to see that \eqref{eq:gr-uq-sl2-cocycle-deform-2} is a Hopf algebra map. Since $u_q(\mathfrak{sl}_2)$ has no non-trivial normal coideal subalgebras \cite[\S6.1]{2024arXiv241010064S}, the map \eqref{eq:gr-uq-sl2-cocycle-deform-2} is surjective. The proof is done by comparing the dimensions of the source and the target of \eqref{eq:gr-uq-sl2-cocycle-deform-2}.
\end{proof}

%% References
% \bibliographystyle{alpha}
% \bibliography{../common}

\newcommand{\etalchar}[1]{$^{#1}$}
\def\cprime{$'$}

\end{document}